\documentclass{amsart}
\usepackage{amssymb}
\usepackage{mathrsfs}
\usepackage[all]{xy}
\usepackage{tikz}
\usepackage{hyperref}
\usetikzlibrary{matrix,arrows}
\usepackage{graphicx}

\newcommand{\bk}{\Bbbk}
\newcommand{\C}{\mathbb{C}}

\newcommand{\N}{\mathbb{N}}
\newcommand{\Z}{\mathbb{Z}}
\newcommand{\Qlb}{\bar{\mathbb{Q}}_\ell}

\newcommand{\D}{\mathrm{D}}
\newcommand{\Db}{\D^{\mathrm{b}}}
\newcommand{\Kb}{\mathrm{K}^{\mathrm{b}}}

\newcommand{\mix}{{\mathrm{mix}}}
\newcommand{\Dmix}{\D^\mix}

\newcommand{\sA}{\mathscr{A}}

\newcommand{\pil}{\Pi^{\mathrm{L}}}
\newcommand{\pir}{\Pi^{\mathrm{R}}}

\newcommand{\scS}{\mathscr{S}}
\newcommand{\scT}{\mathscr{T}}

\newcommand{\Gv}{{G^\vee}}
\newcommand{\bBv}{\bar{B}^\vee}
\newcommand{\Bv}{B^\vee}
\newcommand{\Tv}{T^\vee}
\newcommand{\cBv}{{\mathscr{B}^\vee}}

\newcommand{\sph}{{\mathrm{sph}}}
\newcommand{\cW}{\mathcal{W}}
\newcommand{\ocW}{\overline{\mathcal{W}}}
\newcommand{\cR}{\mathcal{R}}

\newcommand{\Gk}{{G_{\mathbf{K}}}}
\newcommand{\Go}{{G_{\mathbf{O}}}}
\newcommand{\schub}{I\mathord{\cdot}}
\newcommand{\Fl}{{\mathcal{F}\ell}}
\newcommand{\Gr}{{\mathcal{G}r}}
\newcommand{\bX}{\mathbf{X}}
\newcommand{\bXp}{\mathbf{X}^+}
\newcommand{\aff}{{\mathrm{aff}}}

\newcommand{\ubk}{\underline{\bk}}

\newcommand{\Perv}{\mathrm{Perv}}
\newcommand{\Adv}{\mathrm{Adv}}
\newcommand{\p}{{}^p\!}
\newcommand{\ad}{{}^a\!}
\newcommand{\Tilt}{\mathrm{Tilt}}
\newcommand{\Parity}{\mathrm{Parity}}

\newcommand{\DmixI}{\Dmix_{(I)}}
\newcommand{\DmixGo}{\Dmix_{(\Go)}}

\newcommand{\Sat}{\mathcal{S}}

\newcommand{\sky}{\boldsymbol{1}_{\Gr}}
\newcommand{\IC}{\mathcal{IC}}
\newcommand{\cE}{\mathcal{E}}
\newcommand{\cIstd}{\mathcal{I}_!}
\newcommand{\cIcos}{\mathcal{I}_*}
\newcommand{\cJstd}{\mathcal{J}_!}
\newcommand{\cJcos}{\mathcal{J}_*}
\newcommand{\cT}{\mathcal{T}}
\newcommand{\uistd}{\underline{i}_!}
\newcommand{\uicos}{\underline{i}_*}
\newcommand{\pH}{{}^p\!\mathcal{H}}
\newcommand{\bc}{\mathbf{c}}
\newcommand{\bs}{\mathbf{s}}

\newcommand{\cN}{{\mathcal{N}}}
\newcommand{\tcN}{{\tilde{\mathcal{N}}}}
\newcommand{\cO}{\mathcal{O}}
\newcommand{\cON}{\mathcal{O}_\cN}
\newcommand{\fuv}{\mathfrak{u}^\vee}
\newcommand{\perf}{{\mathrm{perf}}}
\newcommand{\ft}{{\mathrm{ft}}}
\newcommand{\bGamma}{\boldsymbol{\Gamma}}
\newcommand{\lmod}{\textrm{-}\mathrm{mod}}
\newcommand{\lgmod}{\textrm{-}\mathrm{gmod}}
\newcommand{\scF}{\mathscr{F}}

\newcommand{\irr}{\mathrm{L}}
\newcommand{\til}{\mathrm{T}}
\newcommand{\wey}{\mathrm{M}}
\newcommand{\cow}{\mathrm{N}}

\newcommand{\Coh}{\mathrm{Coh}}
\newcommand{\Gm}{{\mathbb{G}_{\mathrm{m}}}}

\newcommand{\CohGm}{\Coh^{\Gv \times \Gm}}

\newcommand{\CohN}{\Coh(\cN)}
\newcommand{\Rep}{\mathrm{Rep}}
\newcommand{\SGD}{\mathbb{D}}

\newcommand{\PCohN}{\mathrm{P}\CohN}
\newcommand{\pcstd}{\Delta}
\newcommand{\pccos}{\nabla}
\newcommand{\pcpstd}{\bar{\pcstd}}
\newcommand{\pcpcos}{\bar{\pccos}}

\newcommand{\ExCoh}{\mathrm{ExCoh}}
\newcommand{\exstd}{\rotatebox[origin=c]{180}{$\mathsf{V}$}}
\newcommand{\excos}{\mathsf{V}}

\newcommand{\cH}{\mathcal{H}} 
\newcommand{\cF}{\mathcal{F}}
\newcommand{\cG}{\mathcal{G}}
\newcommand{\cK}{\mathcal{K}}

\newcommand{\simto}{\overset{\sim}{\to}}
\newcommand{\la}{\langle}
\newcommand{\ra}{\rangle}
\newcommand{\lla}{\langle\!\langle}
\newcommand{\rra}{\rangle\!\rangle}
\newcommand{\id}{\mathrm{id}}
\newcommand{\pt}{\mathrm{pt}}

\newcommand{\For}{\mathsf{For}}
\newcommand{\mult}{\mathbf{m}}
\newcommand{\spult}{\mathbf{p}}
\newcommand{\stult}{\mathbf{t}}

\DeclareMathOperator{\chr}{char}
\DeclareMathOperator{\End}{End}
\DeclareMathOperator{\Hom}{Hom}
\DeclareMathOperator{\Mor}{Mor}

\DeclareMathOperator{\RHom}{\mathit{R}Hom}
\DeclareMathOperator{\cRHom}{\mathit{R}\mathscr{H}\!\mathit{om}}
\DeclareMathOperator{\Ext}{Ext}
\DeclareMathOperator{\Tor}{Tor}
\DeclareMathOperator{\cok}{cok}
\newcommand{\uHom}{\underline{\mathrm{Hom}}}
\newcommand{\uRHom}{R\underline{\mathrm{Hom}}}
\newcommand{\sH}{\mathsf{H}}

\newcommand{\uEnd}{\underline{\mathrm{End}}}
\newcommand{\uExt}{\underline{\mathrm{Ext}}}
\newcommand{\gHom}{{\mathbb{H}\mathsf{om}}}
\newcommand{\gRHom}{\mathit{R}\gHom}
\newcommand{\gEnd}{{\mathbb{E}\mathsf{nd}}}

\newcommand{\tboxtimes}{\mathbin{\tilde{\mathord{\boxtimes}}}}
\DeclareMathOperator{\Spec}{Spec}
\DeclareMathOperator{\Proj}{Proj}
\DeclareMathOperator{\bSpec}{\mathbf{Spec}}
\DeclareMathOperator{\supp}{supp}

\numberwithin{equation}{section}
\newtheorem{thm}{Theorem}[section]
\newtheorem*{thm*}{Theorem}
\newtheorem{lem}[thm]{Lemma}
\newtheorem{prop}[thm]{Proposition}
\newtheorem{cor}[thm]{Corollary}
\theoremstyle{definition}

\theoremstyle{remark}
\newtheorem{rmk}[thm]{Remark}

\title[The affine Grassmannian and the Springer resolution]{The affine Grassmannian and the Springer resolution in positive characteristic}

\author{Pramod N. Achar}
\address{Department of Mathematics\\
  Louisiana State University\\
  Baton Rouge, LA 70803\\
  U.S.A.}
\email{pramod@math.lsu.edu}

\author{Laura Rider}
\address{Department of Mathematics\\
  Massachusetts Institute of Technology\\
  Cambridge, MA 02139\\
  U.S.A.}
\email{laurajoy@mit.edu}

\makeatletter
\let\@wraptoccontribs\wraptoccontribs
\makeatother
\contrib[with an appendix joint with]{Simon Riche}

\thanks{P.A. was supported by NSF Grant No.~DMS-1001594. L.R. was supported by an NSF Postdoctoral Research Fellowship.}

\begin{document}

\begin{abstract}
An important result of Arkhipov--Bezrukavnikov--Ginzburg relates constructible sheaves on the affine Grassmannian to coherent sheaves on the dual Springer resolution.  In this paper, we prove a positive-characteristic analogue of this statement, using the framework of ``mixed modular sheaves'' recently developed by the first author and Riche.  As an application, we deduce a relationship between parity sheaves on the affine Grassmannian and Bezrukavnikov's ``exotic t-structure'' on the Springer resolution.
\end{abstract}

\maketitle

\section{Introduction}
\label{sect:intro}

\subsection{Main result}
\label{ss:intro-main}
Let $G$ be a connected reductive complex algebraic group, and let $\Gv$ be the Langlands dual group over an algebraically closed field $\bk$. Recall that the \emph{geometric Satake equivalence} is an equivalence of tensor abelian categories
\begin{equation}\label{eqn:satake}
\Sat: \Rep(\Gv) \overset{\sim}{\longrightarrow} \Perv_\Go(\Gr,\bk),
\end{equation}
where $\Rep(\Gv)$ is the category of finite-dimensional rational representations of $\Gv$, and $\Perv_\Go(\Gr,\bk)$ is the category of spherical perverse $\bk$-sheaves on the affine Grassmannian $\Gr$.  When $\bk = \C$, there is an extensive body of work (see~\cite{ab,abg,bez:psaf,bf}, among others) exhibiting various ways of extending $\Sat$ to an equivalence of derived or triangulated categories.  In particular, an important theorem due to Arkhipov--Bezrukavnikov--Ginzburg~\cite{abg} relates Iwahori-constructible sheaves on $\Gr$ to coherent sheaves on the Springer resolution $\tcN$ for $\Gv$.

In this paper, we begin the project of studying derived versions of~\eqref{eqn:satake} in positive characteristic.  We work in the framework of ``mixed modular derived categories'' recently developed by the first author and S.~Riche~\cite{arc:f2,arc:f3}.  The main result of the paper is the following modular analogue of the result of~\cite{abg}.

\begin{thm}\label{thm:intro-main}
Assume that the characteristic of $\bk$ is a JMW prime for $\Gv$, and that $\Gv$ satisfies~\eqref{eqn:reasonable} below.  Then there is an equivalence of triangulated categories
\[
P: \DmixI(\Gr,\bk) \overset{\sim}{\longrightarrow} \Db\CohGm(\tcN)
\]
satisfying $P(\cF\la 1\ra) \cong P(\cF)\la -1\ra[1]$.  Moreover, this equivalence is compatible with the geometric Satake equivalence:  for any $\cF \in \DmixI(\Gr,\bk)$ and $V \in \Rep(\Gv)$, there is a natural isomorphism $P(\cF \star \Sat(V)) \cong P(\cF) \otimes V$.
\end{thm}

Recall that a \emph{JMW prime} for $\Gv$ is a good prime such that the main result of~\cite{jmw2} holds in that characteristic: that is, $\Sat$ sends tilting $\Gv$-modules to spherical parity sheaves.  (Recently, Mautner and Riche have shown that every good prime is a JMW prime; see~\S\ref{ss:mr} below).  When the characteristic of $\bk$ is a JMW prime, the Mirkovi\'c--Vilonen conjecture holds~\cite{arid}. The additional condition we impose on $\Gv$ is this:
\begin{equation}\label{eqn:reasonable}
\begin{minipage}{3.6in}
The derived group of $\Gv$ is simply connected, and its Lie algebra admits a nondegenerate $\Gv$-invariant bilinear form.
\end{minipage}
\end{equation}
Finally, $\DmixI(\Gr,\bk)$ is the mixed modular derived category of complexes that are constructible with respect to the stratification of $\Gr$ by orbits of an Iwahori subgroup $I \subset \Go$. (For full details on notation and terminology, see Section~\ref{sect:notation}.)

\subsection{Comparison with the work of Arkhipov--Bezrukavnikov--Ginzburg}

Readers who are familiar with~\cite{abg} will recognize a number of familiar ingredients in this paper, including
Wakimoto sheaves; the ind-perverse sheaf corresponding to the regular representation; and realizations of the coordinate rings of $\cN$ and $\tcN$ as $\Ext$-algebras on~$\Gr$. However, the behavior of these objects is often more complicated than in~\cite{abg}, both because of the nonsemisimplicity of the representation theory of $\Gv$, and because mixed modular sheaves are harder to work with than mixed $\Qlb$-sheaves.

One salient difficulty with mixed modular sheaves is that it is not known whether there is a well-behaved ``forgetful'' functor $\DmixI(\Gr,\bk) \to \Db_{(I)}(\Gr,\bk)$ (see~\cite[\S2.2]{arc:f2}), so we cannot compare mixed and ordinary perverse $\bk$-sheaves.  As a consequence, a key construction of~\cite{abg}, giving a dg-model for $\DmixI(\Gr,\Qlb)$ in terms of projective pro-perverse sheaves, cannot be carried out in positive characteristic.  Instead, we use a dg-model for $\DmixI(\Gr,\bk)$ based on parity sheaves. (Indeed, perverse sheaves are almost absent from this paper.)  The lack of a forgetful functor also means that unlike in~\cite{abg}, we do not know how to deduce a non-mixed analogue of Theorem~\ref{thm:intro-main}, describing the ordinary derived category $\Db_{(I)}(\Gr,\bk)$.

In~\cite{abg}, the result we have been discussing is used as a step in the proof that $\Perv_{(I)}(\Gr,\C)$ is equivalent to the principal block of the quantum group $U_q(\mathfrak{g}^\vee)$ at a root of unity.  We expect Theorem~\ref{thm:intro-main} (or its conjectural non-mixed analogue) to likewise play a role in the proof of the Finkelberg--Mirkovi\'c conjecture~\cite{fm}, which asserts that $\Perv_{(I)}(\Gr,\bk)$ is equivalent to the principal block of $\Rep(\Gv)$.

\subsection{Koszul-type duality and the exotic t-structure}

One of the main results of~\cite{arc:f2} gives an equivalence of categories between parity sheaves on a 
flag variety and mixed tilting sheaves on the Langlands dual flag variety.  Separately, according to~\cite[Proposition~5.7]{arid}, there is an equivalence of categories
\[
\Parity_{(\Go)}(\Gr,\bk) \simto \Tilt(\PCohN),
\]
where $\PCohN$ is the category of perverse-coherent sheaves on the nilpotent cone for $\Gv$ (see~\S\ref{ss:pcoh}).  These results raise the question of whether $\Parity_{(I)}(\Gr,\bk)$ participates in a ``parity--tilting'' equivalence.

When $\bk = \C$, this question has a positive answer~\cite{bez:ctm}. The other side of the equivalence involves the \emph{exotic t-structure} on $\Db\CohGm(\tcN)$, and the equivalence itself is understood as an instance of Koszul duality. (See~\cite[\S1.2]{bez:ctm} for the Koszul duality perspective, and~\cite{bez:ctm,bm} for applications of the exotic t-structure.) 

In this paper, we prove that this holds in positive characteristic as well.

\begin{thm}\label{thm:intro-exotic}
Under the assumptions of Theorem~\ref{thm:intro-main}, there is an equivalence of additive categories
$
P: \Parity_{(I)}(\Gr,\bk) \simto \Tilt(\ExCoh(\tcN))
$.
\end{thm}

This result ends up being quite an easy corollary of Theorem~\ref{thm:intro-main}, because the entire proof of Theorem~\ref{thm:intro-main} is structured in a way that anticipates this application.  As noted earlier, the perverse t-structure on $\DmixI(\Gr,\bk)$ does not play much of a role in this paper---but a different t-structure, the \emph{adverse t-structure}, appears quite prominently.  Ultimately, the adverse t-structure turns out to be the transport of the exotic t-structure across the equivalence of Theorem~\ref{thm:intro-main}. 

\subsection{Relationship to the work of Mautner--Riche}
\label{ss:mr}

While this work was underway, the authors learned that C.~Mautner and S.~Riche~\cite{mr} were independently pursuing a rather different approach to Theorem~\ref{thm:intro-exotic}, not relying on the geometric Satake equivalence or the Mirkovi\'c--Vilonen conjecture.  Their proof requires the characteristic of $\bk$ to be very good for $\Gv$, but a priori not necessarily a JMW prime.  In fact, their work implies that every good prime is a JMW prime, improving on the bounds established~\cite[Theorem~1.8]{jmw2}.  As a consequence, the main result of~\cite{arid} and Theorem~\ref{thm:intro-main} of the present paper both hold in good characteristic.

Nevertheless, we maintain the distinction between good primes and JMW primes in the body of this paper, so as to preserve its logical independence from~\cite{mr}. 

\subsection{Contents of the paper}

Section~\ref{sect:notation} introduces notation and recalls basic facts about the various varieties and categories we will work with.  In Section~\ref{sect:spherical}, we revisit the main results of~\cite{arid} and translate them to the mixed modular setting.  In Section~\ref{sect:regular}, we carry out some computations related to the regular representation of $\Gv$ and the corresponding ind-perverse sheaf.  Section~\ref{sect:wakimoto} develops the theory of mixed modular Wakimoto sheaves, which serve as constructible counterparts to line bundles on $\tcN$.  They are a key tool in Section~\ref{sect:extalg}, which realizes the coordinate ring of $\tcN$ as an $\Ext$-algebra on $\Gr$.  Theorem~\ref{thm:intro-main} is proved in Section~\ref{sect:main}. Finally, in Section~\ref{sect:exotic}, we discuss the exotic t-structure and prove Theorem~\ref{thm:intro-exotic}.

The language of mixed modular derived categories is ubiquitous in this paper.  For general background on these categories, see~\cite{arc:f2,arc:f3}.  Appendix~\ref{sect:mixed}, written jointly with S.~Riche, is a companion to those papers. It contains general results on mixed modular derived categories that were not included in~\cite{arc:f2,arc:f3}, and it can be read independently of the main body of the paper.

\subsection{Acknowledgments}

We are grateful to Carl Mautner and Simon Riche for discussing their work-in-progress with us.

\section{Notation and preliminaries}
\label{sect:notation}

\subsection{Graded vector spaces and graded Hom-groups}
\label{ss:grvec}

For a graded $\bk$-vector space $V = \bigoplus V_n$, or, more generally, a graded module over a graded $\bk$-algebra, we define the shift-of-grading functor $V \mapsto V\la m\ra$ by
\[
(V\la m\ra)_n = V_{m+n}.
\]
If $V$ and $W$ are two graded vector spaces, we define $\uHom(V,W)$ to be the graded vector space given by
\[
\uHom(V,W)_n = \Hom(V,W\la n\ra).
\]
More generally, if $\sA$ is any additive category equipped with an automorphism $\la 1\ra: \sA \to \sA$, we define $\uHom(A,B)$ for $A, B \in \sA$ as above.  We clearly have $\uHom(V\la n\ra, W\la m\ra) = \uHom(V,W)\la m-n\ra$.  Note that these conventions are consistent with those of~\cite{arid}, but opposite to those of~\cite{a}.

In the setting of mixed modular derived categories, it is often convenient to work with the automorphism $\{1\} = \la -1\ra[1]$.  As in~\S\ref{ss:ghom}, if $\cF$ and $\cG$ are two objects in a mixed modular derived category, we define a graded vector space $\gHom(\cF,\cG)$ by
\[
\gHom(\cF,\cG)_n = \Hom(\cF,\cG\{n\}).
\]
This satisfies $\gHom(\cF\{ n\}, \cG\{ m\}) = \gHom(\cF,\cG)\la m-n\ra$.

Finally, if $A$ and $B$ are objects in some triangulated category, we may write $\Hom^i(A,B)$ for $\Hom(A,B[i])$, and likewise for $\uHom^i({-},{-})$ and $\gHom^i({-},{-})$.

\subsection{Reductive groups and representations}

As in~\S\ref{ss:intro-main}, $G$ will always denote a fixed connected complex reductive group, and $\Gv$ will denote the Langlands dual group to $G$ over an algebraically closed field $\bk$.  In addition, the following assumptions will be in effect throughout the paper, except in Section~\ref{sect:wakimoto}:
\begin{itemize}
\item The characteristic of $\bk$ is a JMW prime for $G$.
\item The group $\Gv$ satisfies~\eqref{eqn:reasonable}.
\end{itemize}
The latter can be weakened slightly.  For instance, if $\Gv$ satisfies~\eqref{eqn:reasonable} and there is a separable central isogeny $\Gv \twoheadrightarrow H^\vee$, then the main results hold for $H^\vee$ as well.  However, to simplify the exposition, we assume~\eqref{eqn:reasonable} throughout.

Fix a Borel subgroup $B \subset G$ and a maximal torus $T \subset B$, along with corresponding subgroups $\Tv \subset \bBv \subset \Gv$.  Let $\Bv \subset \Gv$ be the opposite Borel subgroup to $\bBv$.  We regard $B$ as a ``positive'' Borel subgroup and $\Bv$ as a ``negative'' one.  That is, we call a character of $\Tv$ \emph{dominant} if its pairing with any root of $B$ is nonnegative, or equivalently, if its pairing with any coroot of $\Bv$ is nonpositive.  Let $\bX$ denote the character lattice of $\Tv$, identified with the cocharacter lattice of $T$, and let $\bXp \subset \bX$ be the set of dominant weights.  The set $\bX$ carries two natural partial orders, which we denote as follows:
\begin{align*}
\lambda &\preceq \mu &&\text{if $\mu - \lambda$ is a sum of positive roots;} \\
\lambda &\le \mu &&\text{if $\schub \lambda \subset \overline{\schub \mu}$ (see~\S\ref{ss:affine-grassmannian} below).}
\end{align*}
These two orders coincide on $\bXp$.

For $\lambda \in \bXp$, let $\irr(\lambda)$, $\wey(\lambda)$, $\cow(\lambda)$, and $\til(\lambda)$ denote the irreducible, Weyl, dual Weyl, and indecomposable tilting $\Gv$-modules, respectively, of highest weight $\lambda$.

Let $W$ denote the Weyl group of $G$ or $\Gv$, and let $w_0$ denote the longest element of $W$.  For any $\lambda \in \bX$, we put
\[
\delta_\lambda = \text{length of the shortest $w \in W$ such that $w\lambda$ is dominant.}
\]
This is consistent with~\cite[\S1.4.1]{bez:ctm}.  The notation ``$\delta_\lambda$'' also appears in~\cite{a,arid,minnthuaye} with a slightly different meaning: in those papers, only dominant weights occur, and the integer they call ``$\delta_\lambda$'' is called $\delta_{w_0\lambda}$ in the present paper.

\subsection{The affine Grassmannian}
\label{ss:affine-grassmannian}

Let $\Gr = \Gk/\Go$, where $\mathbf{K} = \C((t))$ is the field of Laurent series in an indeterminate $t$, and $\mathbf{O} = \C[[t]]$ is its subring of power series. Let $I \subset \Go$ be the Iwahori subgroup corresponding to $B \subset G$.  Recall that the $I$-orbits on $\Gr$ are naturally parametrized by $\bX$.  For $\lambda \in \bX$, the corresponding $I$-orbit is denoted simply by $\schub \lambda$, and the inclusion map by
\[
i_\lambda: \schub \lambda \hookrightarrow \Gr.
\]
The $\Go$-orbits are parametrized instead by $\bXp$.  Recall that these are sometimes called \emph{spherical} orbits, and that sheaves on $\Gr$ smooth along the $\Go$-orbits are sometimes called \emph{spherical sheaves}. For $\lambda \in \bXp$, the corresponding $\Go$-orbit is denoted by $\Gr_\lambda$, and the inclusion map by
\[
i^\sph_\lambda: \Gr_\lambda \hookrightarrow \Gr.
\]

\subsection{Constructible sheaves}

All constructible sheaves will be assumed to have coefficients in $\bk$.  From now on, we will omit the coefficients from the notation for categories of constructible complexes.

Let $\Perv_\Go(\Gr)$ be the category of $\Go$-equivariant perverse $\bk$-sheaves on $\Gr$. For $\lambda \in \bXp$, the objects in $\Perv_\Go(\Gr)$ arising from various $\Gv$-representations of highest weight $\lambda$ via the geometric Satake equivalence~\eqref{eqn:satake} are denoted as follows:
\[
\IC(\lambda) = \Sat(\irr(\lambda)),
\quad
\cIstd(\lambda) = \Sat(\wey(\lambda)),
\quad
\cIcos(\lambda) = \Sat(\cow(\lambda)),
\quad
\cT(\lambda) = \Sat(\til(\lambda)).
\]

Let $\Parity_{(I)}(\Gr)$ denote the additive category of parity complexes on $\Gr$ that are constructible with respect to the stratification by $I$-orbits, and let $\DmixI(\Gr)$ denote the corresponding mixed derived category.  More generally, if $X \subset \Gr$ is any locally closed $I$-stable subset, then $\DmixI(X)$ and related notations are defined similarly.  If $X$ is smooth, we denote by $\ubk_X$, or simply $\ubk$, the constant sheaf on $X$ with value $\ubk$, regarded as an object of $\Parity_{(I)}(X)$ or $\DmixI(X)$.

Let $\Perv^\mix_{(I)}(\Gr) \subset \DmixI(\Gr)$ denote the abelian category of mixed perverse sheaves.  This is a graded quasihereditary category.  Given $\lambda \in \bX$, the corresponding standard and costandard objects will be denoted by
\[
\uistd(\lambda) = i_{\lambda!}\ubk_{\schub \lambda}\{\dim \schub\lambda\}
\qquad\text{and}\qquad
\uicos(\lambda) = i_{\lambda*}\ubk_{\schub \lambda}\{\dim \schub\lambda\},
\]
respectively. The image of the canonical morphism $\uistd(\lambda) \to \uicos(\lambda)$ is denoted $\IC(\lambda)$. (Lemma~\ref{lem:sph-embed} below will resolve the apparent conflict with the notation for $\Sat(\irr(\lambda))$.) Lastly, let $\cE(\lambda)$ denote the unique indecomposable parity sheaf supported on $\overline{\schub \lambda}$ and whose restriction to $\schub \lambda$ is $\ubk\{\dim \schub \lambda\}$.  When $\lambda \in \bXp$, \cite{jmw2} tells us that $\cE(\lambda) = \cT(\lambda)$.

We will also work with the spherical categories $\Parity_{(\Go)}(\Gr)$, $\DmixGo(\Gr)$, and $\Perv^\mix_{(\Go)}(\Gr)$, and occasionally with the equivariant versions $\Dmix_I(\Gr)$, $\Dmix_\Go(\Gr)$, etc.  The spherical case is not explicitly covered by the papers~\cite{arc:f2,arc:f3}, which required the variety to be stratified by affine spaces.  See~\S\ref{ss:non-affine} for a discussion of this case.  For $\lambda \in \bXp$, we put
\[
\cJstd(\lambda) = (i^\sph_\lambda)_!\ubk_{\Gr_\lambda}\{\dim \Gr_\lambda\}
\qquad\text{and}\qquad
\cJcos(\lambda) = (i^\sph_\lambda)_*\ubk_{\Gr_\lambda}\{\dim \Gr_\lambda\}.
\]

The following lemma lets us identify $\Perv_\Go(\Gr)$ with a full subcategory of $\Perv^\mix_\Go(\Gr)$.  Via this identification, we will henceforth regard $\Sat$ as taking values in $\Perv^\mix_\Go(\Gr)$.  In particular, the objects $\cIstd(\lambda)$, $\cT(\lambda)$, etc., defined above will henceforth be regarded as objects of $\Perv^\mix_\Go(\Gr)$.

\begin{lem}\label{lem:sph-embed}
There is a t-exact fully faithful functor $\Db\Perv_\Go(\Gr) \to \Dmix_\Go(\Gr)$ which, for each $\lambda \in \bXp$, sends $\IC(\lambda) \in \Perv_\Go(\Gr)$ to $\IC(\lambda) \in \Perv^\mix_\Go(\Gr)$, and sends $\cT(\lambda)$ to $\cE(\lambda)$.
\end{lem}
Note that the domain of this functor is \emph{not} $\Db_\Go(\Gr)$; rather, it is the derived category of the heart.  It is equivalent $\Db\Rep(\Gv)$.
\begin{proof}
Note that $\Db\Perv_\Go(\Gr) \cong \Kb\Tilt(\Perv_\Go(\Gr))$, as usual for a quasihereditary category.  Since $\chr \bk$ is a JMW prime for $G$, we have
\[
\Tilt(\Perv_\Go(\Gr)) = \Parity_\Go(\Gr) \cap \Perv_\Go(\Gr).
\]
The desired functor is induced by the fully faithful embedding $\Tilt(\Perv_\Go(\Gr)) \hookrightarrow \Parity_\Go(\Gr)$.
\end{proof}

Via Lemma~\ref{lem:sph-embed}, we will henceforth identify $\Perv_\Go(\Gr)$ with a full subcategory of $\Dmix_\Go(\Gr)$.  In particular, for any $\cF \in \DmixI(\Gr)$ and any $V \in \Rep(\Gv)$, it makes sense to form the convolution product
\[
\cF \star \Sat(V).
\]

\subsection{The Springer resolution and the nilpotent cone}

Let $\cBv = \Gv/\Bv$ be the flag variety for $\Gv$.  Let $\fuv$ be the Lie algebra of the unipotent radical of $\Bv$, and let $\tcN = \Gv \times^{\Bv} \fuv$ be the Springer resolution.  Finally, let $\cN$ be the nilpotent cone in the Lie algebra of $\Gv$; and let $\pi: \tcN \to \cN$ be the obvious map.

We equip $\cN$ with an action of the multiplicative group $\Gm$ by setting $z \cdot x = z^{-2}x$, where $z \in \Gm$ and $x \in \cN$. We likewise make $\Gm$ act on $\tcN$ by having $z \in \Gm$ scale the fibers of $\tcN \to \Gv/\Bv$ by $z^{-2}$. In both cases, this $\Gm$-action commutes with the natural $\Gv$-action. Moreover, the map $\pi$ is $(\Gv \times \Gm)$-equivariant. The induced action of $\Gm$ on the coordinate ring $\bk[\cN]$ has even nonnegative weights. In other words, $\bk[\cN]$ becomes a graded ring concentrated in even nonnegative degrees. 

In this paper, coherent sheaves on $\cN$ or $\tcN$ will always be $(\Gv \times \Gm)$-equivariant.  For brevity, we write $\CohN$ instead of $\CohGm(\cN)$ for the category of $(\Gv \times \Gm)$-equivariant coherent sheaves on $\cN$, and likewise for $\Coh(\tcN)$.  The notation $\pi_*$ should always be understood as a derived functor $\Db\Coh(\tcN) \to \Db\CohN$.

Let $\cON$ and $\cO_\tcN$ denote the structure sheaves of $\cN$ and $\tcN$, respectively.  Given $m \in \Z$, let $\cON\la m\ra$ denote the coherent sheaf that corresponds to the graded $\bk[\cN]$-module $\bk[\cN]\la m\ra$, where the latter is defined as in~\S\ref{ss:grvec}.  We also put $\cO_\tcN\la m\ra = \pi^*\cON\la m\ra$.  More generally, for any $\cF \in \Db\CohN$, we let $\cF\la m\ra = \cF \otimes \cON\la m\ra$, and likewise in $\Db\Coh(\tcN)$.

Any weight $\lambda \in \bX$ determines a line bundle $\cO_\tcN(\lambda)$ on $\tcN$.  The push-forwards $\pi_*\cO_\tcN(\lambda)$ will be discussed in~\S\ref{ss:pcoh} below.
In the special case where $\lambda = 0$, it is known (see~\cite[Theorem~5.3.2]{bk}) that
\begin{equation}\label{eqn:rational-resoln}
\pi_*\cO_\tcN \cong \cON.
\end{equation}
Separately, by~\cite[Lemmas~3.4.2 and~5.1.1]{bk}, one has
\begin{equation}\label{eqn:springer-dualizing}
\pi^! \cON \cong \cO_\tcN.
\end{equation}

It will sometimes be more convenient to work in the language of ``$\Gv$-equivariant graded finitely generated $\bk[\cN]$-modules'' rather than in that of ``$(\Gv \times \Gm)$-equi\-var\-i\-ant coherent sheaves on $\cN$,'' and we will pass freely between the two.  We identify the space of global sections $\Gamma(\tcN,\cO_\tcN)$ with the ring $\bk[\cN]$ via~\eqref{eqn:rational-resoln}, and given $\cF \in \Coh(\tcN)$, we think of $\Gamma(\tcN,\cF)$ as a $\Gv$-equivariant graded finitely generated $\bk[\cN]$-module.  For instance, the cohomology-vanishing result of~\cite[Theorem~2]{klt} says that for $\lambda \in \bXp$, $\pi_*(\cO_\tcN(\lambda))$ is a coherent sheaf, so
\begin{equation}\label{eqn:dom-line-sec}
\pi_*\cO_\tcN(\lambda) = \Gamma(\tcN, \cO_\tcN(\lambda))
\qquad\text{for $\lambda \in \bXp$.}
\end{equation}

\subsection{Perverse-coherent sheaves}
\label{ss:pcoh}

The category $\Db\CohN$ admits a t-structure whose heart is known as the category of \emph{perverse-coherent sheaves}, and is denoted by $\PCohN$.  For general background on this category, see~\cite{bez:qes,a}.  Some key features of this category are as follows:
\begin{itemize}
\item It is stable under $\cF \mapsto \cF\la 1\ra$.
\item Every object has finite length.  Up to grading shift, the isomorphism classes of simple objects are in bijection with $\bXp$.
\item It is a \emph{properly stratified category}.
\end{itemize}
For background on properly stratified categories, see~\cite[\S2]{arid}. In a properly stratified category---a notion that generalizes that of a quasihereditary category---there are four important classes of indecomposable objects, called \emph{standard}, \emph{proper standard}, \emph{costandard}, and \emph{proper costandard} objects.  In $\PCohN$, we denote these objects by
\[
\pcstd(\lambda), 
\qquad
\pcpstd(\lambda),
\qquad
\pccos(\lambda),
\qquad
\pcpcos(\lambda),
\]
respectively, where $\lambda \in \bXp$.  The proper ones are given by
\[
\pcpstd(\lambda) = \pi_*\cO_\tcN(-w_0\lambda)\la\delta_{w_0\lambda}\ra,
\qquad
\pcpcos(\lambda) = \pi_*\cO_\tcN(\lambda)\la -\delta_{w_0\lambda}\ra.
\]
Revisiting~\eqref{eqn:dom-line-sec}, we find that the proper costandard objects satisfy
\begin{equation}\label{eqn:aj}
\pcpcos(\lambda) \in \CohN
\qquad\text{for all $\lambda \in \bXp$.}
\end{equation}
More generally, any object of $\PCohN$ with a proper costandard filtration is actually a coherent sheaf.  (Proper standard objects, in constrast, are generally not coherent sheaves.)  For descriptions of $\pcstd(\lambda)$ and $\pccos(\lambda)$, see~\cite[Definition~4.2]{minnthuaye}.

Lastly, let $\SGD = \cRHom({-},\cO_\cN)$ be the Serre--Grothendieck duality functor on $\Db\Coh(\cN)$.  The category $\PCohN$ is stable under $\SGD$, and we have
\[
\SGD(\pcpcos(\lambda)) \cong \pcpstd(-w_0\lambda)
\qquad\text{and}\qquad
\SGD(\pccos(\lambda)) \cong \pcstd(-w_0\lambda).
\]

\section{The Mirkovi\'c--Vilonen conjecture for mixed sheaves}
\label{sect:spherical}

In this section, we recast the main results of~\cite{arid} in the setting of mixed modular derived categories, obtaining a mixed version of the Mirkovi\'c--Vilonen conjecture.  The main idea is to compare spherical parity sheaves on $\Gr$ with perverse-coherent sheaves on $\cN$.  Along the way, we carry out various auxiliary computations in $\PCohN$ that will be useful in the sequel.

\subsection{Derived equivalences for spherical sheaves}

Let $\Gamma \subset \bXp$ be a finite order ideal, i.e., a finite subset such that if $\gamma \in \Gamma$ and $\mu < \gamma$, then $\mu \in \Gamma$.  Let $\Gr_\Gamma = \bigcup_{\gamma \in \Gamma} \Gr_\gamma$ be the corresponding closed subset of $\Gr$, and let
\[
U_\Gamma = \Gr \smallsetminus \Gr_\Gamma.
\]
This is an open $\Go$-stable subset of $\Gr$.  Let $j_\Gamma: U_\Gamma \hookrightarrow \Gr$ be the inclusion map.  

Recall that $\PCohN$ is equipped with a recollement structure (see~\cite[Proposition~2.2]{arid}). Let $\PCohN_\Gamma \subset \PCohN$ denote the Serre subcategory generated by $\pcpcos(\gamma)\la m\ra$ with $\gamma \in \Gamma$, and let $\Pi_\Gamma: \PCohN \to \PCohN/\PCohN_\Gamma$ be the Serre quotient functor.  We will denote its derived version by the same symbol:
\[
\Pi_\Gamma: \Db\CohN \to \Db(\PCohN/\PCohN_\Gamma).
\]
Here, we are using the main result of~\cite{a} to identify
\begin{equation}\label{eqn:pcoh-dereq}
\Db\PCohN \cong \Db\CohN.
\end{equation}
Next, let
\[
\Db_\ft(\PCohN/\PCohN_\Gamma) \subset \Db(\PCohN/\PCohN_\Gamma)
\]
be the full triangulated subcategory generated by tilting objects.  (The subscript ``$\ft$'' refers to the fact that this category consists of ``finite tilting complexes.'')  Note that the natural functor
\begin{equation}\label{eqn:pcoh-tilt-dereq}
\Kb\Tilt(\PCohN/\PCohN_\Gamma) \simto \Db_\ft(\PCohN/\PCohN_\Gamma)
\end{equation}
is an equivalence of categories: both sides are generated by tilting objects, so it suffices to compare $\Hom^i(\cF,\cG)$ on each sides for $\cF, \cG \in \Tilt(\PCohN/\PCohN_\Gamma)$.  When $i = 0$, these groups agree, and when $i \ne 0$, $\Hom^i(\cF,\cG)$ vanishes on both sides.  (See~\cite[Proposition~1.5]{bbm} or~\cite[Theorem~3.17]{minnthuaye}.)

In the special case where $\Gamma = \varnothing$, the equivalence~\eqref{eqn:pcoh-dereq} restricts to an equivalence
\[
\Db_\ft\PCohN \cong \Db_\perf\CohN,
\]
where the right-hand side is the category of \emph{perfect complexes} on $\cN$, i.e., those with a finite resolution whose terms are direct sums of objects of the form $\cO_\cN \otimes V\la n\ra$ with $V \in \Rep(\Gv)$.

\begin{prop}\label{prop:sph-main}
There is an equivalence of triangulated categories
\[
P_\sph: \DmixGo(\Gr) \simto \Db_\perf\CohN
\]
satisfying $P_\sph(\cF\{ 1\}) \cong P_\sph(\cF)\la 1\ra$. Moreover, this equivalence is compatible with the geometric Satake equivalence: for any $\cF \in \DmixGo(\Gr)$ and $V \in \Rep(\Gv)$, there is a natural isomorphism
$
P_\sph(\cF \star \Sat(V)) \cong P_\sph(\cF) \otimes V
$.
\end{prop}
\begin{proof}
The existence of the equivalence is just a restatement of~\cite[Proposition~5.7]{arid}.  That result also gives us compatibility with geometric Satake when $V$ is a tilting $\Gv$-module.  One can then extend that to, say, any $V$ with a Weyl filtration, by induction on the ``tilting dimension'' (see~\cite[Definition~2.10]{arid}) of $V$.  Finally, every $\Gv$-module admits a finite resolution by modules with a Weyl filtration.  By induction on the length of such a resolution, one obtains the full result.
\end{proof}

\begin{prop}
Let $\Gamma \subset \bXp$ be a finite order ideal.  There is an equivalence of triangulated categories
\[
P_{\sph,\Gamma}: \DmixGo(U_\Gamma) \simto \Db_\ft(\PCohN/\PCohN_\Gamma)
\]
such that the following diagram commutes up to isomorphism:
\[
\xymatrix{
\DmixGo(\Gr) \ar[d]_{j_\Gamma^*} \ar[r]^-{P_\sph}_-{\sim} &
  \Db_\perf\CohN \ar[d]^{\Pi_\Gamma} \\
\DmixGo(U_\Gamma) \ar[r]^-{P_{\sph,\Gamma}}_-{\sim} &
  \Db_\ft(\PCohN/\PCohN_\Gamma)}
\]
\end{prop}
\begin{proof}
This is an immediate consequence of~\cite[Corollary~5.8]{arid}, using the equivalence~\eqref{eqn:pcoh-tilt-dereq}.
\end{proof}

The functor $j_\Gamma^*$ has left and right adjoints $j_{\Gamma!}, j_{\Gamma*}: \DmixGo(U_\Gamma) \to \DmixGo(\Gr)$.  On the other hand, $\Pi_\Gamma$ has left and right adjoints $\pil_\Gamma, \pir_\Gamma$ that are a priori defined as functors $\Db(\PCohN/\PCohN_\Gamma) \to \Db\CohN$, but according to~\cite[Proposition~5.4]{minnthuaye}, they actually take objects in $\Db_\ft(\PCohN/\PCohN_\Gamma)$ to $\Db_\perf\CohN$.  From these observations, we obtain the following consequence of the preceding proposition.

\begin{cor}
Let $\Gamma \subset \bXp$ be a finite order ideal.  The following diagrams commute up to isomorphism:
{\tiny\[
\xymatrix{
\DmixGo(\Gr) \ar[r]^-{P_\sph}_-{\sim} &
  \Db_\perf\CohN \\
\DmixGo(U_\Gamma) \ar[r]^-{P_{\sph,\Gamma}}_-{\sim} \ar[u]_{j_{\Gamma!}}&
  \Db_\ft(\PCohN/\PCohN_\Gamma) \ar[u]_{\pil_\Gamma}}
\quad
\xymatrix{
\DmixGo(\Gr) \ar[r]^-{P_\sph}_-{\sim} &
  \Db_\perf\CohN \\
\DmixGo(U_\Gamma) \ar[r]^-{P_{\sph,\Gamma}}_-{\sim} \ar[u]_{j_{\Gamma*}}&
  \Db_\ft(\PCohN/\PCohN_\Gamma) \ar[u]_{\pir_\Gamma}}
\]}%
\end{cor}

\begin{cor}\label{cor:psph-std}
We have
\[
P_\sph(\cJcos(\lambda)) \cong \pccos(\lambda)\la -\delta_{w_0\lambda}\ra
\qquad\text{and}\qquad
P_\sph(\cJstd(\lambda)) \cong \pcstd(\lambda)\la \delta_{w_0\lambda}\ra.
\]
\end{cor}

\begin{proof}
Let $\Gamma = \{\mu \in \bXp \mid \mu < \lambda\}.$ Note that $P_\sph(\cIstd(\lambda)) \cong \cON \otimes \wey(\lambda)$. The corollary follows from the observations that $\pcstd(\lambda) = \pil_\Gamma \Pi_\Gamma(\cON \otimes \wey(\lambda))\la -\delta_{w_0\lambda}\ra$ \cite[Definition 4.2]{minnthuaye} and $\cJstd(\lambda) \cong j_{\Gamma !}j_\Gamma^* \cIstd(\lambda).$
\end{proof}

\subsection{Further study of perverse-coherent sheaves}

In this subsection, we collect a number of results about $\uHom$-groups, quotients, and subobjects in $\PCohN$.

\begin{lem}\label{lem:uend-nabla}
Let $\lambda \in \bXp$. There are isomorphisms of graded rings $\uEnd(\pccos(\lambda)) \cong \uEnd(\pcstd(\lambda)) \cong \sH^\bullet(\Gr_\lambda)$.
\end{lem}
\begin{proof}
This is a consequence of~\cite[Theorem~5.9]{arid}.  Specifically, let $\Gamma = \{ \mu \in \bXp \mid \mu < \lambda \}$.  Consider the tilting module $\til(\lambda)$, which corresponds under the geometric Satake equivalence to the parity sheaf $\cE(\lambda)$.  Note that $\cE(\lambda)|_{U_\Gamma}$ is just the shifted constant sheaf $\ubk\{\dim \Gr_\lambda\}$ on $\Gr_\lambda$.  Thus,~\cite[Theorem~5.9]{arid} gives us the first isomorphism below:
\[
\sH^\bullet(\Gr_\lambda) \cong \uEnd(\Pi_\Gamma(\cON \otimes \til(\lambda))) \cong \uEnd(\pir_\Gamma \Pi_\Gamma(\cON \otimes \til(\lambda))).
\]
The second isomorphism holds because $\pir_\Gamma$ is fully faithful. Finally, from \cite[Definition~4.2]{minnthuaye}, we see that $\pir_\Gamma\Pi_\Gamma(\cON \otimes \til(\lambda)) \cong \pccos(\lambda)\la -\delta_{w_0\lambda}\ra$.
\end{proof}

The preceding lemma lets us regard the coherent sheaf $\pccos(\lambda)$ as a graded $\sH^\bullet(\Gr_\lambda)$-module.  We can of course also regard $\bk$ (thought of as a graded vector space concentrated in degree $0$) as a $\sH^\bullet(\Gr_\lambda)$-module in the obvious way.

\begin{prop}\label{prop:aj-compute}
There is an isomorphism of $\Gv$-equivariant graded $\bk[\cN]$-modules
\[
\bk \otimes_{\sH^\bullet(\Gr_\lambda)} \pccos(\lambda) \cong \pcpcos(\lambda)\la 2\delta_{w_0\lambda}\ra.
\]
\end{prop}
\begin{proof}
Let $\uEnd(\pccos(\lambda))^+ \subset \uEnd(\pccos(\lambda))$ denote the subspace spanned by homogeneous elements of strictly positive degree.  Let $\{f_1, \ldots, f_n\}$ be a basis of homogeneous elements for $\uEnd(\pccos(\lambda))^+$, and let $d_i$ denote the degree of $f_i$. In other words, we may regard each $f_i$ as a map $\pccos(\lambda)\la -d_i\ra \to \pccos(\lambda)$. Form their sum
\[
\bigoplus_{i = 1}^n \pccos(\lambda)\la -d_i\ra \xrightarrow{f = \sum f_i} \pccos(\lambda).
\]
This is a morphism in both $\CohN$ and $\PCohN$.  We will study its kernel and cokernel in both categories.  First, via the isomorphism of Lemma~\ref{lem:uend-nabla}, we have
\begin{equation}\label{eqn:aj-coh-cok}
\bk \otimes_{\sH^\bullet(\Gr_\lambda)} \pccos(\lambda) \cong \uEnd(\pccos(\lambda))/\uEnd(\pccos(\lambda))^+ \otimes_{\uEnd(\pccos(\lambda))} \pccos(\lambda) \cong \cok_{\CohN} f.
\end{equation}

We now turn our attention to $\PCohN$.
Let $\Gamma = \{ \mu \in \bXp \mid \mu \le \lambda\}$, and let $\Upsilon = \Gamma \smallsetminus \{\lambda\}$. Consider the quotient functor
\[
\Pi_{\Gamma,\Upsilon}: \PCohN_\Gamma \to \PCohN_\Gamma/\PCohN_\Upsilon,
\]
and let $\pir_{\Gamma,\Upsilon}$ be its right adjoint. Then $\PCohN_\Gamma/\PCohN_\Upsilon$ is a properly stratified category with a unique simple object up to Tate twist: namely, the object $S = \Pi_{\Gamma,\Upsilon}(\pcpcos(\lambda))$. This object has an injective envelope $I = \Pi_{\Gamma,\Upsilon}(\pccos(\lambda))$.  We have $\pcpcos(\lambda) \cong \pir_{\Gamma,\Upsilon}(S)$ and $\pccos(\lambda) \cong \pir_{\Gamma,\Upsilon}(I)$. Moreover, as in Lemma~\ref{lem:uend-nabla}, we have $\uEnd(I) \cong \sH^\bullet(\Gr_\lambda)$.  On the other hand, by~\cite[Lemma~2.7(1) and Theorem~2.15]{arid}, the object $I$ is also isomorphic to $\Pi_{\Gamma,\Upsilon}(\pcstd(\lambda)\la 2\delta_{w_0\lambda}\ra)$.  Thus, $I$ is the projective cover of $S\la2\delta_{w_0\lambda}\ra$.

Let $\tilde f_i: I\la -d_i\ra \to I$ be the map corresponding to $f_i$ under the isomorphism $\pir_{\Gamma,\Upsilon}: \uEnd(I) \simto \uEnd(\pccos(\lambda))$, and define $\tilde f$ in the same way as $f$ above.  Then the image of $\tilde f$ is the radical of the indecomposable projective object $I$, and so $\cok \tilde f \cong S\la 2\delta_{w_0\lambda}\ra$.  Also, trivially, $\ker \tilde f$ has a filtration whose subquotients are various $S\la k\ra$.  Applying $\pir_{\Gamma,\Upsilon}$, we obtain an exact sequence in $\PCohN$
\begin{equation}\label{eqn:aj-pcoh-ses}
0 \to \ker_{\PCohN} f \to \bigoplus_{i = 1}^n \pccos(\lambda)\la -d_i\ra \overset{f}{\to} \pccos(\lambda) \to \pcpcos(\lambda)\la 2\delta_{w_0\lambda}\ra \to 0,
\end{equation}
where $\ker_{\PCohN} f$ has a filtration whose subquotients are various $\pcpcos(\lambda)\la k\ra$.

Let $K$ be the cone of $f$ in $\Db\CohN$.  Then, considering both the natural and perverse-coherent t-structures on this category, we have two distinguished triangles
\begin{gather*}
(\ker_{\CohN} f)[1] \to K \to \cok_{\CohN} f \to,\\
(\ker_{\PCohN} f)[1] \to K \to \cok_{\PCohN} f \to.
\end{gather*}
But we saw in~\eqref{eqn:aj-pcoh-ses} that both $\ker_{\PCohN} f$ and $\cok_{\PCohN} f$ have proper costandard filtrations, and hence happen to lie in $\CohN$.  So by~\cite[Proposition~1.3.3(ii)]{bbd}, the two distinguished triangles above must be canonically isomorphic.  In particular, we have $\cok_{\CohN} f \cong \cok_{\PCohN} f$.  The result then follows by comparing~\eqref{eqn:aj-coh-cok} and~\eqref{eqn:aj-pcoh-ses}.
\end{proof}

The next lemma is a related fact involving standard objects rather than costandard ones.


\begin{lem}\label{lem:uend-delta-k}
There is an isomorphism $\uEnd(\pcstd(\lambda))$-modules
\[
\uHom(\pcpstd(\lambda)\la -2\delta_{w_0\lambda}\ra,\pcstd(\lambda)) \simto \bk.
\]
\end{lem}
\begin{proof}
Let $S, I \in \PCohN_\Gamma/\PCohN_\Upsilon$ be as in the preceding proof, and let $\pil_{\Gamma,\Upsilon}$ be the left adjoint to $\Pi_{\Gamma,\Upsilon}$.  Since $I$ is the injective envelope of $S$, we certainly have $\uHom(S\la -2\delta_{w_0\lambda}\ra, I\la -2\delta_{w_0\lambda}\ra) \cong \bk$.  Applying the fully faithful functor $\pil_{\Gamma,\Upsilon}$ yields the result.
\end{proof}

\begin{lem}\label{lem:pcstd-compare}
Let $M \in \PCohN$ be an object with a costandard filtration.  Then $\uHom(\pcstd(\lambda), M)$ is a free $\uEnd(\pcstd(\lambda))$-module.  Moreover, there is a natural isomorphism
\[
\bk \otimes_{\uEnd(\pcstd(\lambda))} \uHom(\pcstd(\lambda),M) \cong \uHom(\pcpstd(\lambda)\la -2\delta_{w_0\lambda}\ra, M).
\]
\end{lem}
\begin{proof}
The assertion that $\uHom(\pcstd(\lambda), M)$ is a free $\uEnd(\pcstd(\lambda))$-module is just a restatement of (the dual of)~\cite[Lemma~2.12]{arid}. Next, let us identify $\bk$ with $\uHom(\pcpstd(\lambda)\la -2\delta_{w_0\lambda}\ra,\pcstd(\lambda))$ by Lemma~\ref{lem:uend-delta-k}.  We wish to show that the natural map
\begin{equation}\label{eqn:pcstd-compare}
\uHom(\pcpstd(\lambda)\la -2\delta_{w_0\lambda}\ra,\pcstd(\lambda))
\mathop{\otimes}_{\uEnd(\pcstd(\lambda))} \uHom(\pcstd(\lambda), M) \to
\uHom(\pcpstd(\lambda)\la -2\delta_{w_0\lambda}\ra, M)
\end{equation}
is an isomorphism.  We proceed by induction on the number of steps in a costandard filtration of $M$.

Suppose first that $M = \pccos(\mu)\la n\ra$ for some $\mu \in \bXp$ and some $n \in \Z$.  If $\mu \ne \lambda$, then both sides of~\eqref{eqn:pcstd-compare} vanish, and there is nothing to prove.  If $\mu = \lambda$, then an argument like that in~\cite[Lemma~2.7(3)]{arid} shows that we can replace $M$ on both sides of~\eqref{eqn:pcstd-compare} by the standard object $\pcstd(\lambda)\la n+2\delta_{w_0\lambda}\ra$.  After this change,~\eqref{eqn:pcstd-compare} is obviously an isomorphism.

For general $M$, choose a short exact sequence $0 \to M' \to M \to M'' \to 0$ where both $M'$ and $M''$ have costandard filtrations with fewer steps than that of $M$.  We claim that both sides of~\eqref{eqn:pcstd-compare} take this sequence to a short exact sequence.  For the right-hand side, this holds simply because $\uExt^1(\pcpstd(\lambda)\la -2\delta_{w_0\lambda}\ra, M') = 0$.  For the left-hand side, we first note that $\uExt^1(\pcstd(\lambda), M') = 0$; then, the functor $\uHom(\pcstd(\lambda),{-})$ takes our sequence to a short exact sequence of free $\uEnd(\pcstd(\lambda))$-modules. The desired exactness follows. As a consequence, if~\eqref{eqn:pcstd-compare} is already known to be an isomorphism for $M'$ and $M''$, then it is for $M$ as well.
\end{proof}

\begin{lem}\label{lem:pcpcos-deg0}
Let $\lambda \in \bXp$.  The degree-$0$ component of $\Gamma(\tcN,\cO_\tcN(\lambda))$, regarded just as a graded $\Gv$-representation, can be identified with $\cow(\lambda)$.
\end{lem}
\begin{proof}
It follows from the definition of the grading that the $2i$-th graded component of $\Gamma(\tcN, \cO_\tcN(\lambda))$ is isomorphic to the $\Gv$-representation $\mathrm{ind}_{\Bv}^\Gv (\bk_\lambda \otimes \mathrm{Sym}^i(\fuv)^*)$, where $\mathrm{Sym}^i(\fuv)^*$ is the $i$-th symmetric power of the dual vector space to $\fuv$.  In particular, when $i = 0$, this reduces to $\mathrm{ind}_{\Bv}^\Gv \bk_\lambda \cong \cow(\lambda)$.
\end{proof}

The preceding lemma and the following one together tell us that $\Gamma(\tcN, \cO_\tcN(\lambda))$ is generated as a $\bk[\cN]$-module by its graded component of degree $0$.

\begin{lem}\label{lem:cow-pcpcos-surj}
For any $\lambda \in \bXp$, the obvious map $\cON \otimes \cow(\lambda) \to \Gamma(\tcN, \cO_\tcN(\lambda))$ is surjective.
\end{lem}
\begin{proof}
There is a surjective map of $\Bv$-representations $\cow(\lambda) \to \bk_\lambda$, where $\bk_\lambda$ denotes the $1$-dimensional $\Bv$-representation with weight $\lambda$. From this, we obtain a surjective map of vector bundles $\cO_\tcN \otimes \cow(\lambda) \to \cO_\tcN(\lambda)$ on $\tcN$. Applying $\pi_*$ and using~\eqref{eqn:rational-resoln}, we obtain a map $h: \cON \otimes \cow(\lambda) \to \pi_*\cO_\tcN(\lambda) \cong \pcpcos(\lambda)\la \delta_{w_0\lambda}\ra$.  Let $K$ be the cocone of $h$, so that we have a distinguished triangle
\[
K \to \cON \otimes \cow(\lambda) \overset{h}{\to} \pcpcos(\lambda)\la \delta_{w_0\lambda}\ra \to.
\]
To prove that $h$ is surjective, we must show that $K$ lies in $\CohN$.  The proof of~\cite[Lemma~5.4]{a} yields a slightly different fact: that $h$ is surjective as a morphism in $\PCohN$, and hence that $K \in \PCohN$.  (The statement of~\cite[Lemma~5.4]{a} involves $\wey(\lambda)$ instead of $\cow(\lambda)$, but its proof goes through for any $\Gv$-representation with highest weight $\lambda$.) On the other hand, by~\cite[Theorem~2.15(3)]{arid}, $\cON \otimes \cow(\lambda)$ has a costandard filtration, and hence a proper costandard filtration.  It follows that $K$, which is the kernel of $h$ in $\PCohN$, also has a proper costandard filtration, so it lies in $\CohN$, as desired.
\end{proof}

\subsection{The Mirkovi\'c--Vilonen conjecture for mixed sheaves}

We are now ready adapt the arguments in~\cite[\S6]{arid} to the mixed modular setting.

\begin{lem}\label{lem:mixed-sph-orbit}
Let $\lambda \in \bXp$.  The following conditions on an object $\cF \in \DmixGo(\Gr_\lambda)$ are equivalent:
\begin{enumerate}
\item $\cF$ is pure of weight~$0$.
\item $\gHom(\cF,\ubk)$ is a free graded $\sH^\bullet(\Gr_\lambda)$-module, and $\gHom(\bk,\cF[k]) = 0$ if $k \ne 0$.
\end{enumerate}
\end{lem}
\begin{proof}
Essentially identical to~\cite[Lemma~6.1]{arid}.
\end{proof}

\begin{thm}\label{thm:mixed-mv}
Let $\lambda \in \bXp$.  Then $\cIstd(\lambda)$ is $*$-pure, and $\cIcos(\lambda)$ is $!$-pure.
\end{thm}
One can also show that the stalks of $\cIstd(\lambda)$ and the costalks of $\cIcos(\lambda)$ obey certain parity-vanishing conditions, by using the decomposition of $\DmixI(\Gr)$ into ``even'' and ``odd'' objects as explained in~\cite[\S2.1]{arc:f2}, 
\begin{proof}
Let $\mu$ be a dominant weight such that $\mu \preceq \lambda$. Using adjunction and the equivalence $P_\sph$, we obtain:
\begin{multline*}
\gHom((i^\sph_\mu)^*\cIstd(\lambda), \ubk[k]) \cong
\gHom(\cIstd(\lambda), \cJcos(\mu)\{-\dim \Gr_\mu\}[k]) \\
\cong
\uHom(\cON \otimes \wey(\lambda), \pccos(\mu)\la -\delta_{w_0\mu}-\dim \Gr_\mu\ra[k]).
\end{multline*}
Recall that $\cON \otimes \wey(\lambda)$ has a standard filtration as an object of $\PCohN$.  It follows that the last $\uHom$-group above vanishes for $k \ne 0$.  On the other hand, for $k = 0$, it is a free module over $\uEnd(\pccos(\mu))$, by~\cite[Lemma~2.12]{arid}.

Using Lemma~\ref{lem:uend-nabla}, we see that $\gHom((i^\sph_\mu)^*\cIstd(\lambda), \ubk[k])$ obeys the second condition in Lemma~\ref{lem:mixed-sph-orbit}.  By that lemma, $(i^\sph_\mu)^*\cIstd(\lambda)$ is pure of weight $0$, as desired.
\end{proof}

\section{The regular representation and the regular perverse sheaf}
\label{sect:regular}

In this section, we review a number of basic facts about the regular representation $\bk[\Gv]$ of $\Gv$, and then we translate them into geometric statements about $\Gr$. 

\subsection{The regular representation}


Regard $\bk[\Gv]$ as a $(\Gv \times \Gv)$-module in the usual way: given $f \in \bk[\Gv]$ and $g,h \in \Gv$, we put $((g,h)\cdot f)(x) = f(g^{-1}xh)$.  The results below are elementary and very close to those in, say,~\cite[\S I.3.7]{jantzen}.  We include proofs because we will require slightly finer information about the right $\Gv$-action than is given in~{\it loc.~cit}.

If $V$ and $V'$ are two $\Gv$-representations, we write $V \boxtimes V'$ for their tensor product regarded as a $(\Gv \times \Gv)$-representation.  In an abuse of notation, we sometimes identify $V$ with $V \boxtimes \bk$; i.e., we regard a $\Gv$-representation as a $(\Gv \times \Gv)$-representation by making the second copy of $\Gv$ act trivially.  (To make the first copy act trivially instead, we explicity write $\bk \boxtimes V$.)

\begin{lem}\label{lem:gvgv-swap}
For any $\Gv$-module $V$, there is a natural isomorphism of $(\Gv \times \Gv)$-modules $V \otimes \bk[\Gv] \cong (\bk \boxtimes V) \otimes \bk[\Gv]$.
\end{lem}
\begin{proof}
Identify the underlying vector space of both sides with the space $\Mor(\Gv,V)$ of morphisms $\Gv \to V$.  The two $(\Gv \times \Gv)$-actions above correspond to the following two actions on $\Mor(\Gv,V)$:
\[
((g,h) \cdot_1 f)(x) = gf(g^{-1}xh)
\qquad\text{and}\qquad
((g,h) \cdot_2 f)(x) = hf(g^{-1}xh).
\]
Let $\phi: \Mor(\Gv,V) \to \Mor(\Gv,V)$ be the bijective map given by $\phi(f)(x) = x^{-1}f(x)$. Then $\phi$ intertwines the two actions: $\phi((g,h) \cdot_1 f) = (g,h) \cdot_2 \phi(f)$.
\end{proof}

In the next few statements, given a $(\Gv\times \Gv)$-module $M$, we let
\[
M^{\Gv \times 1} = \{ m \in M \mid \text{$(g,1) \cdot m = m$ for all $g \in \Gv$} \}.
\]
Of course, the second copy of $\Gv$ still acts on $M^{\Gv \times 1}$.  That is, we can regard $M^{\Gv \times 1}$ in a natural way as a $\Gv$-module.

\begin{lem}\label{lem:gv-invt}
For any $\Gv$-module $V$, there is a natural isomorphism of $\Gv$-modules $\theta: V \simto (V \otimes \bk[\Gv])^{\Gv \times 1}$.
\end{lem}
\begin{proof}
Given $v \in V$, let $\theta(v) \in \Mor(\Gv,V)$ be given by $\theta(v)(x) = xv$.  Then, in the notation from the proof of Lemma~\ref{lem:gvgv-swap}, we have
\[
((g,1) \cdot_1 \theta(v))(x) = g \theta(v)(g^{-1}x) = gg^{-1}xv = xv = \theta(v)(x).
\]
That is, $\theta(v) \in \Mor(\Gv,V)^{\Gv \times 1}$.  To see that $\theta$ is an isomorphism, we observe that the map sending $f \in \Mor(\Gv,V)^{\Gv \times 1}$ to $f(1) \in V$ is its inverse.
\end{proof}

\begin{lem}\label{lem:gv-invt-tensor}
Let $M$ be a $\Gv$-equivariant graded $\bk[\cN]$-module. Let $a: M \otimes \bk[\cN] \to M$ be the action map, and let $m: (M \otimes \bk[\Gv])^{\Gv \times 1} \otimes (\bk[\cN] \otimes \bk[\Gv])^{\Gv \times 1} \to (M \otimes \bk[\Gv])^{\Gv \times 1}$ be the map induced by $a$ and by the multiplication map $\bk[\Gv] \otimes \bk[\Gv] \to \bk[\Gv]$.  Then the following diagram commutes:
\[
\xymatrix{
M \otimes \bk[\cN] \ar[d]_a \ar[r]^-{\theta \otimes \theta}_-{\sim} &
  (M \otimes \bk[\Gv])^{\Gv \times 1} \otimes (\bk[\cN] \otimes \bk[\Gv])^{\Gv \times 1} \ar[d]^m \\
M \ar[r]^-{\theta}_-{\sim} & (M \otimes \bk[\Gv])^{\Gv \times 1}}
\]
\end{lem}
\begin{proof}
This is easily seen by tracing through the definition of $\theta$.
\end{proof}

In the special case where $M = \bk[\cN]$, the map $m$ on the right-hand side of the diagram above makes $(\bk[\cN] \otimes \bk[\Gv])^{\Gv \times 1}$ into a commutative ring, equipped with a grading inherited from that on $\bk[\cN]$.  Then, for any $\Gv$-equivariant graded $\bk[\cN]$-module $N$, the space $(N \otimes \bk[\Gv])^{\Gv \times 1}$ is naturally a $\Gv$-equivariant graded $(\bk[\cN] \otimes \bk[\Gv])^{\Gv \times 1}$-module.  The following proposition is immediate consequence of Lemma~\ref{lem:gv-invt-tensor}.

\begin{prop}\label{prop:gv-invt-ring}
There is an isomorphism of $\Gv$-equivariant graded rings $\bk[\cN] \cong (\bk[\cN] \otimes \bk[\Gv])^{\Gv \times 1}$.
If we identify these rings, then for any $\Gv$-equivariant graded $\bk[\cN]$-module $M$, there is a natural isomorphism $M \cong (M \otimes \bk[\Gv])^{\Gv \times 1}$.
\end{prop}

\begin{prop}\label{prop:gv-invt-dual}
For any finitely generated $\Gv$-equivariant graded $\bk[\cN]$-module $M$, there is a natural isomorphism of $\Gv$-equivariant graded $\bk[\cN]$-modules
\[
M \cong \uHom_{\Gv,\bk[\cN]}(\SGD(M), \bk[\cN] \otimes \bk[\Gv]).
\]
\end{prop}
\begin{proof}
Recall that the functor $\SGD({-})$ is defined as $\uRHom_{\bk[\cN]}({-},\bk[\cN])$.  That is, we compute $\uRHom$ in the category of graded $\bk[\cN]$-modules, ignoring the $\Gv$-action; the resulting complex of $\bk[\cN]$-modules is still acted on by $\Gv$.  Since $M \cong \SGD(\SGD(M))$, the complex $\uRHom_{\bk[\cN]}(\SGD(M), \bk[\cN])$ is concentrated in degree zero.  In other words, we have a natural isomorphism $M \cong \uHom_{\bk[\cN]}(\SGD(M), \bk[\cN])$.  Using Lemma~\ref{lem:gv-invt}, we obtain
\begin{multline*}
M \cong \uHom_{\bk[\cN]}(\SGD(M), \bk[\cN])
 \cong \big(\uHom_{\bk[\cN]}(\SGD(M), \bk[\cN]) \otimes \bk[\Gv]\big)^{\Gv \times 1} \\
 \cong \uHom_{\bk[\cN]}(\SGD(M), \bk[\cN] \otimes \bk[\Gv])^{\Gv \times 1}
 \cong \uHom_{\Gv,\bk[\cN]}(\SGD(M), \bk[\cN] \otimes \bk[\Gv]),
\end{multline*}
as desired.
\end{proof}

\subsection{The regular perverse sheaf}
\label{ss:regular-perv}

Let $\cR$ denote the ind-object of $\Perv_\Go(\Gr)$ corresponding to the (left) regular representation $\bk[\Gv]$.  The right action of $\Gv$ on $\bk[\Gv]$ gives rise to a $\Gv$-action on $\cR$.  The multiplication map $\mult: \bk[\Gv] \otimes \bk[\Gv] \to \bk[\Gv]$ is equivariant for the right $\Gv$-action, so it corresponds to a $\Gv$-equivariant morphism $\Sat(\mult): \cR \star \cR \to \cR$.  Consider the graded vector space
\[
\gHom(\sky, \cR).
\]
We make this into a ring in the following way: given $g \in \Hom(\sky, \cR\{n\})$ and $f \in \Hom(\sky, \cR\{m\})$, we define $gf \in \Hom(\sky, \cR\{n+m\})$ to be the composition
\[
\sky \xrightarrow{f} \cR\{m\} \simto \sky \star \cR\{m\} \xrightarrow{g \star \id} \cR\{n\} \star \cR\{m\} \xrightarrow{\Sat(\mult)\{n+m\}} \cR\{n+m\}
\]
Because $\Sat(\mult)$ is $\Gv$-equivariant, $\Gv$ acts on the ring $\gHom(\sky,\cR)$.

Given $\cF \in \DmixI(\Gr)$, a similar construction makes the graded vector space
\[
\gHom(\sky, \cF \star \cR)
\]
into a $\Gv$-equivariant graded right $\gHom(\sky,\cR)$-module. Specifically, given $m \in \Hom(\sky, (\cF \star \cR)\{n\})$ and $f \in \Hom(\sky, \cR\{m\})$, we define $mf \in \Hom(\sky, (\cF \star \cR)\{n+m\})$ to be the composition
\begin{multline*}
\sky \xrightarrow{f} \cR\{m\} \simto \sky \star \cR\{m\} \xrightarrow{m \star \id} (\cF \star \cR)\{n\} \star \cR\{m\} \\ 
\simto
(\cF \star \cR \star \cR)\{n+m\}\xrightarrow{(\id \star \Sat(\mult))\{n+m\}} (\cF \star \cR)\{n+m\}.
\end{multline*}

\begin{thm}\label{thm:coh-cn}
There is an isomorphism of $\Gv$-equivariant graded rings
\[
\gHom(\sky, \cR) \cong \bk[\cN].
\]
If we identify these rings, then for any $\cF \in \DmixGo(\Gr)$ such that $P_\sph(\cF) \in \CohN$, there is a natural isomorphism
\[
\gHom(\sky, \cF \star \cR) \cong P_\sph(\cF)
\]
of $(\Gv \times \Gm)$-equivariant coherent sheaves on $\cN$.
\end{thm}

\begin{figure}
{\small\[
\xymatrix{
\gHom(\sky, \cF \star \cR) \otimes \gHom(\sky, \cR) \ar[d]_p \ar@{<->}[r]^-{\sim} &
  {\begin{array}{l} \uHom_\Gv(\bk, M \otimes \bk[\Gv]) \\
   \qquad\qquad\qquad {}\otimes \uHom_\Gv(\bk,\bk[\cN] \otimes \bk[\Gv]) \end{array}} 
   \ar[d]^{p'} \\
\gHom(\cR, \cF \star \cR \star \cR) \otimes \gHom(\sky, \cR) \ar[d]_q \ar@{<->}[r]^-{\sim} &
  {\begin{array}{l} \uHom_\Gv(\bk[\Gv], M \otimes \bk[\Gv] \otimes \bk[\Gv])  \\
   \qquad\qquad\qquad {}\otimes \uHom_\Gv(\bk,\bk[\cN] \otimes \bk[\Gv]) \end{array}} \ar[d]^{q'} \\
\gHom(\sky, \cF \star \cR \star \cR) \ar[d]_r \ar@{<->}[r]^-{\sim} &
  \uHom_\Gv(\bk, M \otimes \bk[\Gv] \otimes \bk[\Gv]) \ar[d]^{r'} \\
\gHom(\sky, \cF \star \cR) \ar@{<->}[r]^-{\sim} &
  \uHom_\Gv(\bk, M \otimes \bk[\Gv])}
\]}
\caption{Modules for $\gHom(\sky,\cR)$}\label{fig:mod1}
\end{figure}

\begin{proof}
We have the following sequence of isomorphisms of graded vector spaces, where the first step is implied by Proposition~\ref{prop:sph-main}, and the last by Proposition~\ref{prop:gv-invt-ring}:
\begin{multline*}
\gHom(\sky, \cR) \cong \uHom_{\CohN}(\cON, \cON \otimes \bk[\Gv]) \\
\cong \uHom_\Gv(\bk, \bk[\cN] \otimes \bk[\Gv]) \cong (\bk[\cN] \otimes \bk[\Gv])^{\Gv \times 1} \cong \bk[\cN].
\end{multline*}
In fact, this is an isomorphism of $\Gv$-modules, since the $\Gv$-action on $\cR$ is defined in terms of the right $\Gv$-action on $\bk[\Gv]$.  Next, for $\cF \in \DmixGo(\Gr)$, let $M = P_\sph(\cF)$, and assume that $M \in \CohN$.  The same reasoning as above gives us isomorphisms of graded $\Gv$-representations
\[
\gHom(\sky, \cF \star \cR) \cong \uHom_\Gv(\bk, M \otimes \bk[\Gv]) \cong M.
\]

To study the ring structure on $\gHom(\sky, \cR)$ as well as the module structure on $\gHom(\sky, \cF \star \cR)$, we refer to Figure~\ref{fig:mod1}. The horizontal arrows all arise from natural isomorphisms of the kind described above.  The arrow labeled $p$ is induced by convolution with the morphism of (ind-)perverse sheaves $\Sat(\eta): \sky \to \cR$, where $\eta: \bk \to \bk[\Gv]$ is the unit.  The map $p'$ is induced by $\eta$ itself.  Thus, the commutativity of the uppermost square in Figure~\ref{fig:mod1} follows from the compatibility with $\Sat$ in Proposition~\ref{prop:sph-main}.

Similar reasoning applies to the bottommost square.  There, $r'$ is induced by the multiplication map $\mult: \bk[\Gv] \otimes \bk[\Gv] \to \bk[\Gv]$, and $r$ by $\Sat(\mult): \cR \star \cR \to \cR$.  Finally, the arrows labeled $q$ and $q'$ are both given by composition of maps, so the commutativity of the middle square is obvious.  We conclude that the entire diagram in Figure~\ref{fig:mod1} commutes.

The composition $rqp$ defines the $\gHom(\sky,\cR)$-module structure on $\gHom(\sky, \cF \star \cR)$. On the other hand, we can identify the space $\uHom_\Gv(\bk, M \otimes \bk[\Gv])$ with $(M \otimes \bk[\Gv])^{\Gv \times 1}$, and likewise for the other $\uHom$-groups in the right-hand column of Figure~\ref{fig:mod1}. Under these identifications, the composition $r'q'p'$ coincides with the map that was denoted $m$ in Lemma~\ref{lem:gv-invt-tensor}.  Thus, by combining Figure~\ref{fig:mod1} with Lemma~\ref{lem:gv-invt-tensor}, we obtain the following commutative diagram:
\[
\xymatrix{
\gHom(\sky, \cF \star \cR) \otimes \gHom(\sky, \cR) \ar[d] \ar@{<->}[r]^-{\sim} &
  M \otimes \bk[\cN] \ar[d] \\
\gHom(\sky, \cF \star \cR) \ar@{<->}[r]^-{\sim} & M}
\]
In the special case where $\cF = \sky$ and $M = \bk[\cN]$, this diagram shows that the isomorphism of graded $\Gv$-modules $\gHom(\sky,\cR) \cong \bk[\cN]$ is actually a ring isomorphism.  Then, for general $\cF$, it identifies the $\gHom(\sky,\cR)$- and $\bk[\cN]$-module structures on $\gHom(\sky,\cF \star \cR) \cong M$, as desired.
\end{proof}

\subsection{Standard sheaves and the regular perverse sheaf}

We conclude this section with a study of certain $\gHom$-groups involving standard sheaves.

\begin{lem}\label{lem:costalk-compare}
Suppose $\cF \in \DmixGo(\Gr)$ has the property that $(i^\sph_\lambda)^!\cF$ is pure of weight $0$.  Then $\gHom(\cJstd(\lambda),\cF)$ is a free $\sH^\bullet(\Gr_\lambda)$-module.  Moreover, there is a natural isomorphism
\[
\bk \otimes_{\sH^\bullet(\Gr_\lambda)} \gHom(\cJstd(\lambda),\cF) \cong \gHom(\uistd(w_0\lambda)\{-\delta_{w_0\lambda}\},\cF).
\]
\end{lem}
\begin{proof}
By adjunction, we see that $\gHom(\cJstd(\lambda),\cF) \cong \gHom(\cJstd(\lambda), (i^\sph_\lambda)_*(i^\sph_\lambda)^!\cF)$, and that $\gHom(\uistd(\lambda),\cF) \cong \gHom(\uistd(\lambda), (i^\sph_\lambda)_*(i^\sph_\lambda)^!\cF)$.  Thus, we may as well assume that $\cF \cong (i^\sph_\lambda)_*(i^\sph_\lambda)^!\cF$.  Since $(i^\sph_\lambda)^!\cF$ is assumed to be pure, it is a direct sum of objects of the form $\ubk\{n\}$.  It suffices, then, to prove the lemma in the special case where $(i^\sph_\lambda)^!\cF \cong \ubk\{\dim \Gr_\lambda\}$, and $\cF \cong \cJcos(\lambda)$.

In that case, we have $\gHom(\cJstd(\lambda),\cF) \cong \gHom(\ubk_{\Gr_\lambda}, \ubk_{\Gr_\lambda}) \cong \sH^\bullet(\Gr_\lambda)$.  On the other hand, by adjunction, we have
\begin{multline*}
\gHom(\uistd(w_0\lambda)\{-\delta_{w_0\lambda}\}, \cF) \cong \gHom(\ubk_{\schub{w_0\lambda}} \{\dim \Gr_\lambda - 2\delta_{w_0\lambda}\}, i_{w_0\lambda}^!\cF) \\
\cong \gHom(\ubk_{\schub{w_0\lambda}} \{\dim \Gr_\lambda - 2\delta_{w_0\lambda}\}, \ubk_{\schub{w_0\lambda}} \{\dim \Gr_\lambda - 2\delta_{w_0\lambda}\}) \cong \bk.
\end{multline*}
The adjunction map $\uistd(w_0\lambda)\{-\delta_{w_0\lambda}\} \to \cJstd(\lambda)$ induces a natural map
\[
\gHom(\cJstd(\lambda),\cF) \to \gHom(\uistd(w_0\lambda)\{-\delta_{w_0\lambda}\}, \cF)
\]
that can clearly be identified with the natural quotient map $\sH^\bullet(\Gr_\lambda) \to \bk$ of $\sH^\bullet(\Gr_\lambda)$-modules.  The result follows.
\end{proof}

\begin{cor}\label{cor:costalk-compare}
Suppose $\cF \in \Perv^\mix_{(\Go)}(\Gr)$ has the property that $(i^\sph_\lambda)^!\cF$ is pure of weight $0$.  Then the natural map $\Hom(\cJstd(\lambda),\cF) \to \Hom(\uistd(w_0\lambda)\{-\delta_{w_0\lambda}\},\cF)$ induced by the adjunction $\uistd(w_0\lambda)\{-\delta_{w_0\lambda}\} \to \cJstd(\lambda)$ is an isomorphism.
\end{cor}
\begin{proof}
Note that the $\Hom$-groups in this statement are the degree-$0$ components of the graded $\gHom$-groups in the preceding lemma.  Since $\cJstd(\lambda)$ lies in ${}^p\!\DmixGo(\Gr)^{\le 0}$, the assumption that $\cF$ is perverse implies that $\Hom(\cJstd(\lambda),\cF\{n\}) = 0$ for $n < 0$, or in other words, that $\gHom(\cJstd(\lambda), \cF)$ is concentrated in nonnegative degrees.  The result then follows from Lemma~\ref{lem:costalk-compare}.
\end{proof} 

\begin{prop}\label{prop:std-regular}
We have the following isomorphisms in $\CohN$:
\[
\gHom(\cJstd(-w_0\lambda), \cR) \cong \pccos(\lambda)\la -\delta_{w_0\lambda}\ra
\qquad\text{and}\qquad
\gHom(\uistd(-\lambda), \cR) \cong \pcpcos(\lambda).
\]
\end{prop}
\begin{proof}
Via $P_\sph$ and Corollary~\ref{cor:psph-std}, we have
\[
\gHom(\cJstd(-w_0\lambda), \cR) \cong \uHom(\pcstd(-w_0\lambda)\la \delta_{w_0\lambda}\ra, \cON \otimes \bk[\Gv]).
\]
By Proposition~\ref{prop:gv-invt-dual}, the latter is naturally isomorphic to $\pccos(\lambda)\la -\delta_{w_0\lambda}\ra$.

Next, by Theorem~\ref{thm:mixed-mv}, the ind-perverse sheaf $\cR$ is $!$-pure of weight $0$, so Lemma~\ref{lem:costalk-compare} tells us that $\gHom(\uistd(-\lambda)\{-\delta_{w_0\lambda}\}, \cR) \cong \bk \otimes_{\sH^\bullet(\Gr_\lambda)} \gHom(\cJstd(-w_0\lambda), \cR)$.  Proposition~\ref{prop:aj-compute} then implies that $\gHom(\uistd(-\lambda)\{-\delta_{w_0\lambda}\}, \cR) \cong \pcpcos(\lambda)\la \delta_{w_0\lambda}\ra$, as desired.
\end{proof}

\section{Mixed modular Wakimoto sheaves}
\label{sect:wakimoto}

\emph{Wakimoto sheaves}, introduced by Mirkovi\'c, are certain sheaves on the affine flag variety or the affine Grassmannian that have favorable convolution and $\Ext$-vanishing properties.  In this section, we study the basic properties of Wakimoto sheaves in the mixed modular setting.  The results are closely modeled on those of~\cite[\S3.2]{ab} and~\cite[\S8]{abg}.

\subsection{Preliminaries on the affine flag variety}

Let $\Fl = \Gk/I$ denote the affine flag variety for $G$. Recall that $I$-orbits on $\Fl$ are labeled by the extended affine Weyl group $W_\aff$. Given $w \in W_\aff$, let $\Fl_w$ denote the corresponding orbit, and let $j_w: \Fl_w \hookrightarrow \Fl$ be the inclusion map.  For $w \in W_\aff$, let $\cE_w$ denote the unique indecomposable parity sheaf supported on $\overline{\Fl_w}$ and whose restriction to $\Fl_w$ is $\ubk\{\dim \Fl_w\}$.  We denote the standard and costandard perverse sheaves in $\Perv^\mix_I(\Fl)$ by
\[
\bs_w := j_{w!}\ubk\{\dim \Fl_w\}
\qquad\text{and}\qquad
\bc_w := j_{w*}\ubk\{\dim \Fl_w\}.
\]
The category $\Dmix_I(\Fl)$ is equipped with a convolution product, and there is a convolution action of $\Dmix_I(\Fl)$ on $\Dmix_I(\Gr)$.  For basic results on convolution in the mixed modular setting, see~\cite[\S4]{arc:f2}.  We will need the following slight refinement of~\cite[Proposition~4.4]{arc:f2}.

\begin{lem}\label{lem:cos-conv}
For $w_1, w_2\in W_\aff$ such that $\ell (w_1w_2) = \ell(w_1) +\ell(w_2)$, there is a canonical isomorphism
\begin{equation}\label{eqn:cos-conv}
\bc_{w_1w_2} \cong \bc_{w_1} \star \bc_{w_2}.
\end{equation}
Moreover, for $w_1, w_2, w_3 \in W_\aff$ with $\ell(w_1w_2w_3) = \ell(w_1) + \ell(w_2) + \ell(w_3)$, the two isomorphisms $\bc_{w_1w_2w_3} \cong \bc_{w_1} \star \bc_{w_2} \star \bc_{w_3}$ induced by~\eqref{eqn:cos-conv} coincide.

In addition, each $\bc_w$ is invertible: we have $\bc_w \star \bs_{w^{-1}} \cong \bs_{w^{-1}} \star \bc_w \cong \bc_e$.
\end{lem}

For $\Qlb$-sheaves (see~\cite[Lemma~8]{ab}), a shorter proof is possible: one can prove a property like~\eqref{eqn:parity-conv} below for standard sheaves directly. The definition of convolution in the mixed modular setting always involves parity sheaves as an intermediary; for this reason, the argument below must consider parity sheaves first.

\begin{proof}
We begin with the observation that if $\ell (w_1w_2) = \ell(w_1) +\ell(w_2)$, then there is a canonical isomorphism
\begin{equation}\label{eqn:parity-conv}
(\cE_{w_1} \star \cE_{w_2})|_{\Fl_{w_1w_2}} \cong \ubk\{\ell(w_1w_2)\}.
\end{equation}
Indeed, this follows from a study of the convolution diagram (see~\cite[\S4.1]{jmw}).  Now, $\cE_{w_1w_2}$ is a direct summand of the parity complex $\cE_{w_1} \star \cE_{w_2}$.  Choose maps
\[
\cE_{w_1w_2} \xrightarrow{i} \cE_{w_1} \star \cE_{w_2} \xrightarrow{p} \cE_{w_1w_2}
\]
such that $p \circ i = \id_{\cE_{w_1w_2}}$, $i \circ p$ is an idempotent, and both $i$ and $p$ are compatible with~\eqref{eqn:parity-conv}.  That is, the restriction to $\Fl_{w_1w_2}$ of each of $i$ and $p$ should coincide with the isomorphism~\eqref{eqn:parity-conv}.  The fact that the last condition can be satisfied follows from~\cite[Corollary~2.9]{jmw}.

Next, consider the canonical maps $\cE_{w_i} \to \bc_{w_i}$ and $\cE_{w_1w_2} \to \bc_{w_1w_2}$.  It is easy to see that there are unique maps $i_0$, $p_0$ making the following diagram commute:
\[
\xymatrix{
\cE_{w_1w_2} \ar[d] \ar[r]^{i} & \cE_{w_1} \star \cE_{w_2} \ar[d] \ar[r]^{p} & \cE_{w_1w_2} \ar[d]  \\
\bc_{w_1w_2} \ar[r]^{i_0} & \bc_{w_1} \star \bc_{w_2} \ar[r]^{p_0} & \bc_{w_1w_2} }
\]
In fact, the maps $i_0$ and $p_0$ are determined by the restrictions $i|_{\Fl_{w_1w_2}}$ and $p|_{\Fl_{w_1w_2}}$. In other words, they are determined by the canonical isomorphism~\eqref{eqn:parity-conv}, and are independent of the choice of $i$ and $p$.

We already know by~\cite[Proposition~4.6]{arc:f2} that $\bc_{w_1} \star \bc_{w_2}$ is abstractly isomorphic to $\bc_{w_1w_2}$.  Since $\End(\bc_{w_1w_2}) \cong \bk$, the maps $i_0$ and $p_0$ must both  be isomorphisms, inverse to one another.  These maps constitute the canonical isomorphism~\eqref{eqn:cos-conv}.

The associativity assertion follows from the fact that the two isomorphisms
\[
(\cE_{w_1} \star \cE_{w_2} \star \cE_{w_3})|_{\Fl_{w_1w_2w_3}} \cong \ubk\{\ell(w_1w_2w_3)\}
\]
induced by~\eqref{eqn:parity-conv} coincide.  Finally, the invertibility assertion is just a restatement of~\cite[Proposition~4.4(2)]{arc:f2}.
\end{proof}

\begin{lem}[cf.~{\cite[Proposition~8.2.4]{abg}}]\label{lem:waki-fl-support}
For any $y,w \in W_\aff$, the object $\bc_y \star \bs_w$ is a perverse sheaf.  It is supported on $\overline{\Fl_{yw}}$, and $(\bc_y \star \bs_w)|_{\Fl_{yw}} \cong \ubk\{\dim \Fl_{yw}\}$.  The same results hold for $\bs_y \star \bc_w$.
\end{lem}
\begin{proof}
The fact that $\bc_y \star \bs_w$ is perverse is contained in~\cite[Proposition~4.6]{arc:f2}. 

Let $\cH$ be the Grothendieck group of $\Dmix_I(\Fl)$.  For an object $\cF \in \Dmix_I(\Fl)$, we denote its class in $\cH$ by $[\cF]$. Of course, the convolution product on $\Dmix_I(\Fl)$ makes $\cH$ into a ring. We also make it into a $\Z[q^{1/2},q^{-1/2}]$-algebra (where $q^{1/2}$ is an indeterminate) by setting $q^{1/2}[\cF] = [\cF\la 1\ra]$.  It is well known that $\cH$ is none other than the extended affine Hecke algebra associated to $W_\aff$.  Indeed, one can use~\cite[Proposition~4.4]{arc:f2} to check that the elements $T_w := [\bs_w \{-\ell(w)\}]$ satisfy the defining relations for the Hecke algebra.

Now, consider the ring $\bar\cH := \cH/(q^{1/2}-1)$, which can be identified with the group ring $\Z[W_\aff]$.  Let $\{ \bar T_w \mid w \in W_\aff \}$ be the standard basis for $\Z[W_\aff]$. For $\cF \in \Dmix_I(\Fl)$, let $\overline{[\cF]}$ denote its image in $\bar \cH$.  We then have $\overline{[\bs_w]} = (-1)^{\ell(w)} \bar T_w$.  On the other hand, we have $[\bc_w] = [\bs_{w^{-1}}]^{-1} = (-1)^{\ell(w)} q^{\ell(w)/2} T_{w^{-1}}^{-1}$.  It follows that
\[
\overline{[\bs_w]} = \overline{[\bc_w]} =(-1)^{\ell(w)} \bar T_w
\qquad\text{for all $w \in W_\aff$.}
\]
Since $\ell(y) + \ell(w) \equiv \ell(yw) \pmod 2$, we have
\begin{equation}\label{eqn:mirk-lem-groth}
\overline{[\bc_y \star \bs_w]} = (-1)^{\ell(y)+\ell(w)}\bar T_y \bar T_w = (-1)^{\ell(yw)}\bar T_{yw} = \overline{[\bc_{yw}]}.
\end{equation}

Recall that for two objects $\cF, \cG \in \Perv^\mix_I(\Fl)$, we have $[\cF] = [\cG]$ in $\cH$ if and only if $\cF$ and $\cG$ have the same composition factors (with multiplicities).  Similarly, we have $\overline{[\cF]} = \overline{[\cG]}$ if and only if $\cF$ and $\cG$ have the same composition factors up to Tate twist.  Thus,~\eqref{eqn:mirk-lem-groth} lets us compare the composition factors of $\bc_y \star \bs_w$ with those of $\bc_{yw}$.  Specifically, $\bc_y \star \bs_w$ must contain some $\IC_{yw}\la n\ra$ as a composition factor with multiplicity $1$, and all other composition factors must be supported on $\overline{\Fl_{yw}} \smallsetminus \Fl_{yw}$.  In particular, $\bc_y \star \bs_w$ is supported on $\overline{\Fl_{yw}}$, and $(\bc_y \star \bs_w)|_{\Fl_{yw}} \cong \ubk\{\dim \Fl_{yw}\}\la n\ra$ for some $n$.  

It remains to show that $n = 0$.  For this, we proceed by induction on the length of $w$. If $\ell(w) = 0$, then $\bs_w = \bc_w$, so we have $\bc_y \star \bs_w \cong \bc_y \star \bc_w \cong \bc_{yw}$, and the statement is clear.  Otherwise, write $w = w's$ where $s$ is a simple reflection, and $\ell(w') = \ell(w) - 1$.  By induction, we have $(\bc_y \star \bs_{w'})|_{\Fl_{yw'}} \cong \ubk\{ \dim \Fl_{yw'}\}$.

Suppose first that $yw' < yw$.  This implies that $\bs_{yw'} \star \bs_s \cong \bs_{yw}$.  There is a natural (nonzero) map $\bs_{yw'} \to \bc_y \star \bs_{w'}$. Since $\bs_s$ is an invertible object, applying $({-})\star \bs_s$ gives a nonzero map $\bs_{yw} \to \bc_y \star \bs_w$.  By adjunction, we obtain a nonzero map $\ubk\{\dim \Fl_{yw}\} \to (\bc_y \star \bs_w)|_{\Fl_{yw}}$.  Therefore, $n = 0$.

Similarly, if $yw' > yw$, we consider the natural (nonzero) map $\bc_y \star \bs_{w'} \to \bc_{yw'}$.  This time, we have $\bc_{yw'} \star \bs_s \cong \bc_{yw}$, so applying $({-})\star \bs_s$ gives a nonzero map $\bc_y \star \bs_w \to \bc_{yw}$.  Again, by adjunction, we obtain a nonzero map
$(\bc_y \star \bs_w)|_{\Fl_{yw}} \to \ubk\{\dim \Fl_{yw}\}$, and the result follows.
\end{proof}

\subsection{Projection to the affine Grassmannian}

Let $\varpi: \Fl \to \Gr$ be the obvious projection map.  This is a smooth, proper, stratified morphism.  The following elementary lemmas relating convolution and $\varpi$ are well known (at least in the non-mixed case), but we give their proofs for completeness.

\begin{lem}\label{lem:conv-flgr}
Let $\cF \in \DmixI(\Fl)$.
\begin{enumerate}
\item For $\cG \in \Dmix_I(\Fl)$, there is a natural isomorphism $\cF \star \varpi_*\cG \cong \varpi_*(\cF \star \cG)$.
\item For $\cG \in \Dmix_I(\Gr)$, there is a natural isomorphism $\cF \star \varpi^*\cG \cong \varpi^*(\cF \star \cG)$.
\end{enumerate}
\end{lem}
\begin{proof}
For both statements, it suffices to consider the case where $\cF$ and $\cG$ are both parity sheaves.  In this case, we can compute the convolution product in the ordinary (non-mixed) derived category instead.  Note that in the following diagram, every square is cartesian:
\[
\xymatrix{
\Fl \times \Fl \ar[d]_{\id \times \varpi} &
  \Gk \times \Fl \ar[l] \ar[r] \ar[d]_{\id \times \varpi} &
  \Gk \times^I \Fl \ar[d] \ar[r] &
  \Fl \ar[d]^{\varpi} \\
\Fl \times \Gr &
  \Gk \times \Gr \ar[l] \ar[r] &
  \Gk \times^I \Gr \ar[r] &
  \Gr}
\]
The results follow by tracing through the definition of convolution.
\end{proof}

\newcommand{\tboxtimesp}{\mathbin{\tilde{\mathord{\boxtimes}}{}'}}
\begin{lem}\label{lem:conv-fl-sph}
Let $\cF \in \DmixI(\Fl)$ and $\cG \in \Perv_\Go(\Gr)$.  There is a natural isomorphism $\cF \star \cG \cong \varpi_*\cF \star \cG$.
\end{lem}
\begin{proof}
As above, assume that $\cF$ and $\cG$ are both parity sheaves.  We will give an alternative description of the object $\cF \tboxtimes \cG$ on $\Gk \times^I \Gr$, using the following commutative diagram.
\[
\xymatrix{
& \Gk \times \Gr \ar[dl]_p \ar[dr]^q \ar[d]_i \\
\Fl \times \Gr & (\Gk \times^I \Go) \times \Gr \ar[l]_-{\tilde p} \ar[r]^-{\tilde q} & \Gk \times^I \Gr}
\]
The maps are defined as follows:
\begin{align*}
p(g, x\Go) &= (gI,x\Go) & q(g,x\Go) &= (g,x\Go) & i(g,x\Go) &= (g,1,x\Go) \\
\tilde p(g,h,x\Go) &= (gI,x\Go) & \tilde q(g,h,x\Go) &= (g,hx\Go)
\end{align*}
Recall that $\cF \tboxtimes \cG$ is defined to be the unique object on $\Gk \times^I \Gr$ such that $q^*(\cF \tboxtimes \cG) \cong p^*(\cF \boxtimes \cG)$.  We claim that it is also the unique object 
satisfying
\[
\tilde q^*(\cF \tboxtimes \cG) \cong \tilde p^*(\cF \boxtimes \cG).
\]
To see this, observe first that because $\cG$ is $\Go$-equivariant, the object $\tilde p^*(\cF \boxtimes \cG)$ is $\Go$-equivariant for the ``diagonal'' $\Go$-action on $(\Gk \times^I \Go) \times \Gr$, in which $m \in \Go$ acts by $m \cdot (g,h,x\Go) = (g,hm^{-1}, mx\Go)$.  This action is free, and $\tilde q$ is the quotient by this action, so there exists a unique object on $\Gk \times^I \Gr$, say $\cF \tboxtimesp \cG$, such that $\tilde q^*(\cF \tboxtimesp \cG) \cong \tilde p^*(\cF \boxtimes \cG)$.  Applying $i^*$, we find that $q^*(\cF \tboxtimesp \cG) \cong p^*(\cF \boxtimes \cG)$, so we must have $\cF \tboxtimesp \cG = \cF \tboxtimes \cG$.

Now form the following commutative diagram, where $r(g,h,x\Go) = (gh,x\Go)$.
\[
\xymatrix{
\Fl \times \Gr \ar[d] &
  (\Gk \times^I \Go) \times \Gr \ar[d]_r \ar[l]_-{\tilde p} \ar[r]^-{\tilde q} &
  \Gk \times^I \Gr \ar[d]_s \ar[dr] \\
\Gr \times \Gr &
  \Gk \times \Gr \ar[l] \ar[r] &
  \Gk \times^\Go \Gr \ar[r] & \Gr}
\]
It is easy to check that both squares are cartesian. Then, the base change theorem implies that $s_*(\cF \tboxtimes \cG) \cong \varpi_*\cF \tboxtimes \cG$, and the result follows.
\end{proof}

\subsection{Wakimoto sheaves}
\label{ss:wakimoto}

Identify $\bX$ with a subset of $W_\aff$ as usual. Also, for $\lambda \in \bX$, let $\bk_\lambda$ denote the $\Tv$-representation of weight $\lambda$.

\begin{lem}[cf.~{\cite[Corollary~1(a)]{ab}}]
The assignment $\bk_\lambda \mapsto \bc_\lambda$ for $\lambda \in \bXp$ extends to a monoidal functor $\cW: \Rep(\Tv) \to \Dmix_I(\Fl)$.
\end{lem}
\begin{proof}
Recall~\cite[Proposition~4.4(2)]{arc:f2} that each $\bc_w$ is an invertible object of $\Dmix_I(\Fl)$.  With this observation in hand, the proof of~\cite[Corollary~1(a)]{ab} can be repeated verbatim.
\end{proof}

As in~\cite{abg}, we will write $\cW_\lambda$ instead of $\cW(\bk_\lambda)$.  Objects of this form are called \emph{Wakimoto sheaves}.  The construction implies that for any $\lambda \in \bX$, we have
\begin{equation}\label{eqn:waki-defn}
\cW_\lambda \cong \bc_\mu \star \bs_{-\nu}\qquad
\text{if $\lambda = \mu - \nu$ and $\mu, \nu \in \bXp$.}
\end{equation}
In particular, for $\lambda\in\bXp$, we have $\cW_\lambda = \bc_\lambda$ and $\cW_{-\lambda} \cong \bs_{-\lambda}.$  By~\cite[Proposition 4.6]{arc:f2} and Proposition~\ref{prop:adv-conv}, $\cW_\lambda$ is both perverse and adverse. We also put
\[
\ocW_\lambda := \cW_\lambda \star 1_{\Gr} \cong \varpi_*\cW_\lambda.
\]
By Lemma~\ref{lem:adv-smooth}, the $\ocW_\lambda$ are again adverse.  (They are not perverse in general.)  The following fact about these objects is a consequence of Lemma~\ref{lem:waki-fl-support}.

\begin{lem}\label{lem:waki-support}
For any $\lambda \in \bX$, $\ocW_\lambda$ is supported on $\overline{\schub\lambda}$, and there is a canonical isomorphism $\ocW_\lambda|_{\schub\lambda} \cong \ubk\{\dim \schub\lambda - \delta_\lambda\}$.  In the special case where $\lambda \in \bXp$, we have
\begin{equation}\label{eqn:waki-support}
\ocW_\lambda \cong \uicos(\lambda)
\qquad\text{and}\qquad
\ocW_{-\lambda} \cong \uistd(-\lambda)\{-\delta_{-\lambda}\}.
\end{equation}
\end{lem}
\begin{proof}
Identical to~\cite[Corollary~8.3.2]{abg}.
\end{proof}

The next lemma is a variation on Proposition~\ref{prop:conv-twist}.

\begin{lem}\label{lem:conv-waki-twist}
Let $\lambda \in \bX$, and let $\cF, \cG \in \Dmix_I(\Gr)$.  The natural map
\[
\gHom(\cF,\cG) \to \gHom(\cW_\lambda \star \cF, \cW_\lambda \star \cG)
\]
is an isomorphism of $\sH^\bullet_I(\pt)$-modules.
\end{lem}
\begin{proof}[Proof Sketch]
This map is, of course, at least an isomorphism of graded vector spaces, since $\cW_\lambda \star ({-})$ is an equivalence of categories.  If we prove the statement for dominant weights, then it follows for antidominant weights (since $\bs_{-\lambda} \star ({-})$ is the inverse functor to $\bc_\lambda \star ({-})$), and then for all weights by~\eqref{eqn:waki-defn}.

For dominant weights, the proof is very close that of Proposition~\ref{prop:conv-twist}.  We will review the main points.  Recall that $W_\aff$ acts naturally on the maximal torus $T$, and that this action factors through $W_\aff \twoheadrightarrow W$ (see~\cite[\S13.2.2]{kumar}). This induces an action of $W_\aff$ on $\sH^\bullet_T(\pt) \cong \sH^\bullet_I(\pt)$.  Now, the equivariant cohomology $\sH^\bullet_I(\Fl_w)$ is equipped with two actions of $\sH^\bullet_I(\pt)$.  These actions do not coincide in general; rather, they differ by the action of $w$ of $\sH^\bullet_I(\pt)$.  As a consequence, the natural map
\[
\gHom(\cF,\cG) \to \gHom(\bc_w \star \cF, \bc_w \star \cG)
\]
is a $w$-twisted homomorphism of $\sH^\bullet_I(\pt)$-modules.  When $w$ is a dominant weight, it acts trivially on $T$ and on $\sH^\bullet_I(\pt)$, so the map above is indeed a homomorphism of $\sH^\bullet_I(\pt)$-modules, as desired.
\end{proof}

\subsection{Convolution with monodromic objects}

For $\Qlb$-sheaves, it is explained in~\cite[\S8.9]{abg} that by considering lifts of the $\bc_w$ and $\bs_w$ to the ``thick affine flag variety'' $\widetilde{\Fl}$, one can define a functor
\begin{equation}\label{eqn:abg-convolve}
\cW_\lambda \star ({-}): \DmixI(\Gr,\Qlb) \to \DmixI(\Gr,\Qlb).
\end{equation}
That is, one can drop the $I$-equivariance condition for objects on $\Gr$.  Unfortunately, we cannot imitate this in the mixed modular setting, because there is currently no suitable theory of ``mixed modular sheaves'' on $\widetilde{\Fl}$.

We will not attempt to define convolution in the generality of~\eqref{eqn:abg-convolve}.  Instead, we will see in the next few statements that for certain special classes of morphisms and objects in $\DmixI(\Gr)$, we can recover a ``shadow'' of the undefined functor~\eqref{eqn:abg-convolve}, by lifting to $\Dmix_I(\Gr)$.

\begin{prop}\label{prop:conv-waki-hom}
Let $\cF, \cG \in \Dmix_I(\Gr)$ be objects such that $\gHom^i(\cF,\cG)$ is a free $\sH^\bullet_I(\pt)$-module for all $i \in \Z$.  Then there is a unique isomorphism $\omega_\lambda$ making the following diagram commute:
\[
\xymatrix{
\gHom_{\Dmix_I(\Gr)}(\cF, \cG) \ar[r]^-{\sim}_-{\cW_\lambda\star -}  \ar[d] &
  \gHom_{\Dmix_I(\Gr)}(\cW_\lambda \star \cF, \cW_\lambda \star \cG) \ar[d] \\
\gHom_{\DmixI(\Gr)}(\For(\cF), \For(\cG)) \ar[r]^-{\sim}_-{\omega_\lambda} &
  \gHom_{\DmixI(\Gr)}(\For(\cW_\lambda \star \cF), \For(\cW_\lambda \star \cG))
}
\]
\end{prop}
This proposition applies, for instance, when $\cF$ is $*$-pure and $\cG$ is $!$-pure.  In particular, when $\mu \in \bXp$ and $\cG$ is $!$-pure, this proposition gives us a map
\[
\omega_\lambda: \gHom(\ocW_{-\mu}, \For(\cG)) \simto \gHom(\ocW_{\lambda - \mu}, \For(\cW_\lambda \star \cG)).
\]
This is the most common circumstance in which Proposition~\ref{prop:conv-waki-hom} will be invoked.

\begin{proof}
Proposition~\ref{prop:eqvt-hom-dg} gives rise to a spectral sequence
\[
\Tor_{-p}^{\sH^\bullet_I(\pt)}(\gHom^q(\cF,\cG), \bk) \Longrightarrow \gHom^{p+q}(\For(\cF), \For(\cG)).
\]
But the assumption that all the $\gHom^q(\cF,\cG)$ are free means that the $\Tor$-groups vanish except when $p = 0$.  Setting $q = 0$, we obtain an isomorphism
\begin{equation}\label{eqn:conv-waki-hom1}
\gHom(\cF,\cG) \otimes_{\sH^\bullet_I(\pt)} \bk \cong \gHom(\For(\cF),\For(\cG)).
\end{equation}
In particular, the map $\gHom(\cF,\cG) \to \gHom(\For(\cF),\For(\cG))$ is surjective.  It follows immediately that if $\omega_\lambda$ exists, it is unique.

Next, Lemma~\ref{lem:conv-waki-twist} implies that all $\gHom^i(\cW_\lambda \star \cF, \cW_\lambda \star \cG)$ are also free $\sH^\bullet_I(\pt)$-modules, so the considerations above apply here as well.  In particular, we have
\begin{equation}\label{eqn:conv-waki-hom2}
\gHom(\cW_\lambda \star \cF,\cW_\lambda \star\cG) \otimes_{\sH^\bullet_I(\pt)} \bk \cong \gHom(\For(\cW_\lambda \star \cF),\For(\lambda \star \cG)).
\end{equation}
Via~\eqref{eqn:conv-waki-hom1} and~\eqref{eqn:conv-waki-hom2}, we define $\omega_\lambda$ to be the map $(\cW_\lambda \star ({-})) \otimes_{\sH^\bullet_I(\pt)} \bk$.
\end{proof}

Next, we show that the maps $\omega_\lambda$ enjoy a kind of compatibility with composition.

\begin{lem}\label{lem:conv-waki-assoc}
Let $V \in \Rep(\Gv)$ be a representation with a good filtration.  Let $\cF \in \DmixI(\Gr)$ be an object such that both $\cF$ and $\cF \star \Sat(V)$ are $!$-pure of weight $0$.  Let $\sigma \in \bX$, and let $\lambda, \mu \in \bXp$.  Given $f: \ocW_{-\lambda} \to \Sat(V)\{n\}$ and $g: \ocW_{-\mu} \to \cF$, consider the composition
\[
\ocW_{-\lambda-\mu} \xrightarrow{\omega_{-\mu}(f)} \ocW_{-\mu} \star \Sat(V)\{n\} \xrightarrow{g \star \id} \cF \star \Sat(V)\{n\}.
\]
The following diagram commutes:
\[
\xymatrix{
\ocW_{\sigma - \lambda - \mu} \ar[rr]^-{\omega_{\sigma-\mu}(f)}
\ar[drr]_{\omega_\sigma\left((g \star \id)  \circ \omega_{-\mu}(f)\right)\qquad\quad} &&
\ocW_{\sigma - \mu} \star \Sat(V)\{n\} \ar[d]^{\omega_\sigma(g) \star \id} \\
&& \cW_\sigma \star \cF \star \Sat(V)\{n\} }
\]
\end{lem}
\begin{proof}
As we observed in the proof of Proposition~\ref{prop:conv-waki-hom}, the maps
\begin{gather*}
\Hom_{\Dmix_I(\Gr)}(\ocW_{-\lambda}, \Sat(V)\{n\}) \to \Hom_{\DmixI(\Gr)}(\ocW_{-\lambda}, \Sat(V)\{n\}), \\
\Hom_{\Dmix_I(\Gr)}(\ocW_{-\mu}, \cF) \to \Hom_{\DmixI(\Gr)}(\ocW_{-\mu}, \cF)
\end{gather*}
are surjective. Choose maps $\tilde f: \ocW_{-\lambda} \to \Sat(V)\{n\}$ and $\tilde g: \ocW_{-\mu} \to \cF$ in $\Dmix_I(\Gr)$ that lift $f$ and $g$, respectively.  The commutative diagram in Proposition~\ref{prop:conv-waki-hom} says that
\[
\omega_{\sigma - \mu}(f) = \For(\cW_{\sigma - \mu} \star \tilde f)
\qquad\text{and}\qquad
\omega_{\sigma}(g) = \For(\cW_\sigma \star \tilde g).
\]
The following calculation completes the proof:
\begin{multline*}
(\omega_{\sigma}(g) \star \id) \circ \omega_{\sigma - \mu}(f)
= \For ((\cW_\sigma \star \tilde g \star \id) \circ (\cW_{\sigma - \mu} \star \tilde f)) \\
= \For (\cW_\sigma \star ((\tilde g \star \id) \circ (\cW_{-\mu} \star \tilde f))) 
= \omega_\sigma (\For(\tilde g \star \id) \circ \For(\cW_{-\mu} \star \tilde f)) \\
= \omega_\sigma((g \star \id) \circ \omega_{-\mu}(f)). \qedhere
\end{multline*}
\end{proof}

\newcommand{\fakestar}{\mathbin{\text{\normalfont``}\mathord{\star}\text{\normalfont''}}}

At the moment, the closest we can get to~\eqref{eqn:abg-convolve} is the following statement.

\begin{prop}\label{prop:conv-waki-parity}
For any $\lambda \in \bX$, there is a functor
\[
\cW_\lambda \fakestar ({-}): \Parity_{(I)}(\Gr) \to \DmixI(\Gr)
\]
such that the following diagram commutes:
\[
\xymatrix{
\Parity_I(\Gr) \ar[rr]^{\cW_\lambda \star ({-})} \ar[d]_{\For} &&
  \Dmix_I(\Gr) \ar[d]^{\For} \\
\Parity_{(I)}(\Gr) \ar[rr]^{\cW_\lambda \fakestar ({-})} && \DmixI(\Gr)}
\]
\end{prop}
\begin{proof}
It is known that $\For: \Parity_I(\Gr) \to \Parity_{(I)}(\Gr)$ is essentially surjective.  Given $\cF \in \Parity_{(I)}(\Gr)$, choose an object $\tilde \cF \in \Parity_I(\Gr)$ together with an isomorphism $u: \For(\tilde\cF) \simto \cF$.  We define $\cW_\lambda \fakestar \cF$ to be $\For(\cW_\lambda \star \tilde \cF)$.  Suppose now that $\cG$ is another object, for which we have chosen $v: \For(\tilde\cG) \simto \cG$.  Given a morphism $f: \cF \to \cG$, we define $\cW_\lambda \fakestar f$ to be the map $\omega_\lambda(v^{-1} \circ f \circ u)$. It is easy to see that different choices would lead to a canonically isomorphic functor.  The fact that the diagram in the proposition commutes is obvious by construction.
\end{proof} 

\subsection{Subcategories generated by Wakimoto sheaves}

We end this section with a few results about subcategories of $\DmixI(\Gr)$ that can be generated by various collections of Wakimoto and spherical sheaves.  These facts will be used in Section~\ref{sect:main}.

\begin{lem}\label{lem:waki-support-gen}
Let $Z \subset \Gr$ be a closed union of $I$-orbits.  Then $\DmixI(Z)$ is generated as a triangulated category by $\{ \ocW_\mu\{n\} \mid \schub\mu \subset Z,\ n \in \Z \}$.
\end{lem}
\begin{proof}
This is an immediate consequence of Lemma~\ref{lem:waki-support}.
\end{proof}

\begin{lem}\label{lem:waki-po-hom}
Let $\lambda, \mu \in \bX$.
\begin{enumerate}
\item If $\lambda \not\preceq \mu$, then $\gHom^\bullet(\ocW_\mu, \ocW_\lambda) = 0$.\label{it:waki-po-hom-van}
\item We have $\gEnd(\ocW_\lambda) \cong \bk$, and $\gHom^i(\ocW_\lambda,\ocW_\lambda) = 0$ for $i \ne 0$.\label{it:waki-po-end}
\end{enumerate}
\end{lem}
\begin{proof}
For part~\eqref{it:waki-po-hom-van}, the statement involves $\DmixI(\Gr)$, but by Proposition~\ref{prop:eqvt-hom-dg}, it suffices to prove the corresponding vanishing in $\Dmix_I(\Gr)$.  For the remainder of the proof, we work in the latter category.  For any $\nu \in \bX$, applying $\cW_\nu \star ({-})$ gives us an isomorphism $\gHom^\bullet(\ocW_\mu, \ocW_\lambda) \cong \gHom^\bullet(\ocW_{\mu+\nu}, \ocW_{\lambda+\nu})$.  Now choose $\nu$ to be dominant and large enough so that $\mu+\nu$ and $\lambda + \nu$ are both dominant.  By~\eqref{eqn:waki-support} and adjunction, we have
\[
\gHom^\bullet(\ocW_\mu,\ocW_\lambda) \cong \gHom^\bullet(i_{\lambda+\nu}^*\uicos(\mu+\nu),\ubk\{\dim \schub(\lambda+\nu)\}.
\]
This is nonzero if and only if $i_{\lambda+\nu}^*\uicos(\mu+\nu)\neq 0$.  The latter implies that $\schub (\lambda+\nu) \subset \overline{\schub (\mu+\nu)}$.  Since $\mu + \nu$ and $\lambda + \nu$ are both dominant, this holds only when $\lambda \preceq \mu$.

Next, Proposition~\ref{prop:conv-waki-hom} implies that it is enough to prove part~\eqref{it:waki-po-end} in the case where $\lambda = 0$, and in this case, the result is clear.
\end{proof}


Let $\lambda \in \bXp$, and recall that $\ocW_{w_0\lambda} \cong \uistd(w_0\lambda)\{-\delta_{w_0\lambda}\}$.  Suppose now that $\cF$ is an object of $\DmixGo(\Gr)$ such that $(i^\sph_\lambda)^!\cF \cong \ubk\{\dim \Gr_\lambda\}$.  (This applies, for instance, to $\cIstd(\lambda)$ and $\cIcos(\lambda)$.)  Then
\[
i_{w_0\lambda}^!\cF \cong \ubk\{\dim \Gr_\lambda - 2(\dim \Gr_\lambda - \dim \schub(w_0\lambda))\} = \ubk\{\dim \schub(w_0\lambda) - \delta_{w_0\lambda}\}.
\]
By adjunction, we obtain a map $\ocW_{w_0\lambda} \to \cF$.  In particular, we have canonical maps
\begin{equation}\label{eqn:waki-adjunct}
\ocW_{w_0\lambda} \to \cIstd(\lambda)
\qquad\text{and}\qquad
\ocW_{w_0\lambda} \to \cIcos(\lambda).
\end{equation}

\begin{lem}\label{lem:proto-waki-filt}
Let $\lambda \in \bXp$.  Extend the natural adjunction maps $\ocW_{w_0\lambda} \to \cIstd(\lambda)$ and $\ocW_{w_0\lambda} \to \cIcos(\lambda)$ to distinguished triangles
\[
\ocW_{w_0\lambda} \to \cIstd(\lambda) \to \cK \to, \qquad
\ocW_{w_0\lambda} \to \cIcos(\lambda) \to \cK' \to.
\]
Then both $\cK$ and $\cK'$ lie in the full triangulated subcategory of $\DmixI(\Gr)$ generated by the set of objects
\[
\{ \ocW_\mu\{n\} \mid \text{$n \in \Z$, $\mu \succ w_0\lambda$ and $\mu \le \lambda$} \}.
\]
\end{lem}
\begin{proof}
Let $D' \subset \DmixI(\Gr)$ be the category generated by the objects above.  On the other hand, let
\[
D'' = \{ \cF \in \DmixI(\Gr) \mid \text{the support of $\cF$ is contained in $\overline{\Gr_\lambda}$, and $i_{w_0\lambda}^!\cF = 0$} \}.
\]
We will show that $D' = D''$.

We first claim that if $\mu \le \lambda$ and $\mu \succ w_0\lambda$, then $i_{w_0\lambda}^! \ocW_\mu = 0$.  Indeed, by adjunction and~\eqref{eqn:waki-support}, this claim is equivalent to the assertion that $\gHom^\bullet(\ocW_{w_0\lambda}, \ocW_\mu) = 0$.  The latter holds by Lemma~\ref{lem:waki-po-hom}. We have shown that $D' \subset D''$.

Next, note that among the weights${}\le \lambda$, the weight $w_0\lambda$ is the unique minimal weight with respect to $\preceq$.  Thus, if $\schub \mu \subset \overline{\Gr_\lambda} \smallsetminus \Gr_\lambda$, then $\ocW_\mu \in D'$.  More generally, for any $\cF \in \DmixI(\Gr)$ supported on $\overline{\Gr_\lambda} \smallsetminus \Gr_\lambda$, Lemma~\ref{lem:waki-support-gen} implies that $\cF \in D'$.  

We will now show that every object $\cF \in D''$ lies in $D'$ by induction on the number of $I$-orbits in $(\supp \cF) \cap \Gr_\lambda$.  If that intersection is empty, the previous paragraph tells us that $\cF \in D'$.  Otherwise, choose a $\mu \in W \cdot \lambda$ such that $\schub \mu$ is open in the support of $\cF$.  Then there is a distinguished triangle
\[
\cF' \to \cF \to i_{\mu*}i_\mu^*\cF \to.
\]
Note that $\mu \ne w_0\lambda$, because $i_\mu^*\cF \cong i_\mu^!\cF$ is nonzero.  Since $i_{w_0\lambda}^!\cF \cong i_{w_0\lambda}^!i_{\mu*}i_\mu^*\cF = 0$, we find that $i_{w_0\lambda}^!\cF' = 0$ as well.  Thus, $\cF'$ lies in $D''$, and its support meets fewer $I$-orbits in $\Gr_\lambda$ than that of $\cF$, so $\cF' \in D'$.  On the other hand, $i_{\mu*}i_\mu^*\cF$ is a direct sum of various $\uicos(\mu)\{n\}[k]$.  By Lemma~\ref{lem:waki-support}, there is a distinguished triangle
\[
\cG \to \ocW_\mu \to \uicos(\mu)\{-\delta_\mu\} \to.
\]
The same reasoning as above shows that $\cG$ lies in $D''$ and hence, by induction, in $D'$.  Therefore, $i_{\mu*}i_\mu^*\cF$ lies in $D'$, and hence $\cF$ lies in $D'$ as well.  We have now shown that $D' = D''$.

Finally, because $\ocW_{w_0\lambda} \to \cIstd(\lambda)$ and $\ocW_{w_0\lambda} \to \cIcos(\lambda)$ were defined by adjunction, the maps $i_{w_0\lambda}^!\ocW_{w_0\lambda} \to i_{w_0\lambda}^!\cIstd(\lambda)$ and $i_{w_0\lambda}^!\ocW_{w_0\lambda} \to i_{w_0\lambda}^!\cIcos(\lambda)$ are isomorphisms.  It follows that $i_{w_0\lambda}^!\cK = i_{w_0\lambda}^!\cK' = 0$.  In other words, $\cK$ and $\cK'$ lie in $D''$, and hence in $D'$, as desired.
\end{proof}

\begin{lem}\label{lem:sph-without-antidom}
Let $\lambda \in \bXp$. The category $\DmixI(\overline{\Gr_\lambda})$ is generated as a triangulated category by the set of objects
\[
\{ \ocW_\mu\{n\} \mid \text{$n \in \Z$, $\mu \succ w_0\lambda$ and $\mu \leq \lambda$} \} \cup \{ \cIcos(\lambda)\{n\} \mid n \in \Z \}.
\]
\end{lem}
\begin{proof}
Let $D \subset \DmixI(\overline{\Gr_\lambda})$ be the triangulated subcategory generated by the objects indicated above.  The second distinguished triangle in Lemma~\ref{lem:proto-waki-filt} shows that $D$ contains $\ocW_{w_0\lambda}$.  Of course, every weight $\mu \le \lambda$ satisfies $\mu \succeq w_0\lambda$, so we see that $D$ contains all $\ocW_\mu$ with $\mu \le \lambda$.  The lemma follows by Lemma~\ref{lem:waki-support-gen}.
\end{proof}

\begin{prop}\label{prop:dmixigr-gen-x}
The category $\DmixI(\Gr)$ is generated as a triangulated category by the set
\[
\{ \ocW_\lambda\{n\} \mid n\in \Z,\ \lambda \not\preceq 0 \}
\cup
\{ \ocW_\lambda \star \cIcos(\mu)\{n\} \mid n\in \Z,\ \lambda,\mu \in \bXp \}.
\]
\end{prop}
\begin{proof}
Let $D \subset \DmixI(\Gr)$ be the triangulated category generated by the set of objects indicated above.  We will show that all $\ocW_\lambda\{n\}$ belong to $D$.  Of course, we need only consider the case where $\lambda \preceq 0$.  We proceed by downward induction with respect to $\preceq$: given $\lambda \preceq 0$, let us assume that for all $\mu \succ \lambda$, $\ocW_\mu$ is already known to lie in $D$.  (Note that only finitely many such $\mu$ also satisfy $\mu \preceq 0$, so it does make sense to argue by induction here.)  Write $\lambda = \sigma + w_0\nu$, where $\sigma$ and $\nu$ are both dominant.  Lemma~\ref{lem:sph-without-antidom} tells us that $\ocW_{w_0\nu}$ lies in the triangulated category generated by
\[
\{ \ocW_\mu\{n\} \mid n \in \Z,\ \mu \succ w_0\nu \} \cup \{ \cIcos(\nu)\{n\} \mid n \in \Z \}.
\]
It follows that $\ocW_\lambda \cong \cW_\sigma \star \ocW_{w_0\nu}$ lies in the triangulated subcategory generated by 
\begin{equation}\label{eqn:dmixigr-gen-x}
\{ \cW_\sigma \star \ocW_\mu\{n\} \mid n \in \Z,\ \mu \succ w_0\nu \} \cup \{ \cW_\sigma \star \cIcos(\nu)\{n\} \mid n \in \Z \}.
\end{equation}
The objects $\cW_\sigma \star \ocW_\mu \cong \ocW_{\sigma+\mu}$ lie in $D$ by assumption, since $\sigma+\mu \succ \lambda$.  Thus, all objects in~\eqref{eqn:dmixigr-gen-x} lie in $D$, so $\ocW_\lambda$ lies in $D$ as well, as desired.
\end{proof}

We end with a result relating the adjunction map $\epsilon: \ocW_{w_0\lambda} \to \cIstd(\lambda)$ of~\eqref{eqn:waki-adjunct} to convolution of spherical sheaves.

\begin{lem}\label{lem:weyl-mult}
For $\lambda, \mu \in \bXp$, there is a unique map of $\Gv$-representations
\begin{equation}\label{eqn:spult-defn}
\spult_{\lambda,\mu}: \wey(\lambda + \mu) \to \wey(\lambda) \otimes \wey(\mu)
\end{equation}
such that the following diagram commutes:
\[
\xymatrix{
\ocW_{w_0(\lambda+\mu)} \ar[d]_{\wr} \ar[rrr]^{\epsilon} 
  &&& \cIstd(\lambda+\mu) \ar[d]^{\Sat(\spult_{\lambda,\mu})} \\
\cW_{w_0\lambda} \star \ocW_{w_0\mu} \ar[r]^{\id \star \epsilon} &
  \cW_{w_0\lambda} \star \cIstd(\mu) \ar[r]^{\sim} & 
  \ocW_{w_0\lambda} \star \cIstd(\mu) \ar[r]^{\epsilon \star \id} &
  \cIstd(\lambda) \star \cIstd(\mu)}
\]
Each map $\spult_{\lambda,\mu}$ is nonzero.  Moreover, for $\lambda,\mu,\nu \in \bXp$, the two morphisms $\wey(\lambda+\mu+\nu) \to \wey(\lambda) \otimes \wey(\mu) \otimes \wey(\nu)$ coincide.
\end{lem}
\begin{proof}
It is easy to see that $\Hom(\wey(\lambda+\mu), \wey(\lambda) \otimes \wey(\mu))$ and $\Hom(\ocW_{w_0(\lambda+\mu)}, \cIstd(\lambda) \star \cIstd(\mu))$ are both $1$-dimensional, so the existence and uniqueness of $\spult_{\lambda,\mu}$ are clear.  The associativity property can be deduced from the analogous property for Wakimoto sheaves.  It remains only to show that $\spult_{\lambda,\mu}$ is nonzero.

To rephrase this problem, form distinguished triangles
$
\ocW_{w_0\lambda} \to \cIstd(\lambda) \to \cK_\lambda \to
$
and
$
\ocW_{w_0\mu} \to \cIstd(\mu) \to \cK_\mu \to
$
as in Lemma~\ref{lem:proto-waki-filt}.  From these, we can build the octahedral diagram shown in Figure~\ref{fig:weyl-mult}.  In that figure, $\cG$ is a new object; it occurs in a distinguished triangle
\begin{equation}\label{eqn:weyl-mult}
\ocW_{w_0(\lambda+\mu)} \to \cIstd(\lambda) \star \cIstd(\mu) \to \cG \to.
\end{equation}
We want to show that the first morphism in this triangle is not zero. 

For any weight $\nu$, let $D_{\succ\nu}$ be the full triangulated subcategory of $\DmixI(\Gr)$ generated by $\{ \ocW_\sigma\{n\} \mid \sigma \succ \nu,\ n \in \Z \}$.  We define $D_{\succeq\nu}$ similarly.  Lemma~\ref{lem:proto-waki-filt} tells us that $\cK_\mu \in D_{\succ w_0\mu}$, and likewise for $\cK_\lambda$.  Since $\cIstd(\mu) \in D_{\succeq w_0\mu}$, we see that both $\cW_{w_0\lambda} \star \cK_\mu$ and $\cK_\lambda \star \cIstd(\mu)$ lie in $D_{\succ w_0(\lambda+\mu)}$.  Therefore, $\cG$ lies in $D_{\succ w_0(\lambda+\mu)}$.

Lemma~\ref{lem:waki-po-hom} implies that $\ocW_{w_0(\lambda+\mu)}$ does not lie in $D_{\succ w_0(\lambda+\mu)}$, so it cannot be a direct summand of $\cG[-1]$. We deduce that the first morphism in~\eqref{eqn:weyl-mult} is nonzero, as desired.
\end{proof}

\begin{figure}
\[
\xymatrix@=3pt{
&&&&&& \\ \\ \\
&&&& \cW_{w_0\lambda} \star \cK_\mu \ar[rrddd] \ar[rruuu] \\
\\
&&&&&&&& \\
&& \cW_{w_0\lambda} \star \cIstd(\mu) \ar[rruuu] \ar[rrd] &&&&
   \cG \ar[rrddd] \ar[rru] \\
&&&& \cIstd(\lambda) \star \cIstd(\mu) \ar[rru] \ar[rrrrdd] \\
\\
\ocW_{w_0(\lambda+\mu)} \ar[rruuu] \ar[rrrruu] &&&&&&&&
   \cK_\lambda \star \cIstd(\mu) \ar[rrddd] \ar[rrd] \\
&&&&&&&&&& \\
\\
&&&&&&&&&&}
\]
\caption{Octahedral diagram for Lemma~\ref{lem:weyl-mult}}\label{fig:weyl-mult}
\end{figure}
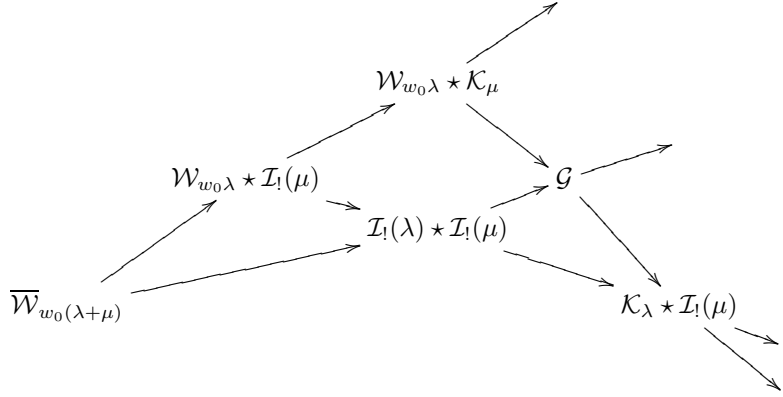

\section{Multihomogeneous coordinate rings and Ext-algebras}
\label{sect:extalg}

\subsection{The multihomogeneous coordinate ring of the flag variety}

Consider the duals of the maps introduced in Lemma~\ref{lem:weyl-mult}:
\begin{equation}\label{eqn:spults-defn}
\spult_{\lambda,\mu}^*: \cow(\lambda) \otimes \cow(\mu) \to \cow(\lambda+\mu).
\end{equation}
That lemma implies that these maps satisfy a certain associativity property, so we can use them make $\bigoplus_{\lambda \in \bXp} \cow(\lambda)$ into a ring.  We introduce the notation
\[
\bGamma[\cBv] := \bigoplus_{\lambda \in \bXp} \cow(\lambda),
\]
and we regard it as a $\Gv$-equivariant $\bX$-graded ring.  Let $\bGamma[\cBv]\lmod$ denote the category of finitely generated $\Gv$-equivariant $\bX$-graded $\bGamma[\cBv]$-modules.  A module $M = \bigoplus_{\lambda \in \bX} M_\lambda$ in $\Gamma[\cBv]\lmod$ is said to be \emph{thin} if there is some $\lambda \in \bX$ such that $M_\mu = 0$ for all $\mu \in \lambda + \bXp$. 

This notation reflects the fact that this ring can be thought of as a multihomogeneous coordinate ring for $\cBv$.  To make this precise, consider the line bundle $\cO_{\cBv}(\lambda)$ on $\cBv$.  We have a canonical identification $\Gamma(\cBv,\cO_\cBv(\lambda)) \cong \cow(\lambda)$.  By adjunction and the projection formula, one sees that there is a canonical bijection $\Hom(\cO_\cBv(\lambda) \otimes \cO_\cBv(\mu), \cO_\cBv(\lambda+\mu)) \cong \Hom(\cow(\lambda) \otimes \cow(\mu), \cow(\lambda+\mu))$.  Let
\begin{equation}\label{eqn:stult-defn}
\stult_{\lambda,\mu}: \cO_\cBv(\lambda) \otimes \cO_\cBv(\mu) \simto \cO_\cBv(\lambda+\mu)
\end{equation}
be the map corresponding to $\spult_{\lambda,\mu}^*$ under this bijection.  Again, these maps enjoy an associativity property like that in Lemma~\ref{lem:weyl-mult}.

Let us assume temporarily that $\Gv$ is semisimple and simply connected, and let $\varpi_1, \ldots, \varpi_r$ be the fundamental weights of $\Gv$.  From~\eqref{eqn:stult-defn}, we obtain for each $\lambda \in \bXp$ a canonical isomorphism
\[
\cO_\cBv(\lambda) \cong \cO_\cBv(\varpi_1)^{\otimes a_1} \otimes \cdots \otimes \cO_\cBv(\varpi_r)^{\otimes a_r}
\qquad\text{if $\lambda = a_1\varpi_1 + \cdots + a_r\varpi_r$,}
\]
and hence a canonical isomorphism of rings
\[
\bGamma[\cBv] \cong \bigoplus_{(a_1, \ldots, a_r) \in \mathbb{N}^r}
\Gamma(\cBv, \cO_\cBv(\varpi_1)^{\otimes a_1} \otimes \cdots \otimes \cO_\cBv(\varpi_r)^{\otimes a_r}).
\]
The right-hand side agrees with the multihomogeneous coordinate ring of $\cBv$ as discussed in, say,~\cite[\S10]{gls} or~\cite[p.~123]{lg}.  A straightforward generalization of the $\Proj$-construction (see, e.g., the discussion following~\cite[Proposition~4.8]{mumford}) recovers the variety $\cBv$ from this ring, and provides an exact functor
\[
\scF_0: \bGamma[\cBv]\lmod \to \Coh(\cBv),
\]
where $\Coh(\cBv)$ is the category of $\Gv$-equivariant coherent sheaves on $\cBv$.  (In a slight abuse of notation, we will also write $\scF_0$ for the operation that takes possibly infinitely-generated $\bGamma[\cBv]$-modules to quasicoherent sheaves on $\cBv$.)  Moreover, this functor induces an equivalence of categories
\[
\bGamma[\cBv]\lmod/(\text{Serre subcategory of thin modules}) \simto \Coh(\cBv).
\]

In fact, the functor $\scF_0$ and the above equivalence are available for arbitrary reductive $\Gv$ satisfying~\eqref{eqn:reasonable}.  The flag variety of $\Gv$ can be identified with that of its derived subgroup $(\Gv)'$, and routine reduction arguments let us build $\scF_0$ for $\Gv$ in terms of that for $(\Gv)'$.

We now describe another way to construct the ring $\bGamma[\cBv]$ in terms of the geometric Satake equivalence.  Consider the $\bX$-graded $\Gv$-representation
\begin{equation}\label{eqn:homWR0}
\bigoplus_{\lambda \in \bXp} \Hom_{\DmixI(\Gr)}(\cIstd(-w_0\lambda), \cR).
\end{equation}
We make this into a ring as follows: given $g \in \Hom(\cIstd(-w_0\lambda), \cR)$ and $f \in \Hom(\cIstd(-w_0\mu),\cR)$, let $gf \in \Hom(\cIstd(-w_0(\lambda+\mu)),\cR)$ be the composition
\[
\cIstd(-w_0(\lambda+\mu)) \xrightarrow{\Sat(\spult_{-w_0\lambda,-w_0\mu})} \cIstd(-w_0\lambda) \star \cIstd(-w_0\mu) \xrightarrow{g \star f} \cR \star \cR \xrightarrow{\Sat(\mult)} \cR.
\]
Here, we have used the fact that $f$ is a morphism in $\Perv_{(\Go)}(\Gr) = \Perv_\Go(\Gr)$, so that it makes sense to form the convolution product $g \star f$.  For later reference, we rewrite this product in a slightly different form:
\begin{multline}\label{eqn:homWR0-ring}
\cIstd(-w_0(\lambda+\mu)) \xrightarrow{\Sat(\spult_{-w_0\lambda,-w_0\mu})} 
\cIstd(-w_0\lambda) \star \cIstd(-w_0\mu) \xrightarrow{\id \star f} \cIstd(-w_0\lambda) \star \cR \\
\xrightarrow{g \star \id} \cR \star \cR \xrightarrow{\Sat(\mult)} \cR.
\end{multline}

\begin{prop}\label{prop:coh-cbv}
There is an isomorphism of $\Gv$-equivariant $\bX$-graded rings
\[
\bGamma[\cBv] \cong \bigoplus_{\lambda \in \bXp} \Hom(\cIstd(-w_0\lambda), \cR).
\]
\end{prop}
\begin{proof}
The maps below give an isomorphism of $\bX$-graded $\Gv$-representations.  It is easily checked that they also constitute a ring isomorphism, as desired.
\begin{multline*}
\Hom(\cIstd(-w_0\lambda), \cR) \xrightarrow[\sim]{\Sat} \Hom(\wey(-w_0\lambda), \bk[\Gv]) \cong
\Hom(\bk, \cow(\lambda) \otimes \bk[\Gv]) \\
\cong (\cow(\lambda) \otimes \bk[\Gv])^{\Gv \times 1} \xrightarrow[\sim]{\text{Lemma~\ref{lem:gv-invt}}} \cow(\lambda). \qedhere
\end{multline*}
\end{proof}

\subsection{The multihomogeneous coordinate ring of the Springer resolution}

We will now upgrade these considerations from $\cBv$ to $\tcN$.  The isomorphisms in~\eqref{eqn:stult-defn} determine a corresponding collection of isomorphisms
\[
\tilde\stult_{\lambda,\mu}: \cO_\tcN(\lambda) \otimes \cO_\tcN(\mu) \to \cO_\tcN(\lambda+\mu).
\]
These, in turn, give rise to a collection of maps
\begin{equation}\label{eqn:tspults-defn}
\tilde\spult_{\lambda,\mu}^*: \Gamma(\tcN, \cO_\tcN(\lambda)) \otimes \Gamma(\tcN, \cO_\tcN(\mu)) \to \Gamma(\tcN, \cO_\tcN(\lambda+\mu))
\end{equation}
that we then use to make the following space into a ring:
\[
\bGamma[\tcN] := \bigoplus_{\lambda \in \bXp} \Gamma(\tcN, \cO_\tcN(\lambda)).
\]
This ring carries a $(\Z \times \bX)$-grading.  Its degree-$(\Z \times \{0\})$ subring (i.e., the subring spanned by homogeneous elements whose degrees lie in $\Z \times \{0\} \subset \Z \times \bX$) is the $\Z$-graded ring $\Gamma(\tcN, \cO_\tcN) \cong \bk[\cN]$.  On the other hand, Lemma~\ref{lem:pcpcos-deg0} gives us an injective homomorphism
\begin{equation}\label{eqn:coh-cbv-tcn}
\bGamma[\cBv] \hookrightarrow \bGamma[\tcN]
\end{equation}
that identifies the former with the degree-$(\{0\} \times \bX)$ subring of the latter.  (To be precise, Lemma~\ref{lem:pcpcos-deg0} just gives us an injective map of $\Gv$-representations.  Because both~\eqref{eqn:stult-defn} and~\eqref{eqn:tspults-defn} are induced by~\eqref{eqn:spults-defn}, this map is actually a ring homomorphism.)

Regard $\bGamma[\tcN]$ as a $\bGamma[\cBv]$-algebra via~\eqref{eqn:coh-cbv-tcn}.  Applying $\scF_0$, we obtain a $\Z$-graded sheaf of algebras $\scS$ on $\cBv$.  This sheaf of algebras can be identified with $p_*\cO_\tcN$, where $p: \tcN \to \cBv$ is the projection map.  In other words, we have $\tcN = \bSpec \scS$, where $\bSpec$ is the relative version of the $\Spec$ construction.

Let $\bGamma[\tcN]\lmod$ denote the category of finitely generated $\Gv$-equivariant $(\Z \times \bX)$-graded $\bGamma[\tcN]$-modules.  Given a module $M \in \bGamma[\tcN]\lmod$, we can regard it as a $\bGamma[\cBv]$-module via~\eqref{eqn:coh-cbv-tcn}, and then form the sheaf $\scF_0(M)$.  This is a quasicoherent sheaf on $\cBv$ that is also a $\Z$-graded sheaf of $\scS$-modules.  The $\bSpec$ construction then associates to $\scF_0(M)$ a $(\Gv \times \Gm)$-equivariant coherent sheaf on $\tcN$.  In this way, we obtain a functor
\[
\scF: \bGamma[\tcN]\lmod \to \Coh(\tcN).
\]
As above, a module $M = \bigoplus_{\lambda \in \bX} M_\lambda$ is called \emph{thin} if there is some $\lambda \in \bX$ such that $M_\mu = 0$ for all $\mu \in \lambda + \bXp$.  (The $\Z$-grading is irrelevant to this condition.)  The functor $\scF$ induces an equivalence of categories
\[
\bGamma[\tcN]\lmod/(\text{Serre subcategory of thin modules}) \simto \Coh(\tcN).
\]

\subsection{An Ext-algebra of Wakimoto sheaves}

Building on the construction of Section~\ref{ss:regular-perv}, we now make 
\[
\bigoplus_{\lambda \in \bXp} \gHom(\ocW_{-\lambda}, \cR)
\]
into a $\Gv$-equivariant $(\mathbb{Z} \times \bX)$-graded $\gHom(\sky,\cR)$-algebra, as follows: given $f \in \Hom(\ocW_{-\lambda}, \cR\{n\})$ and $g \in \Hom(\ocW_{-\mu}, \cR\{m\})$, we define $gf \in \Hom(\ocW_{-\lambda-\mu}, \cR\{n+m\})$ to be the composition
\begin{multline}\label{eqn:homWR-ring}
\ocW_{-\lambda-\mu} \xrightarrow{\omega_{-\mu}(f)} \cW_{-\mu} \star \cR\{n\} \simto \ocW_{-\mu} \star \cR\{n\} \\
{} \xrightarrow{g \star \id} \cR\{m\} \star \cR\{n\} \xrightarrow{\Sat(\mult)\{m+n\}} \cR\{m+n\}.
\end{multline}
A short calculation with Lemma~\ref{lem:conv-waki-assoc} shows that this operation is associative, so we do indeed get a ring.
The main result of this section is the following.

\begin{thm}\label{thm:coh-tcn}
There is an isomorphism of $\Gv$-equivariant $\mathbb{Z} \times \bX$-graded rings
\[
\bGamma[\tcN] \cong \bigoplus_{\lambda \in \bXp} \gHom(\ocW_{-\lambda}, \cR).
\]
\end{thm}
\begin{proof}
Using~\eqref{eqn:waki-support}, Proposition~\ref{prop:std-regular}, and~\eqref{eqn:dom-line-sec}, we have the following chain of isomorphisms in $\CohN$:
\[
\gHom(\ocW_{-\lambda}, \cR) \cong \gHom(\uistd(-\lambda)\{-\delta_\lambda\}, \cR) \cong \pcpcos(\lambda)\la \delta_\lambda\ra \cong \Gamma(\tcN, \cO_\tcN(\lambda)).
\]
Thus, our two rings are at least isomorphic as $(\Z \times \bX)$-graded $\Gv$-equivariant $\bk[\cN]$-modules.  Recall from Lemmas~\ref{lem:pcpcos-deg0} and~\ref{lem:cow-pcpcos-surj} that $\Gamma(\tcN, \cO_\tcN(\lambda))$ is generated as a $\bk[\cN]$-module by its degree-$0$ graded component.  Therefore, the same holds for $\gHom(\ocW_{-\lambda}, \cR)$. To prove that the two rings in the statement of the theorem are isomorphic, then, it suffices to show that their degree-$(\{0\} \times \bX)$ subrings are isomorphic.

We first study the right-hand side.  Recall that we have an adjunction map $\epsilon: \ocW_{-\lambda} \cong \uistd(-\lambda)\{-\delta_{-\lambda}\} \to \cIstd(-w_0\lambda)$. This gives rise to a map
\begin{equation}\label{eqn:coh-tcn-waki0-single}
\Hom(\cIstd(-w_0\lambda), \cR) \to \Hom(\ocW_{-\lambda}, \cR).
\end{equation}
We claim that this map is an isomorphism. Note first that the truncation map $\cJstd(-w_0\lambda) \to \pH^0(\cJstd(-w_0\lambda)) = \cIstd(-w_0\lambda)$ induces an isomorphism
\[
\Hom(\cIstd(-w_0\lambda), \cR) \simto \Hom(\cJstd(-w_0\lambda), \cR),
\]
since $\cR$ is perverse. The claim then follows from Corollary~\ref{cor:costalk-compare}.

From the preceding paragraph, we obtain an injective map of $\Gv$-modules
\begin{equation}\label{eqn:coh-tcn-waki0}
\bigoplus_{\lambda \in \bXp} \Hom(\cIstd(-w_0\lambda), \cR)
\hookrightarrow
\bigoplus_{\lambda \in \bXp} \gHom(\ocW_{-\lambda}, \cR)
\end{equation}
that identifies the former with the degree-$(\{0\} \times \bX)$ subspace of the latter.  

We will show that this is also a ring homomorphism. Let $g \in \Hom(\cIstd(-w_0\lambda),\cR)$ and $f \in \Hom(\cIstd(-w_0\mu), \cR)$.  Let $\tilde g \in \Hom(\ocW_{-\lambda}, \cR)$ and $\tilde f \in \Hom(\ocW_{-\mu}, \cR)$ be the maps corresponding to $g$ and $f$ via~\eqref{eqn:coh-tcn-waki0-single}.  (Thus, $\tilde g = g \circ \epsilon$ and $\tilde f = f \circ \epsilon$.) 

Recall that $f$ can be regarded as a morphism in $\Perv_\Go(\Gr) \subset \Dmix_\Go(\Gr)$.  (Indeed, this observation is essential to the definition of the ring structure in~\eqref{eqn:homWR0}.)  We can then forget from the $\Go$-equivariant derived category to the $I$-equivariant derived category.  Of course, the adjunction map $\epsilon: \ocW_{-\mu} \to \cIstd(-w_0\mu)$ can also naturally be lifted to $\Dmix_I(\Gr)$, so we may regard $\tilde f$ as a morphism in $\Dmix_I(\Gr)$.  In particular, it makes sense to form the convolution product $\id \star \tilde f: \cW_{-\lambda} \star \ocW_{-\mu} \to \cW_{-\lambda} \star \cR$.  

That observation is needed for a portion of the large diagram in Figure~\ref{fig:coh-tcn}, which compares the products on either side of~\eqref{eqn:coh-tcn-waki0}.  The large square labeled ($*$) is the commutative diagram from Lemma~\ref{lem:weyl-mult}.  Each of the remaining small squares obviously commutes.

Thus, the whole of Figure~\ref{fig:coh-tcn} commutes, and hence~\eqref{eqn:coh-tcn-waki0} is a ring homomorphism.  From~\eqref{eqn:coh-tcn-waki0}, \eqref{eqn:coh-cbv-tcn}, and Proposition~\ref{prop:coh-cbv}, we obtain an isomorphism of the degree-$(\{0\} \times \bX)$ subrings of $\bGamma[\tcN]$ and $\bigoplus_{\lambda \in \bXp} \gHom(\ocW_{-\lambda},\cR)$, as desired.
\end{proof}

\begin{figure}
\[
\xymatrix{
\ocW_{-\lambda-\mu} \ar[ddd]_{\epsilon} \ar[rr]_-{\omega_{-\lambda}(\id)} \ar@{}[dddrr]|{\txt{($*$)}}
    \ar@/^4ex/[rrr]^-{\omega_{-\lambda}(\tilde f)} & &
  \cW_{-\lambda} \star \ocW_{-\mu} \ar[d]_{\id \star \epsilon} \ar[r]_-{\id \star \tilde f} &
  \cW_{-\lambda} \star \cR \ar@{=}[d]\\
& & \cW_{-\lambda} \star \cIstd(-w_0\mu) \ar[d] \ar[r]^-{\id \star f} &
  \cW_{-\lambda} \star \cR \ar[d] \\
& & \ocW_{-\lambda} \star \cIstd(-w_0\mu) \ar[d]_{\epsilon \star \id} \ar[r]^-{\id \star f} &
  \ocW_{-\lambda} \star \cR \ar[d]_{\epsilon \star \id} \ar@/^9ex/[dd]^{\tilde g \star \id}  \\
\cIstd(-w_0(\lambda+\mu)) \ar[rr]_{\hspace{-.5cm}\Sat(\spult_{-w_0\lambda,-w_0\mu})} &\hspace{.6cm}&
  \cIstd(-w_0\lambda) \star \cIstd(-w_0\mu) \ar[r]^-{\id \star f} \ar@/_3ex/[dr]_-{g \star f}&
  \cIstd(-w_0\lambda) \star \cR \ar[d]_{g\star\id} \\
 & & &\cR\star\cR }
\]
\caption{Comparing ring structures in Theorem~\ref{thm:coh-tcn}}\label{fig:coh-tcn}
\end{figure}

\section{The main result}
\label{sect:main}

\newcommand{\naive}{{\mathrm{naive}}}

We are now ready to prove the following theorem, which is the main result of the paper.  Its proof will occupy the entire section.

\begin{thm}\label{thm:main}
There is an equivalence of triangulated categories
\[
P: \DmixI(\Gr) \simto \Db\CohGm(\tcN)
\]
satisfying $P(\cF\{ 1\}) \cong P(\cF)\la 1\ra$ and $P(\ocW_\lambda) \cong \cO_\tcN(\lambda)$.  Moreover, this equivalence is compatible with the geometric Satake equivalence:  for $\cF \in \DmixI(\Gr)$ and $V \in \Rep(\Gv)$, there is a natural isomorphism $P(\cF \star \Sat(V)) \cong P(\cF) \otimes V$.
\end{thm}

We begin by constructing the functor $P$.  As a first step, given $\cF \in \DmixI(\Gr)$, form the $(\Z \times \bX)$-graded vector space
\[
Q_\naive(\cF) := \bigoplus_{\lambda \in \bXp} \gHom(\ocW_{-\lambda}, \cF \star \cR).
\]
We make this into a right module over $\bigoplus_{\lambda \in \bXp} \gHom(\ocW_{-\lambda}, \cR)$ by a formula similar to~\eqref{eqn:homWR-ring}: given $f \in \Hom(\ocW_{-\lambda}, \cR\{n\})$ and $m \in \Hom(\ocW_{-\mu}, \cF \star \cR\{m\})$, we define $mf \in \Hom(\ocW_{-\lambda-\mu}, \cF \star \cR\{n+m\})$ to be the composition
\begin{multline*}
\ocW_{-\lambda-\mu} \xrightarrow{\omega_{-\mu}(f)} \cW_{-\mu} \star \cR\{n\} \simto \ocW_{-\mu} \star \cR\{n\} \\
{} \xrightarrow{m \star \id} \cF \star \cR\{m\} \star \cR\{n\} \xrightarrow{\id \star \Sat(\mult)\{m+n\}} \cF \star \cR\{m+n\}.
\end{multline*}
Using the isomorphism of Theorem~\ref{thm:coh-tcn}, we henceforth regard $Q_\naive$ as a functor
$Q_\naive: \DmixI(\Gr) \to \bGamma[\tcN]\lmod$.
We also let $P_\naive := \scF \circ Q_\naive: \DmixI(\Gr) \to \Coh(\tcN)$, and then we put
\[
P^0 = P_\naive|_{\Parity_{(I)}(\Gr)}: \Parity_{(I)}(\Gr) \to \Coh(\tcN).
\]
Finally, we define $P$ to be the composition
\[
\DmixI(\Gr) = \Kb\Parity_{(I)}(\Gr) \xrightarrow{\Kb(P^0)} \Kb\Coh(\tcN) \to \Db\Coh(\tcN).
\]
We begin with the last assertion in the theorem.

\begin{prop}\label{prop:monoidal}
For $\cF \in \DmixI(\Gr)$ and $V \in \Rep(\Gv)$, there is a natural isomorphism $P(\cF \star \Sat(V)) \cong P(\cF) \otimes V$.
\end{prop}
\begin{proof}
Observe first that by applying $\Sat$ to the isomorphism of Lemma~\ref{lem:gvgv-swap}, we obtain a natural isomorphism of $\Gv$-equivariant ind-perverse sheaves
\begin{equation}\label{eqn:sat-gvgv-swap}
\Sat(V) \star \cR \cong \cR \otimes V.
\end{equation}
(Here, $\cR \otimes V$ is isomorphic as an ind-perverse sheaf---but not as a $\Gv$-equivariant ind-perverse sheaf---to $\bigoplus^{\dim V} \cR$.)  From the definitions of the convolution product and the functor $P$, one sees that it is enough to prove the following statement: \emph{For $\cF \in \Parity_{(I)}(\Gr)$ and $V$ a tilting $\Gv$-module, there is a natural isomorphism $P^0(\cF \star \Sat(V)) \cong P^0(\cF) \otimes V$.}  For the latter claim, using~\eqref{eqn:sat-gvgv-swap}, we find that
\begin{multline*}
P^0(\cF \star \Sat(V)) = \scF\left(\bigoplus \gHom(\ocW_{-\lambda}, \cF \star \Sat(V) \star \cR)\right) \\
= \scF\left(\bigoplus \gHom(\ocW_{-\lambda}, \cF \star \cR) \otimes V\right) \cong
P^0(\cF) \otimes V,
\end{multline*}
as desired.
\end{proof}

The next several statements are somewhat technical lemmas aimed at making it possible to compute some values of $P$.

\begin{lem}\label{lem:par-cow-adverse}
Let $\cF \in \Parity_{(I)}(\Gr)$, and let $V \in \Rep(\Gv)$ have a good filtration.  Then $\cF \star \Sat(V)$ is $!$-pure of weight~$0$.
\end{lem}
\begin{proof}
Every indecomposable parity sheaf on $\Gr$ occurs as a direct summand of the direct image along $\varpi: \Fl \to \Gr$ of some parity sheaf on $\Fl$.  Thus, without loss of generality, we may assume that $\cF = \varpi_*\tilde \cF$ for some $\tilde \cF \in \Parity_{(I)}(\Fl)$.  By Lemma~\ref{lem:conv-fl-sph}, we have $\varpi_*\tilde\cF \star \Sat(V) \cong \tilde\cF \star \Sat(V)$.  Via Lemma~\ref{lem:adv-std-filt}, it is enough to show that $\cF \star \Sat(V)$ is an adverse sheaf with a costandard filtration.  

Let us first show that $\varpi^*(\cF \star \Sat(V)) \cong \tilde\cF \star \varpi^*\Sat(V)$ is adverse.  (Here, we have used Lemma~\ref{lem:conv-flgr}.) By Lemmas~\ref{lem:adv-std-filt} and~\ref{lem:adv-smooth-filt}, $\varpi^*\Sat(V)$ is an adverse sheaf with a costandard filtration.  On the other hand, the parity sheaf $\tilde \cF$ is a tilting object in $\Adv_{(I)}(\Fl)$; in particular, it has a standard filtration.  Proposition~\ref{prop:adv-conv} then implies that $\tilde \cF \star \varpi^*\Sat(V)$ is adverse.

Since $\varpi^*$ is adverse-exact and kills no nonzero adverse sheaf (see Lemma~\ref{lem:adv-smooth}), it follows that $\cF \star \Sat(V)$ is adverse.  To show that it has a costandard filtration, we must check that 
\[
\Ext^1(\uistd(\mu)\{n\}, \cF \star \Sat(V)) = 0 
\]
for all $\mu \in \bX$ and all $n \in \Z$.  Since $\Sat(V)$ has weights${}\ge 0$ (see~\cite[Lemma~3.5]{arc:f3}) and $\cF$ is parity, the object $(\cF \star \Sat(V))[1]$ has weights${}\ge 1$.  On the other hand, $\uistd(\mu)\{n\}$ has weights${}\le 0$, so the $\Ext^1$-group above vanishes by~\eqref{eqn:wt-hom-van}.
\end{proof}

\begin{lem}\label{lem:Qnaive-parity}
Let $\lambda, \mu \in \bX$ be such that $\lambda+\mu \in \bXp$. Let $\cF \in \Parity_{(I)}(\Gr)$, and let $V \in \Rep(\Gv)$ have a good filtration.  Then
\[
\gHom^i(\ocW_{-\mu}, \cW_\lambda \star \cF \star \Sat(V)) = 0 \qquad\text{for all $i \ne 0$.}
\]
\end{lem}
\begin{proof}
Lemma~\ref{lem:par-cow-adverse} tells us $\cF \star \Sat(V)$ is $!$-pure of weight $0$, so we can invoke Proposition~\ref{prop:conv-waki-hom} to obtain an isomorphism
\[
\gHom^i(\ocW_{-\mu}, \cW_\lambda \star \cF \star \Sat(V)) \cong
\gHom^i(\ocW_{-\lambda-\mu}, \cF \star \Sat(V)).
\]
By Lemma~\ref{lem:par-cow-adverse} and Lemma~\ref{lem:adv-std-filt}, $\cF \star \Sat(V)$ is an adverse sheaf with a costandard filtration.  Since $\ocW_{-\lambda-\sigma} \cong \uistd(-\lambda-\sigma)\{-\delta_{-\lambda-\sigma}\}$ is a standard adverse sheaf, we have $\Ext^i(\ocW_{-\lambda-\sigma}, \cF \star \Sat(V)\{n\}) = 0$ for all $n \in \Z$ and all $i \ne 0$. By the equivalence at the end of Proposition~\ref{prop:adverse-t-struc}, these $\Ext$-groups can be identified with $\Hom$-groups in $\DmixI(\Gr)$, and the lemma follows.
\end{proof}

The next statement involves the functor introduced in Proposition~\ref{prop:conv-waki-parity}.
  
\begin{lem}\label{lem:Pnaive-parity}
Let $\lambda \in \bX$, and let $\cF \in \Parity_{(I)}(\Gr)$.  For all $i \ne 0$, we have $P_\naive(\cW_\lambda \fakestar \cF[i]) = 0$.  
\end{lem}
\begin{proof}
In view of Proposition~\ref{prop:conv-waki-parity}, we may as well instead take $\cF \in \Parity_I(\Gr)$, and work with $\For(\cW_\lambda \star \cF)$. Choose a weight $\nu \in \bXp$ such that $\lambda + \nu \in \bXp$.  Since $\bk[\Gv]$ is an inductive limit of finite-dimensional $\Gv$-representations with good filtrations, Lemma~\ref{lem:Qnaive-parity} implies that $\gHom(\ocW_{-\sigma}, \cW_\lambda \star \cF[i] \star \cR) = 0$ for all $i \ne 0$ and all $\sigma \in \nu+\bXp$.  This means that $Q_\naive(\cW_\lambda \star \cF[i])$ is thin, so $P_\naive(\cW_\lambda \star \cF[i]) = 0$.
\end{proof}

\begin{lem}
For any $\cF \in \DmixI(\Gr)$ and any $i \in \Z$, there is a natural isomorphism $H^i(P(\cF)) \cong P_\naive(\cF[i])$.
\end{lem}
\begin{proof}
It certainly suffices to prove this for $i = 0$.  Let $\DmixI(\Gr)_{\ge 0}$ and $\DmixI(\Gr)_{\le 0}$ denote the full subcategories of objects with weights${}\ge 0$ and${}\le 0$, respectively.  We proceed in several steps.

{\it Step 1. For $\cF \in \DmixI(\Gr)_{\ge 0}$, there is a natural transformation $H^0(P(\cF)) \to P_\naive(\cF)$.}  Since $\cF$ has weights${}\ge 0$, it can be written as a chain complex $\cE^\bullet$ in $\Kb\Parity_{(I)}(\Gr)$ with $\cE^i = 0$ for $i > 0$.  Consider the obvious map $\cE^0 \to \cF$ in $\DmixI(\Gr)$.  The composition $\cE^{-1} \to \cE^0 \to \cF$ vanishes.  Applying $P_\naive$, we obtain a sequence of maps
\[
P^0(\cE^{-1}) \to P^0(\cE^0) \to P_\naive(\cF)
\]
whose composition vanishes.  Therefore, the map $P^0(\cE^0) \to P_\naive(\cF)$ determines a map $P(\cF) \to P_\naive(\cF)$ in $\Db\Coh(\tcN)$.  Taking cohomology, we obtained the desired natural transformation $H^0(P(\cF)) \to P_\naive(\cF)$.

{\it Step 2. For $\cF \in \DmixI(\Gr)_{\le 0}$, there is a natural transformation $P_\naive(\cF) \to H^0(P(\cF))$.} Similar to Step~1.

{\it Step 3. Let $\cF_1 \in \DmixI(\Gr)_{\ge 0}$ and $\cF_2 \in \DmixI(\Gr)_{\le 0}$.  For any morphism $f: \cF_1 \to \cF_2$, the following diagram commutes:}
\begin{equation}\label{eqn:Pnaive-transf}
\vcenter{\xymatrix@R=10pt{
H^0(P(\cF_1)) \ar[rrr]^{H^0(P(f))} \ar[dr] &&& H^0(P(\cF_2))  \\
& P_\naive(\cF_1) \ar[r]^{P_\naive(f)} & P_\naive(\cF_2) \ar[ur]}}
\end{equation}
Write the objects as complexes $\cF_1 = \cE_1^\bullet$ and $\cF_2 = \cE_2^\bullet$ with $\cE_1^i = 0$ for $i > 0$ and $\cE_2^i = 0$ for $i < 0$.  The morphism $f: \cF_1 \to \cF_2$ corresponds to some map of chain complexes $f^\bullet: \cE_1^\bullet \to \cE_2^\bullet$ where, of course, only $f^0$ can be nonzero.  Because the diagram below commutes, we see from the construction of the natural transformations in Steps~1 and~2 that~\eqref{eqn:Pnaive-transf} commutes.
\[
\xymatrix@R=10pt{
\cE_1^{-1} \ar[r] &
  \cE_1^0 \ar[dr] \ar[rrr]^{f^0} &&& \cE_2^0 \ar[r] & \cE_2^1 \\
&& \cF_1 \ar[r]^{f} & \cF_2 \ar[ur]}
\]

{\it Step 4. The natural transformations of Steps~1 and~2 are isomorphisms.}  We will prove this for Step~1; the other case is similar.  Suppose $\cF$ has weights${}\ge 0$ and${}\le n$.  We proceed by induction on $n$.  When $n = 0$, $\cF$ is pure of weight of $0$, and it is clear from the definitions that $P(\cF) \to P_\naive(\cF)$ is an isomorphism.  Otherwise, write $\cF$ as a complex $\cE^\bullet$ with $\cE^i = 0$ for $i > 0$, and let $\cF'$ be the cone of $\cE^0 \to \cF$.  Then $\cF'$ has weights${}\ge 1$ and${}\le n$.  The distinguished triangle $\cE^0 \to \cF \to \cF' \to$ gives rise to a commutative diagram
\[
\xymatrix{
H^0(P(\cF'[-1])) \ar[r] \ar[d]_{\wr} &
  H^0(P(\cE^0)) \ar[r] \ar[d]_{\wr} &
  H^0(P(\cF)) \ar[r] \ar[d] &
  H^0(P(\cF')) \ar[d] \\
P_\naive(\cF'[-1]) \ar[r] &
  P_\naive(\cE^0) \ar[r] &
  P_\naive(\cF) \ar[r] &
  P_\naive(\cF')}
\]
The first two vertical arrows are isomorphisms by induction.  In the last column, we clearly have $H^0(P(\cF')) = 0$, while Lemma~\ref{lem:Pnaive-parity} implies that $P_\naive(\cF') = 0$.  We conclude that the third vertical arrow is an isomorphism, as desired.

{\it Step 5. The general case.}  Given $\cF \in \DmixI(\Gr)$, say $\cF = \cE^\bullet \in \Kb\Parity_{(I)}(\Gr)$, let $\cF_1$ be the complex obtained by omitting the $\cE^i$ with $i > 0$, and let $\cF_2[-1]$ be the complex obtained by omitting the $\cE^i$ with $i \le 0$.  Thus, there is a distinguished triangle $\cF_2[-1] \to \cF \to \cF_1 \to$.  Note that $\cF_1$ has weights${}\ge 0$ and $\cF_2$ has weights${}\le 0$.  Consider the following diagram, in which the rows are long exact sequences:
\[
\xymatrix{
H^{-1}(P(\cF_2)) \ar[r] &
  H^0(P(\cF)) \ar[r]  &
  H^0(P(\cF_1)) \ar[r] \ar[d]_{\wr} &
  H^0(P(\cF_2)) \\
P_\naive(\cF_2[-1]) \ar[r] &
  P_\naive(\cF) \ar[r] &
  P_\naive(\cF_1) \ar[r] \ar[r] &
  P_\naive(\cF_2) \ar[u]_{\wr} }
\]
We clearly have $H^{-1}(P(\cF_2)) = 0$, while Lemma~\ref{lem:Pnaive-parity} implies that $P_\naive(\cF_2[-1]) = 0$.  It follows that there is a unique isomorphism $H^0(P(\cF)) \simto P_\naive(\cF)$ that would make the diagram commute.  It is a routine exercise in homological algebra that this morphism is independent of the choice of $\cF_1$ and $\cF_2$ and natural in $\cF$.
\end{proof}

\begin{cor}\label{cor:Pnaive-coh}
For $\cF \in \DmixI(\Gr)$, we have that $P(\cF) \in \Coh(\tcN)$ if and only if $P_\naive(\cF[i]) = 0$ for all $i \ne 0$.  When these conditions hold, there is a natural isomorphism $P(\cF) \cong P_\naive(\cF)$.
\end{cor}

\begin{prop}\label{prop:waki-line-bdl}
Let $\lambda \in \bX$.  For $\cF \in \Parity_{(I)}(\Gr)$, there is a natural isomorphism $P(\cW_\lambda \fakestar \cF) \cong \cO_\tcN(\lambda) \otimes P(\cF)$.
In particular, we have $P(\ocW_\lambda) \cong \cO_\tcN(\lambda)$.
\end{prop}
\begin{proof}
By Lemma~\ref{lem:Pnaive-parity} and Corollary~\ref{cor:Pnaive-coh}, we have $P(\cW_\lambda \fakestar \cF) \cong P_\naive(\cW_\lambda \fakestar \cF)$ and $P(\cF) \cong P_\naive(\cF)$.  We will prove that there is a natural isomorphism $P_\naive(\cW_\lambda \fakestar \cF) \cong \cO_\tcN(\lambda) \otimes P_\naive(\cF)$. As in the proof of Lemma~\ref{lem:Pnaive-parity}, we will replace $\cF$ by an object of $\Parity_I(\Gr)$, and work with $\cW_\lambda \star \cF$ throughout.

For a module $M \in \bGamma[\tcN]\lmod$ and a weight $\chi \in \bX$, let $M\lla\chi\rra$ denote the module obtained by shifting the $\bX$-grading by $\chi$.  That is,
\[
(M\lla \chi\rra)_\sigma = M_{\chi+\sigma}.
\]
From the definitions, we have a natural isomorphism
\begin{equation}\label{eqn:bgamma-shift}
\scF(M\lla \chi\rra) \cong \cO_\tcN(\chi) \otimes \scF(M).
\end{equation}
Now, choose a dominant weight $\nu$ such that $\lambda + \nu$ is dominant. There is an obvious surjective map of $\bGamma[\tcN]$-modules
\[
Q_\naive(\cW_\lambda \star \cF[i]) \twoheadrightarrow \bigoplus_{\sigma \in \nu+\bXp} \gHom(\ocW_{-\sigma}, \cW_\lambda[i] \star \cF \star \cR).
\]
On the other hand, Proposition~\ref{prop:conv-waki-hom} gives us an isomorphism $\gHom(\ocW_{-\lambda -\sigma}, \cF \star \cR[i]) \to \gHom(\ocW_{-\sigma}, \cW_\lambda \star \cF[i] \star \cR)$ for $\sigma \in \nu+\bXp$.  Using this, we form a surjective map
\begin{multline*}
Q_\naive(\cF[i])\lla \lambda\rra =
\bigoplus_{\sigma \in -\lambda+\bXp} \gHom(\ocW_{-\sigma-\lambda}, \cF[i] \star \cR) \\
\twoheadrightarrow \bigoplus_{\sigma \in \nu+\bXp} \gHom(\ocW_{-\sigma}, \cW_\lambda[i] \star \cF \star \cR).
\end{multline*}
Both of these maps have thin kernels, and hence become isomorphisms after applying $\scF$.  Using~\eqref{eqn:bgamma-shift}, we conclude that $P_\naive(\cW_{\lambda} \star \cF[i]) \cong \cO_\tcN(\lambda) \otimes P_\naive(\cF)$, as desired.
\end{proof}

\begin{lem}\label{lem:fully-faith-special}
For any $\lambda, \mu\in\bXp$ and $i \in \Z$, the functor $P$ induces an isomorphism
\begin{equation}\label{eqn:fully-faith-special}
\gHom^i (\sky, \ocW_{\lambda}\star\cIcos(\mu)) \simto \uHom^i(P(\sky), P(\ocW_{\lambda}\star\cIcos(\mu))).
\end{equation}
Moreover, both sides vanish for $i \ne 0$.
\end{lem}
\begin{proof}
Suppose first that $i \ne 0$.  Lemma~\ref{lem:Qnaive-parity} tells us that the left-hand side vanishes. For the right-hand side, by Propositions~\ref{prop:monoidal} and~\ref{prop:waki-line-bdl}, we have $P(\sky) = \cO_\tcN$ and $P(\ocW_{\lambda}\star\cIcos(\mu)) = \cO_\tcN(\lambda)\otimes \cow(\mu)$.  Using~\eqref{eqn:rational-resoln} and adjunction, we have
\[
\uHom^i(\cO_\tcN, \cO_\tcN \otimes \cow(\mu))
\cong \uHom^i(\cON, \pi_*(\cO_\tcN \otimes \cow(\mu))) \cong \uHom^i(\cON, \cON \otimes \cow(\mu)).
\]
This vanishes when $i \ne 0$ because $\cON$ is a standard object of $\PCohN$, while $\cON \otimes \cow(\mu)$ has a costandard filtration.  


For later reference, we record the details of the adjunction isomorphism used above: it is the composition of the following sequence of maps, where the last one is induced by the unit $\eta: \cON \to \pi_*\pi^*\cON$:
\begin{multline}\label{eqn:fully-faith-calc}
\uHom(\cO_\tcN, \cO_\tcN(\lambda) \otimes \cow(\mu)) \xrightarrow{\pi_*}
\Hom(\pi_*\cO_\tcN, \pi_*(\cO_\tcN(\lambda) \otimes \cow(\mu))) \\
\cong \Hom(\pi_*\pi^*\cO_\cN, \pi_*(\cO_\tcN(\lambda) \otimes \cow(\mu)))
\xrightarrow{\eta} \Hom(\cO_\cN, \pi_*(\cO_\tcN(\lambda) \otimes \cow(\mu))).
\end{multline}
But since $\eta: \cO_\cN \to \pi_*\pi^*\cO_N$ is itself an isomorphism (see~\eqref{eqn:rational-resoln}), the map induced by $\pi_*$ must also be an isomorphism.

We now study~\eqref{eqn:fully-faith-special} for $i = 0$.  Corollary~\ref{cor:Pnaive-coh} tells us that we may replace it by
\begin{equation}\label{eqn:fully-faith-naive}
\gHom(\sky, \ocW_{\lambda}\star\cIcos(\mu)) \to \uHom(P_\naive(\sky), P_\naive(\ocW_{\lambda}\star\cIcos(\mu))).
\end{equation}
We begin by showing that this map is injective.  Recall that for any $M \in \bGamma[\tcN]\lmod$, there is a natural map
\[
M \to \bigoplus_{\sigma \in \bXp} \Gamma(\tcN, \cO_\tcN(\sigma) \otimes \scF(M)).
\]
Write $M = \bigoplus_{\sigma \in \bX} M_\sigma$, and let $U: M \mapsto M_0$ be the functor that picks out the degree-$(\Z \times \{0\})$ subspace of $M$.  The map above gives rise to a natural map
\begin{equation}\label{eqn:Qnaive-0}
U(M) \to \Gamma(\tcN,\scF(M)).
\end{equation}
Of course, this is not an isomorphism in general, but it may be for specific classes of objects.  In particular, for $M = Q_\naive(\sky)$ or $Q_\naive(\ocW_\lambda \star \cIcos(\mu))$, it is easy to check that~\eqref{eqn:Qnaive-0} is an isomorphism.  Replacing $\Gamma(\tcN, \scF(Q_\naive({-})))$ by $\pi_*P_\naive({-})$, we construct the following commutative diagram:
{\small\[
\xymatrix@C=10pt{
\gHom(\sky, \ocW_\lambda \star \cIcos(\mu)) \ar[d] \\
\uHom(Q_\naive(\sky), Q_\naive(\ocW_\lambda \star \cIcos(\mu))) \ar[r]\ar[d]&
\uHom(U(Q_\naive(\sky)), U(Q_\naive(\ocW_\lambda \star \cIcos(\mu)))) \ar[d]_{\wr}^{\eqref{eqn:Qnaive-0}}\\
\uHom(P_\naive(\sky), P_\naive(\ocW_\lambda \star \cIcos(\mu))) \ar[r]^-{\pi_*}_-{\sim} &
\uHom(\pi_*P_\naive(\sky), \pi_*P_\naive(\ocW_\lambda \star \cIcos(\mu)))}
\]}%
As noted above, the rightmost vertical arrow is an isomorphism.  We saw in~\eqref{eqn:fully-faith-calc} that the bottommost horizontal arrow (which is induced by $\pi_*$) is an isomorphism.  So to prove that~\eqref{eqn:fully-faith-naive} is injective, it suffices to prove that
\begin{equation}\label{eqn:fully-faith-Unaive}
\gHom(\sky, \ocW_{\lambda}\star\cIcos(\mu)) \to \uHom(U(Q_\naive(\sky)), U(Q_\naive(\ocW_{\lambda}\star\cIcos(\mu))))
\end{equation}
is injective.  Let $f: \sky \to \ocW_\lambda \star \cIcos(\mu)\{n\}$ be a nonzero map.  Unwinding the definitions, one finds that
\[
U(Q_\naive(f)): \Hom(\sky, \cR) \to \Hom(\sky, \ocW_\lambda \star \cIcos(\mu)\{n\} \star \cR)
\]
is just given by $U(Q_\naive(f))(g) = f \star g$.  Let $\eta: \sky \to \cR$ be the unit morphism.  Then $f \star \eta$ is nonzero, because $(f \star \epsilon) \circ (f \star \eta) = f$, where $\epsilon: \cR \to \sky$ is the counit coming from the Hopf algebra structure on $\bk[\Gv]$.  Thus, $U(Q_\naive(f))$ is nonzero, and so~\eqref{eqn:fully-faith-Unaive} and~\eqref{eqn:fully-faith-naive} are both injective.

To finish the proof, we must show that~\eqref{eqn:fully-faith-naive} is actually an isomorphism.  It suffices to check that both sides have the same dimension in each degree of the grading.  This is achieved by the following calculation.
\begin{align*}
&\gHom(\sky, \ocW_\lambda \star \cIcos(\mu)) \\
& \cong \gHom(\uistd(-\lambda)\{-\delta_{-\lambda}\}, \cIcos(\mu)) 
  && \text{by~Prop.~\ref{prop:conv-waki-hom}} \\
& \cong \bk \otimes_{\sH^\bullet(\Gr_{-w_0\lambda})} \gHom(\cJstd(-w_0\lambda),\cIcos(\mu)) 
  && \text{by Lemma~\ref{lem:costalk-compare}} \\
& \cong
\bk \otimes_{\sH^\bullet(\Gr_{-w_0\lambda})} \uHom(\pcstd(-w_0\lambda)\la\delta_{w_0\lambda}\ra, \cO_\cN\otimes \cow(\mu)) 
  && \text{by Prop.~\ref{prop:sph-main} \&\ Cor.~\ref{cor:psph-std}}\\
& \cong
\uHom(\pcpstd(-w_0\lambda)\la -\delta_{w_0\lambda}\ra, \cO_\cN \otimes \cow(\mu)) 
  && \text{by Lemma~\ref{lem:pcstd-compare}} \\
& \cong
\uHom(\cO_\tcN(-\lambda), \cO_\tcN \otimes \cow(\mu))
  && \text{by~\eqref{eqn:springer-dualizing}} \\
& \cong
\uHom(\cO_\tcN, \cO_\tcN(\lambda) \otimes \cow(\mu))  \\
& \cong
\uHom(P_\naive(\sky), P_\naive(\ocW_\lambda \star \cIcos(\mu))). &&\qedhere
\end{align*}
\end{proof}

\begin{lem}[cf.~{\cite[Lemma~5]{bez:ctm}}]\label{lem:line-po-hom}
Let $\lambda, \mu \in \bX$.
\begin{enumerate}
\item If $\lambda \not\preceq \mu$, then $\uHom^\bullet(\cO_\tcN(\mu), \cO_\tcN(\lambda)) = 0$.
\item We have $\uEnd(\cO_\tcN(\lambda)) \cong \bk$, and $\uHom^i(\cO_\tcN(\lambda), \cO_\tcN(\lambda)) = 0$ for $i \ne 0$.
\end{enumerate}
\end{lem}
\begin{proof}
By applying the equivalence of categories $\cO_\tcN(-\lambda) \otimes({-})$, we may assume without loss of generality that $\lambda = 0$.  Using~\eqref{eqn:springer-dualizing}, we find that
\[
\uHom^\bullet(\cO_\tcN(\mu), \cO_{\tcN}) \cong \uHom^\bullet(\cO_\tcN(\mu), \pi^!\cO_\cN) \cong \uHom^\bullet_{\Db\CohN}(\pi_*\cO_\tcN(\mu), \cO_\cN).
\]
By~\cite[Proposition~5.6]{a}, the latter vanishes unless $\mu \succeq 0$.  In the special case where $\mu = 0$, we use~\eqref{eqn:rational-resoln} and~\cite[Lemma~5.5(2)]{a} to see that $\uHom(\pi_*\cO_\tcN, \cO_\cN) \cong \bk$, and that $\uHom^i(\pi_*\cO_\tcN, \cO_\cN) = 0$ for $i \ne 0$.
\end{proof}

\begin{proof}[Conclusion of the proof of Theorem~\ref{thm:main}]
We begin by showing that for all $\cF \in \DmixI(\Gr)$, the map
\begin{equation}\label{eqn:sky-fully-faith}
\gHom^i(\sky, \cF) \to \uHom^i(P(\sky), P(\cF))
\end{equation}
is an isomorphism. By Proposition~\ref{prop:dmixigr-gen-x}, it suffices to consider the cases where $\cF = \ocW_\lambda$ with $\lambda \not\preceq 0$, or else $\cF = \ocW_\lambda \star \cIcos(\mu)$ with $\lambda, \mu \in \bXp$.  In the former case, both sides of~\eqref{eqn:sky-fully-faith} vanish, by Lemmas~\ref{lem:waki-po-hom} and~\ref{lem:line-po-hom}.  The latter case is covered by Lemma~\ref{lem:fully-faith-special}. Thus,~\eqref{eqn:sky-fully-faith} is an isomorphism in all cases.

Next, let $\cF \in \Parity_{(I)}(\Gr)$, and let $\lambda \in \bXp$. Consider the following diagram of natural maps:
\begin{equation}\label{eqn:waki-fully-faith}
\vcenter{\xymatrix{
\gHom^i(\sky, \cF) \ar[r]^-{\sim} \ar[dd]^{\wr}_{\omega_\lambda} &
\uHom^i(P(\sky), P(\cF)) \ar[d]^{\wr} \\ 
& \uHom^i(\cO_\tcN(\lambda) \otimes P(\sky), \cO_\tcN(\lambda) \otimes P(\cF))
\\
\gHom^i(\ocW_\lambda, \cW_\lambda \fakestar \cF) \ar[r] &
\uHom^i(P(\ocW_\lambda), P(\cW_\lambda \fakestar \cF)) \ar[u]_{\wr} }}
\end{equation}
All the vertical maps are isomorphisms, and the top horizontal map is an isomorphism by~\eqref{eqn:sky-fully-faith}.  

When $i = 0$, the natural isomorphism of Proposition~\ref{prop:waki-line-bdl} tells us that this diagram commutes, and so the bottom horizontal map is an isomorphism as well.  When $i \ne 0$, that naturality is not a priori available---but both $\Hom$-groups in the top row vanish, and so every $\Hom$-group in the diagram vanishes.  

Thus, the bottom arrow in~\eqref{eqn:waki-fully-faith} is an isomorphism in all cases.  Note that the equivariant derived category $\Dmix_I(\Gr)$ is generated by objects of the form $\cW_\lambda \star \cF$ with $\cF \in \Parity_I(\Gr)$, because $\cW_\lambda \star ({-})$ is an autoequivalence of that category.  Since the image of $\For: \Dmix_I(\Gr) \to \DmixI(\Gr)$ generates $\DmixI(\Gr)$, we deduce that objects of the form $\cW_\lambda \fakestar \cF$ with $\cF \in \Parity_{(I)}(\Gr)$ generate $\DmixI(\Gr)$.  Therefore, the bottom isomorphism in~\eqref{eqn:waki-fully-faith} implies that
\[
\gHom^i(\ocW_\lambda,\cG) \to \uHom^i(P(\ocW_\lambda), P(\cG))
\]
is an isomorphism for all $\cG \in \DmixI(\Gr)$.  Finally, the $\ocW_\lambda$ also generate $\DmixI(\Gr)$, so $P$ is fully faithful.  The line bundles $\cO_\tcN(\lambda)$ generate $\Db\Coh(\tcN)$ as a triangulated category, so $P$ is also essentially surjective, and hence an equivalence.
\end{proof}

\section{The exotic t-structure}
\label{sect:exotic}

The \emph{exotic t-structure} on $\Db\Coh(\tcN)$ was defined in~\cite[\S2.3]{bez:ctm}, at least for $\bk = \C$.  We will briefly review the steps of the construction, and check that they go through in positive characteristic as well.

\subsection{Exceptional sets and mutation}
\label{ss:excep}

This subsection contains a very cursory review of the definitions and facts we will need from~\cite[\S2.1]{bez:ctm}.  For details, the reader should consult~\cite{bez:ctm} and the references indicated therein, especially~\cite{bk:rfsfr,bgs}.

Let $D$ be a $\bk$-linear triangulated category equipped with an autoequivalence $\la 1\ra: D \to D$.  Let $\Omega$ be a partially ordered set, with partial order $\preceq$.  A collection of objects $\{X_\gamma \mid \gamma \in \Omega \}$ is called a \emph{full graded $\preceq$-exceptional set} if $D$ is generated by the set $\{ X_\gamma\la n\ra \mid \gamma \in \Omega,\ n \in \Z \}$, and if the following three additional conditions hold:
\[
\text{$\uHom^\bullet(X_\gamma, X_\xi) = 0$ if $\xi \not\preceq \gamma$,}
\qquad
\text{$\uHom^i(X_\gamma, X_\gamma) = 0$ if $i \ne 0$,}
\qquad
\text{$\uEnd(X_\gamma) \cong \bk$.}
\]
Now, suppose $\unlhd$ is another partial order on $\Omega$, and that $\{ Y_\gamma \mid \gamma \in \Omega \}$ is a full graded $\unlhd$-exceptional set.  We say that $\{ Y_\gamma \}$ is a \emph{$\unlhd$-mutation} of $\{ X_\gamma \}$ if the following two conditions hold:
\begin{enumerate}
\item For each $\gamma$, the triangulated category generated by $\{ X_\xi\la n\ra \mid \xi \unlhd \gamma,\ n \in \Z\}$ coincides with that generated $\{ Y_\xi\la n\ra \mid \xi \unlhd \gamma,\ n \in \Z\}$.
\item For each $\gamma$, there is a distinguished triangle $X_\gamma \to Y_\gamma \to U_\gamma \to$ such that $U_\gamma$ lies in the triangulated subcategory generated by $\{ X_\xi\la n\ra \mid \xi \lhd \gamma,\ n \in \Z\}$.
\end{enumerate}
Suppose $(\Omega,\unlhd)$ is isomorphic as a partially ordered set to a subset of $\N$. Then, according to~\cite[Lemma~1]{bez:ctm}, there exists a unique $\unlhd$-mutation of any full graded $\preceq$-exceptional set.

On the other hand, if $(\Omega,\unlhd)$ is isomorphic to a subset of $\N$, then by~\cite[Proposition~2]{bez:ctm}, any full graded $\unlhd$-exceptional set $\{ Y_\gamma \}$ determines a t-structure on $D$.  Specifically, the categories 
\begin{equation}\label{eqn:mut-t-struc}
\begin{aligned}
D^{\le 0} &= \{ A \in D \mid \text{$\uHom(A,Y_\gamma[i]) = 0$ for all $i < 0$} \}, \\
D^{\ge 0} &= 
\begin{array}{c}
\text{the smallest strictly full subcategory of $D$ that is stable under}\\
\text{extensions and contains $Y_\gamma\la n\ra[i]$ for all $\gamma \in \Omega$, $n \in \Z$ and $i \le 0$}
\end{array}
\end{aligned}
\end{equation}
constitute a t-structure on $D$.

The heart $\sA = D^{\le 0} \cap D^{\ge 0}$ is clearly stable under $\la 1\ra$.  According to~\cite[Proposition~2]{bez:ctm}, every object in $\sA$ has finite length, and the isomorphism classes of simple objects, up to $\la 1\ra$, are in bijection with $\Omega$.

In fact, $\sA$ is very close to being a graded quasihereditary category: it satisfies the axioms (1)--(5) of~\cite[\S3.2]{bgs}, but axiom~(6) may fail.  The costandard objects are of the form ${}^tH^0(Y_\gamma\la n\ra)$, where ${}^tH$ denotes cohomology with respect to our t-structure.  The standard objects are of the form ${}^tH^0({}^\backprime Y_\gamma\la n\ra)$, where $\{ {}^\backprime Y_\gamma \}$ is the \emph{dual} exceptional set in the sense of~\cite[\S2.1.2]{bez:ctm}.

\subsection{Exotic sheaves}

By Lemma~\ref{lem:line-po-hom}, the collection $\{ \cO_\tcN(\lambda) \mid \lambda \in \bX \}$ is a full graded $\preceq$-exceptional set.  Now consider the partial order $\leq$ on $\bX$.  Certainly, $\leq$ can be refined to a total order $\leq'$ such that $(\bX,\leq')$ is isomorphic to $\N$, and then we can form the $\leq'$-mutation of $\{ \cO_\tcN(\lambda) \}$.  It will be convenient to name the objects of the new exceptional set with a built-in shift: let $\{ \excos(\lambda)\la - \delta_\lambda \ra \}$ be the $\leq'$-mutation of $\{ \cO_\tcN(\lambda) \}$.  Thus, for each $\lambda$, there is a distinguished triangle
\[
\cO_\tcN(\lambda) \to \excos(\lambda)\la -\delta_\lambda\ra \to \cG \to
\]
in $\Db\Coh(\tcN)$, where $\cG$ lies in the subcategory generated by $\{ \cO_\tcN(\mu)\la n\ra \mid \mu <' \lambda \}$.

As in~\eqref{eqn:mut-t-struc}, the objects $\{ \excos(\lambda) \}$ determine a t-structure on $\Db\Coh(\tcN)$.  We call this the \emph{exotic t-structure}, and we denote its heart by $\ExCoh(\tcN)$. This definition appears to depend on the choice of refinement $\leq'$ of $\leq$, but we will see below that it is actually independent of that choice.

\subsection{Adverse sheaves and exotic sheaves}

We can also apply the notions of~\S\ref{ss:excep} to $\DmixI(\Gr)$ with the autoequivalence $\{1\}: \DmixI(\Gr) \to \DmixI(\Gr)$.

\begin{lem}\label{lem:waki-mut}
\begin{enumerate}
\item The set $\{ \ocW_\lambda \mid \lambda \in \bX \}$ is a full graded $\preceq$-exceptional set.
\item If $\leq'$ is any total order on $\bX$ that refines $\leq$ and such that $(\bX, \leq')$ is isomorphic to $\N$, then the $\leq'$-mutation of $\{ \ocW_\lambda \}$ is isomorphic to $\{ \uicos(\lambda)\{-\delta_\lambda\} \}$.
\end{enumerate}
\end{lem}
\begin{proof}
The assertion that $\{ \ocW_\lambda \}$ is a $\preceq$-exceptional set is just a restatement of Lemma~\ref{lem:waki-po-hom}.  A routine adjunction argument shows that $\{ \uicos(\lambda)\{-\delta_\lambda\} \}$ is a $\leq$-exceptional set, so it is also $\leq'$-exceptional for any choice of $\leq'$.

Lemmas~\ref{lem:waki-support} and~\ref{lem:waki-support-gen} imply that for each $\lambda$, there is a distinguished triangle
\[
\ocW_\lambda \to \uicos(\lambda)\{-\delta_\lambda\} \to \cK_\lambda \to
\]
where $\cK_\lambda$ lies in the subcategory generated by $\{ \ocW_\mu\{n\} \mid \mu < \lambda \}$. Finally, those same lemmas also tell us that $\{ \ocW_\mu\{n\} \mid \mu \leq' \lambda \}$ and $\{ \uicos(\mu)\{n\} \mid \mu \leq' \lambda \}$ generate the same subcategory of $\DmixI(\Gr)$.
\end{proof}

\begin{lem}\label{lem:adv-mut}
Let $\leq'$ be a total order on $\bX$ that refines $\leq$ and such that $(\bX, {\leq'})$ is isomorphic to $\N$.  Then the t-structure on $\DmixI(\Gr)$ determined by the $\leq'$-exceptional set $\{ \uicos(\lambda)\{-\delta_\lambda\} \}$ is the adverse t-structure.
\end{lem}
In particular, this lemma tells us that the t-structure obtained by mutation of the exceptional set $\{\ocW_\lambda\}$ is independent of the choice of $\le'$.
\begin{proof}
It is obvious that the category $D^{\ge 0}$ as described in~\eqref{eqn:mut-t-struc} coincides with $\ad\DmixI(\Gr)^{\ge 0}$ as described in~\eqref{eqn:adv-defn2}.  An easy adjunction argument shows that $D^{\le 0}$ in~\eqref{eqn:mut-t-struc} agrees with $\ad\DmixI(\Gr)^{\le 0}$ as in~\eqref{eqn:adv-defn}.
\end{proof}

The following statement is the main result of this section.

\begin{thm}\label{thm:adv-exotic}
The equivalence $P$ of Theorem~\ref{thm:main} induces an equivalence of  abelian categories
\[
P: \Adv_{(I)}(\Gr) \simto \ExCoh(\tcN).
\]
\end{thm}
\begin{proof}
Recall from Proposition~\ref{prop:waki-line-bdl} that $P$ takes the exceptional set $\{ \ocW_\lambda \}$ to the exceptional set $\{ \cO_\tcN(\lambda) \}$.  It must therefore take the $\le'$-mutation of the former to the $\le'$-mutation of the latter: $P(\uicos(\lambda)) \cong \excos(\lambda)$.  Lastly, $P$ must also take the t-structure determined by $\{ \uicos(\lambda) \}$ to that determined by $\{ \excos(\lambda) \}$.  In view of Lemma~\ref{lem:adv-mut}, we are done.
\end{proof}

Let $\{ \exstd(\lambda) \mid \lambda \in \bX \}$ the the dual exceptional set to $\{ \excos(\lambda) \}$.  The reasoning above shows that we must also have
$P(\uistd(\lambda)) \cong \exstd(\lambda)$.

Since Theorem~\ref{thm:adv-exotic} gives an equivalence of quasihereditary categories, it certainly restricts to an equivalence between their respective subcategories of tilting objects.  We obtain the following statement, which appeared earlier as Theorem~\ref{thm:intro-exotic}.

\begin{prop}
The equivalence $P$ of Theorem~\ref{thm:main} induces an equivalence of additive categories
\[
P: \Parity_{(I)}(\Gr) \simto \Tilt(\ExCoh(\tcN)).
\]
\end{prop}

We also obtain a slew of nontrivial facts about $\ExCoh(\tcN)$ just by transferring facts about $\Adv_{(I)}(\Gr)$ from~\S\ref{ss:adverse} across this equivalence.  Some of these are recorded in the following proposition.    

\begin{prop}
\begin{enumerate}
\item The objects $\excos(\lambda)$ and $\exstd(\lambda)$ and the category $\ExCoh(\tcN)$ are all independent of the choice of $\leq'$.
\item $\ExCoh(\tcN)$ is a graded quasihereditary category, 
\item The $\excos(\lambda)$ (resp.~$\exstd(\lambda)$) lie in $\ExCoh(\tcN)$ and are the costandard (resp.~standard) objects therein.
\item There is a derived equivalence
\[
\Db\ExCoh(\tcN) \simto \Db\Coh(\tcN).
\]
\end{enumerate}
\end{prop}

\begin{prop}
\begin{enumerate}
\item Every line bundle on $\tcN$ (and, more generally, every vector bundle) belongs to $\ExCoh(\tcN)$.\label{it:bundle-exotic}
\item For all $V \in \Rep(\Gv)$, the perverse sheaf $\Sat(V)$ is also an adverse sheaf.  As an object of $\Adv_{(I)}(\Gr)$, $\Sat(V)$ admits a filtration whose subquotients are Wakimoto sheaves.\label{it:sph-adv}
\end{enumerate}
\end{prop}
Part~\eqref{it:sph-adv} of this proposition should be compared to~\cite[Theorem~4]{ab}.
\begin{proof}
Part~\eqref{it:bundle-exotic} follows from the fact that Wakimoto sheaves on $\Gr$ are adverse (see~\S\ref{ss:wakimoto}).  In particular, part~\eqref{it:bundle-exotic} tells us that trivial vector bundles of the form $\cO_\tcN \otimes V$, where $V \in \Rep(\Gv)$, lie in $\ExCoh(\tcN)$.  Since $P^{-1}(\cO_\tcN \otimes V) \cong \Sat(V)$, part~\eqref{it:sph-adv} follows.
\end{proof}

We finish with a fact that may be useful for computations.  It should be compared with the corresponding fact (see~\eqref{eqn:aj}) for $\PCohN$.

\begin{prop}
The costandard objects in $\ExCoh(\tcN)$ also belong to $\Coh(\tcN)$.
\end{prop}
\begin{proof}
Recall that $\uicos(\lambda)$ has weights${}\ge 0$, and thus can be written as a complex of parity sheaves $\cE^\bullet \in \Kb\Parity_{(I)}(\Gr)$ with nonzero terms only in nonpositive degrees.  From the definition of $P$, we see immediately that $H^i(P(\uicos(\lambda))) \cong H^i(\excos(\lambda))$ vanishes when $i > 0$.  Now, let $k$ be the smallest integer such that $H^k(\excos(\lambda)) \ne 0$.  It is easy to see that every nonzero coherent sheaf on $\tcN$ admits a nonzero map from (and indeed, is a quotient of) some vector bundle.  Let $\cF$ be a vector bundle such there is a nonzero map $\cF \to H^k(\excos(\lambda))$.  This gives rise to a nonzero map $\cF[-k] \to \excos(\lambda)$.  Since $\cF$ and $\excos(\lambda)$ both lie in $\ExCoh(\tcN)$, we must have $k \ge 0$.  But we already knew that $k \le 0$, so $k = 0$, and $\excos(\lambda)$ is a coherent sheaf.
\end{proof}

\appendix
\section{Complements on mixed modular derived categories\\ (joint with Simon Riche\protect\footnote{S.R. was partially supported by ANR Grant No.
ANR-13-BS01-0001-01.\\
\strut\hspace{\parindent}{\it Address}: Universit{\'e} Blaise Pascal - Clermont-Ferrand II, Laboratoire de Math{\'e}matiques, CNRS, UMR 6620, Campus universitaire des C{\'e}zeaux, F-63177 Aubi{\`e}re Cedex, France.\\
\strut\hspace{\parindent}{\it E-mail address}: \texttt{simon.riche@math.univ-bpclermont.fr}})}
\label{sect:mixed}

\newcommand{\dmix}{\Delta^\mix}
\newcommand{\nmix}{\nabla^\mix}

\subsection{Overview}
\label{ss:ghom}

Let $X$ be a variety or an ind-variety equipped with a stratification $\scS$ by affine spaces, and let $\bk$ be a field or a complete discrete valuation ring.  Assume that $X$ and $\scS$ satisfy assumptions~\textrm{\bf{(A1)}} and~\textrm{\bf{(A2)}} of~\cite{arc:f2} with respect to $\bk$. Let $\Parity_\scS(X)$ be the additive category of parity complexes on $X$ with coefficients in $\bk$.  For each $s \in \scS$, let $i_s: X_s \hookrightarrow X$ be the inclusion of the corresponding stratum, and let $\cE_s$ denote the unique indecomposable parity sheaf supported on $\overline{X_s}$ and whose restriction to $X_s$ is $\ubk\{\dim X_s\}$.  Following~\cite{arc:f2}, we define the category
\[
\Dmix_\scS(X) := \Kb\Parity_\scS(X).
\]
Below is a summary of the main features of this category from~\cite{arc:f2,arc:f3}. Later subsections give a handful of new results that were not needed in those sources.

\subsubsection*{\it Shift and Tate twist.} In addition to the usual shift functor $[1]: \Dmix_\scS(X) \to \Dmix_\scS(X)$, there is another automorphism $\{1\}: \Dmix_\scS(X) \to \Dmix_\scS(X)$, induced by an automorphism of $\Parity_\scS(X)$. We also set $\la 1\ra := \{-1\}[1]$.  This last automorphism is called the \emph{Tate twist}.

\subsubsection*{\it Sheaf functors.} If $h: Y \hookrightarrow X$ is the inclusion of a locally closed union of strata, then there are functors $h_*$, $h_!$, $h^*$, and $h^!$ between $\Dmix_\scS(X)$ and $\Dmix_\scS(Y)$ that enjoy all the usual adjunction properties.  

\subsubsection*{\it Mixed perverse sheaves.} There is a \emph{perverse t-structure} whose heart is denoted $\Perv^\mix_\scS(X)$.  This category is stable under $\la 1\ra$, and it contains the objects $\dmix_s := i_{s!} \ubk\{ \dim X_s\}$ and $\nmix_s := i_{s*} \ubk\{ \dim X_s\}$. When $\bk$ is a field, $\Perv^\mix_\scS(X)$ is a graded quasihereditary category.

\subsubsection*{\it Weights.} There are notions of \emph{weights} and \emph{purity} that share some formal properties with the corresponding notions in~\cite{deligne, bbd}.  The functor $\{1\}$ preserves weights, while $[1]$ and $\la 1\ra$ increase weights by $1$.  The definitions imply that
\begin{equation}\label{eqn:wt-hom-van}
\Hom(\cF,\cG) = 0\qquad\text{if $\cF$ has weights${}< n$ and $\cG$ has weights${}\ge n$.}
\end{equation}
An object $\cF \in \Dmix_\scS(X)$ is said to be \emph{$*$-pure} (resp.~\emph{$!$-pure}) \emph{of weight $n$} if $i_s^*\cF$ (resp.~$i_s^!\cF$) is pure of weight $n$ for all $s \in \scS$.  The notion of $*$-purity corresponds roughly to pointwise purity in the sense of~\cite{bbd}.  By~\cite[Lemma~3.5]{arc:f3}, an object that is $*$- and $!$-pure of weight $n$ is pure of weight $n$.

\subsubsection*{Hom functors.} One can associate to any $\cF, \cG \in \Dmix_\scS(X)$ a certain object in the derived category of $\bk$-modules denoted $\RHom(\cF,\cG)$.  This construction is functorial in both variables, and it satisfies $H^i(\RHom(\cF,\cG)) \cong \Hom^i(\cF,\cG)$.

The following variation on this construction will be useful:
for $\cF, \cG \in \Dmix_\scS(X)$, let $\gHom(\cF,\cG)$ to be the graded vector space given by
\[
\gHom(\cF,\cG)_n := \Hom(\cF,\cG\{n\}).
\]
One can then define a derived version $\gRHom(\cF,\cG)$ as in~\cite[\S2.7]{arc:f2}, satisfying $H^i(\gRHom(\cF,\cG)) \cong \gHom^i(\cF,\cG)$.

\subsection{The adverse t-structure}
\label{ss:adverse}

In this subsection, we assume for simplicity that $\bk$ is a field.
Consider the following full subcategories of $\Dmix_\scS(X)$:
\begin{equation}\label{eqn:adv-defn}
\begin{aligned}
\ad\Dmix_\scS(X)^{\le 0} &= \{ \cF \mid \text{for all $s \in \scS$, $i_s^*\cF$ has weights${}\ge 0$} \}, \\
\ad\Dmix_\scS(X)^{\ge 0} &= \{ \cF \mid \text{for all $s \in \scS$, $i_s^!\cF$ has weights${}\le 0$} \}.
\end{aligned}
\end{equation}
It is easy to check that these categories admit the following alternative descriptions:
\begin{equation}\label{eqn:adv-defn2}
\begin{aligned}
\ad\Dmix_\scS(X)^{\le 0} &= \begin{array}{@{}c@{}}
\text{the smallest strictly full subcategory that is stable under}\\
\text{extensions and contains $\dmix_s\{n\}[k]$ for all $n \in \Z$ and $k \ge 0$}
\end{array} \\
\ad\Dmix_\scS(X)^{\ge 0} &= \begin{array}{@{}c@{}}
\text{the smallest strictly full subcategory that is stable under}\\
\text{extensions and contains $\nmix_s\{n\}[k]$ for all $n \in \Z$ and $k \le 0$}
\end{array}
\end{aligned}
\end{equation}
We put
\[
\Adv_\scS(X) := \ad\Dmix_\scS(X)^{\le 0} \cap \ad\Dmix_\scS(X)^{\ge 0},
\]
and we call objects of $\Adv_\scS(X)$ \emph{adverse sheaves}.

\begin{prop}\label{prop:adverse-t-struc}
The pair $(\ad\Dmix_\scS(X)^{\le 0}, \ad\Dmix_\scS(X)^{\ge 0})$ constitutes a bounded t-structure on $\Dmix_\scS(X)$.  Its heart $\Adv_\scS(X)$ is a graded quasiheredi\-tary category in which the standard (resp.~costandard) objects are those of the form
\[
\dmix_s\{n\},
\qquad\text{resp.}\qquad
\nmix_s\{n\},
\]
and the category $\Tilt(\Adv_\scS(X))$ of tilting objects in $\Adv_\scS(X)$ is identified with $\Parity_\scS(X)$.  Lastly, there is an equivalence of  categories $\Db\Adv_\scS(X) \simto \Dmix_\scS(X)$.
\end{prop}

\begin{rmk}
The definitions above also make sense in the setting of an equivariant mixed derived category $\Dmix_{H,\scS}(X)$ as in~\cite[\S3.5]{arc:f2} or \S\ref{ss:eqvt-der-cat}.  \emph{However, $\ad\Dmix_{H,\scS}(X)^{\le 0}$ and $\ad\Dmix_{H,\scS}(X)^{\ge 0}$ do not constitute a t-structure in general.}  Specifically, truncation distinguished triangles as in~\cite[D\'efinition~1.3.1(iii)]{bbd} can fail to exist.  This can be seen already in the case where $X$ is a single stratum.
\end{rmk}

\begin{rmk}
When $\bk$ is not a field, there is a unique t-structure with $\ad\Dmix_\scS(X)^{\le 0}$ as in~\eqref{eqn:adv-defn} or~\eqref{eqn:adv-defn2}, but the descriptions of $\ad\Dmix_\scS(X)^{\ge 0}$ must be modified in this case (cf.~\cite[Proposition~3.4]{arc:f2}).  The heart of this t-structure behaves in many ways like a quasihereditary category.  In particular, it has properties like those discussed in~\cite[\S3.3]{arc:f2}.
\end{rmk}

\begin{proof}[Proof Sketch]
This statement is very similar to~\cite[Lemma~10.8]{arc:kd}.  

Suppose first that $X$ consists of a single stratum.  Then $\Dmix_\scS(X)$ is a semisimple category.  The description given in~\cite[Lemma~3.1]{arc:f2} can be used to check that $(\ad\Dmix_\scS(X)^{\le 0}, \ad\Dmix_\scS(X)^{\ge 0})$ is indeed a t-structure.  For general $X$, the claim that this is a t-structure follows by the machinery of recollement.

Next, we claim that all the $\dmix_s\{n\}$ lie in the heart of this t-structure.  It suffices to show that $\Hom^k(\dmix_s\{n\}, \dmix_t\{m\}) = 0$ for $k < 0$.  By adjunction, we have
\[
\Hom^k(\dmix_s\{n\}, \dmix_t\{m\}) \cong \Hom(\ubk\{\dim X_s + n\}, i_s^!\dmix_t\{m\}[k]).
\]
By~\cite[Lemmas~3.3 and~3.4]{arc:f3}, $i_s^!\dmix_t\{m\}[k]$ has weights${}\le k$, while $\ubk\{\dim X_s + n\}$ is, of course, pure of weight~$0$.  When $k < 0$, the semisimplicity of $\Dmix_\scS(X_s)$ (together with, say, the description in~\cite[Example~3.2]{arc:f3}) implies that the $\Hom$-group above vanishes.  A similar argument shows that the $\nmix_s\{n\}$ lie in $\Adv_\scS(X)$ as well.

General principles from the theory of quasihereditary categories then imply that $\Adv_\scS(X)$ is quasihereditary, that it has the claimed standard and costandard objects, and that we have $\Db\Adv_\scS(X) \simto \Dmix_\scS(X)$.

Finally, let us check that $\Tilt(\Adv_\scS(X)) = \Parity_\scS(X)$.  By considering weights, we see that $\Hom^k(\dmix_t\{m\}, \cE_s\{n\}) = \Hom^k(\cE_s\{n\}, \nmix_t\{m\}) = 0$ for all $k > 0$.  According to the criterion in~\cite[Lemma~4]{bez:ctm}, each $\cE_s\{n\}$ is an indecomposable tilting object in $\Adv_\scS(X)$.  On the other hand, we have produced ``enough'' tilting objects: by the classification in, say,~\cite[Proposition~A.4]{arc:f2}, every indecomposable tilting object in $\Adv_\scS(X)$ must be isomorphic to some $\cE_s\{n\}$.
\end{proof}

When $X$ is a (finite-dimensional) flag variety, the adverse t-structure is the transport of the perverse t-structure across the ``self-duality'' equivalence of~\cite{arc:f2}.

\begin{lem}\label{lem:adv-std-filt}
Let $\cF \in \Dmix_\scS(X)$.  The following conditions are equivalent:
\begin{enumerate}
\item $\cF$ is $*$-pure of weight $0$.\label{it:adv-std-filt-stalk}
\item $\cF$ lies in $\Adv_\scS(X)$ and has a standard filtration.\label{it:adv-std-filt-filt}
\end{enumerate}
Likewise, the following conditions are equivalent:
\begin{enumerate}
\item $\cF$ is $!$-pure of weight $0$.
\item $\cF$ lies in $\Adv_\scS(X)$ and has a costandard filtration.
\end{enumerate}
\end{lem}
\begin{proof}
We will just prove the first equivalence.  It is clear that every standard object satisfies condition~\eqref{it:adv-std-filt-stalk}, so~\eqref{it:adv-std-filt-filt} implies~\eqref{it:adv-std-filt-stalk}.  For the other implication, we proceed by induction on the number of strata in the support of $\cF$.  Let $X_s$ be a stratum that is open in the support of $\cF$.  Let $Z$ be the union of the closures of all strata other than $X_s$ in the support of $\cF$, and let $h: Z \hookrightarrow X$ be the inclusion map.  Then there is a distinguished triangle $i_{s!}i_s^*\cF \to \cF \to h_*h^*\cF \to$.  By induction, $h_*h^*\cF$ is adverse and has a standard filtration.  (Note that the recollement setup implies that $h_*$ is t-exact for the adverse t-structure.)  On the other hand, $i_s^*\cF$ is a direct sum of various $\ubk\{n\}$, so $i_{s!}i_s^*\cF$ is a direct sum of various $\dmix_s\{n\}$.  The result follows.
\end{proof}

Below, we will study the exactness of various functors related to \emph{stratified morphisms} in the sense of~\cite[\S2.6]{arc:f2}.  These statements will sometimes be invoked in the equivariant setting, but since there is no t-structure in that case, some caution is required.  Let us spell out what ``exactness'' means.  Let $Y = \bigcup_{t \in \scT} Y_t$ be another variety equipped with a stratification by affine spaces and satisfying \textrm{\bf{(A1)}} and \textrm{\bf{(A2)}}.  Suppose $H$ and $K$ are connected algebraic groups acting on $X$ and $Y$, respectively, and that these actions preserve the strata.  A functor $F: \Dmix_{H,\scS}(X) \to \Dmix_{K,\scT}(Y)$ is said to be \emph{left adverse-exact}, resp.~\emph{right adverse-exact},~if
\[
F(\ad\Dmix_{H,\scS}(X)^{\ge 0}) \subset \ad\Dmix_{K,\scT}(Y)^{\ge 0},
\qquad\text{resp.}\qquad
F(\ad\Dmix_{H,\scS}(X)^{\le 0}) \subset \ad\Dmix_{K,\scT}(Y)^{\le 0}.
\]
If $F$ is both left and right adverse-exact, we say simply that it is \emph{adverse-exact}.  Of course, in the nonequivariant case, these notions coincide with the usual (left or right) t-exactness for the adverse t-structure.

\begin{lem}\label{lem:adv-smooth}
Suppose $f: X \to Y$ is a proper, smooth stratified morphism.  Then $f_*$, $f^*$, and $f^!$ are all adverse-exact.
\end{lem}
\begin{proof}
The adverse-exactness of $f_*$ is immediate from~\cite[Lemma~3.7]{arc:f2}.  Next, let $t \in \scT$ and $s \in \scS$, and observe that
\[
i_s^*(f^*\dmix_t\{n\}) \cong
\begin{cases}
0 & \text{if $X_s \not\subset f^{-1}(X_t)$,} \\
\ubk\{\dim X_t+n\} & \text{if $X_s \subset f^{-1}(X_t)$.}
\end{cases}
\]
In particular, $i_s^*(f^*\dmix_t\{n\})$ has weights${}\ge 0$, so $f^*\dmix_t\{n\}$ lies in $\ad\Dmix_\scS(X)^{\le 0}$.  Similar reasoning with $\nmix_t\{n\}$ shows that $f^*$ is adverse-exact.  Since $f^! \cong f^*\{-2d\}$ where $d$ is the relative dimension of $f$, the functor $f^!$ is adverse-exact as well.
\end{proof}

\begin{lem}\label{lem:adv-smooth-filt}
Suppose $f: X \to Y$ is a proper, smooth, surjective stratified morphism.  Then $f^*$ kills no nonzero adverse sheaf.  Moreover, if $\cF \in \Adv_\scT(Y)$ has a standard (resp.~costandard) filtration, then $f^*\cF$ does as well.  The same statements also hold for $f^!$.
\end{lem}
(Note that, in contrast with~\cite[Corollary~3.9]{arc:f2}, $f^*$ and $f^!$ do not, in general, take simple adverse sheaves to simple adverse sheaves.)
\begin{proof}
For the first assertion, it suffices to show that $f^*$ sends any simple adverse sheaf to a nonzero adverse sheaf.  Let $\cF$ be a simple adverse sheaf on $Y$.  Then $\cF$ is supported on the closure of some stratum $\overline{Y_t}$, and $\cF|_{Y_t} \cong \ubk\{n\}$ for some $n$.  The object $f^*\cF$ clearly has nonzero restriction to any stratum $X_s \subset f^{-1}(Y_t)$, so it is nonzero.

If all $i_t^*\cF$ are pure of weight~$0$, it is easy to see that all $i_s^*(f^*\cF)$ are pure of weight~$0$, so by Lemma~\ref{lem:adv-std-filt}, $f^*$ preserves the property of having a standard filtration.  Using the fact that $f^* \cong f^!\{2d\}$, we obtain the corresponding statement for costandard filtrations, or for $f^!$ in place of $f^*$.
\end{proof}

\subsection{Non-affine stratifications}
\label{ss:non-affine}

The general theory developed in~\cite{arc:f2, arc:f3} involves the assumption throughout that we have a stratification by affine spaces, but occasionally we will want to weaken this requirement.  Let $X$ and $\scS$ be as above, but suppose that we also have another, coarser stratification $\scT$ of $X$.  To distinguish between the two stratifications, the stratum corresponding to $t \in \scT$ will be denoted with a superscript: $X^t$.  Let $i^t: X^t \hookrightarrow X$ be the inclusion map.

By assumption, each $X^t$ is a locally closed smooth subvariety that is a (finite) union of strata $X_s$ for $s \in \scS$.  We further assume that all the $X^t$ are connected and simply connected, and we impose a version of condition \textbf{(A1)}: 
\begin{enumerate}
\item[\textbf{(A1)}$_\scT$] For each $t \in \scT$, there is an indecomposable parity complex $\cE^t \in \Db_\scT(X)$ that is supported on $\overline{X^t}$ and satisfies $i^{t*}\cE^t \cong \ubk_{X_t} \{\dim X_t\}$.
\end{enumerate}
Note that each $\scT$-stratum $X^t$ must contain a unique dense $\scS$-stratum $X_s$, so the parity complex $\cE^t$ above must coincide with the parity complex $\cE_s$.  Thus, since \textbf{(A1)} is already assumed to hold, the new condition \textbf{(A1)}$_\scT$ can be rephrased as follows: for any $s \in \scS$ such that $X_s$ is dense in some $\scT$-stratum, the parity sheaf $\cE_s$ is constructible with respect to $\scT$.

In particular, the additive category $\Parity_\scT(X)$ of parity complexes constructible with respect to $\scT$ is a full subcategory of $\Parity_\scS(X)$.  We define the category
\[
\Dmix_\scT(X) := \Kb\Parity_\scT(X),
\]
and identify it with a full subcategory of $\Dmix_\scS(X)$.  Great care must be taken in working with $\Dmix_\scT(X)$, as the results of~\cite{arc:f2, arc:f3} do not automatically apply.  One basic fact we will need is the following.

\begin{prop}\label{prop:recollement-coarse}
Let $X$, $\scS$, and $\scT$ be as above.  Let $j: U \hookrightarrow X$ be an open inclusion of $\scT$-strata, and let $i: Z \hookrightarrow X$ be the complementary closed inclusion.
\begin{enumerate}
\item If $\cF \in \Dmix_\scT(U)$, then $j_!\cF$ and $j_*\cF$ lie in $\Dmix_\scT(X)$.
\item If $\cF \in \Dmix_\scT(X)$, then $i^*\cF$ and $i^!\cF$ lie in $\Dmix_\scT(Z)$.
\end{enumerate}
\end{prop}
Note that the analogous statements for $\Db_\scT(X)$ follow from \textbf{(A1)}$_\scT$.

\begin{proof}
We will need to make use of sheaf functors in the nonmixed setting, and we will need to distinguish them from their mixed analogues.  Thus, for the body of this proof only, we adopt the convention of~\cite[\S2.4]{arc:f2} that functors in the mixed setting are decorated with parentheses: $i^{(*)}$, $i^{(!)}$, $j_{(*)}$, $j_{(!)}$.  An undecorated symbol such as $i^*$ denotes a functor $\Db_\scS(X) \to \Db_\scS(Z)$ or $\Db_\scT(X) \to \Db_\scT(Z)$. (However, $i_*$ and $j^*$ are always undecorated, as in~\cite[\S2.3]{arc:f2}.) 

We proceed by induction on the number of $\scT$-strata in $Z$.  Suppose first that $Z = X^t$ is a single stratum.  Recall that for any $\scT$-constructible parity sheaf $\cE^u$, the object $i^*\cE^u \in \Db_\scT(Z)$ is a parity complex.  Let $\cE^{u,+} \in \Kb\Parity_\scT(X)$ denote the complex
\[
\cdots \to 0 \to \cE^t \to i_*i^*\cE^t \to 0 \to \cdots,
\]
where the nontrivial terms are in degrees $0$ and $1$, and the morphism is given by adjunction.  Thus, in $\Kb\Parity_\scT(X)$, there is a distinguished triangle
\begin{equation}\label{eqn:recolle-single}
\cE^{t,+} \to \cE^t \to i_*i^*\cE^t \to,
\end{equation}
while in $\Db_\scT(X)$, there is a distinguished triangle
\begin{equation}\label{eqn:recolle-unmixed}
j_!j^*\cE^t \to \cE^t \to i_*i^*\cE^t \to.
\end{equation}
We claim that for any $s \in \scS$ with $X_s \subset X^t$, we have
\begin{equation}\label{eqn:recolle-test}
\Hom_{\Dmix_\scS(X)}(\cE^{t,+}, \cE_s\{m\}[n]) = 0\qquad\text{for all $m,n \in \Z$.}
\end{equation}
To prove this, it suffices to show that the natural map $\Hom(i_*i^*\cE^t, \cE_s\{m\}[n]) \to \Hom(\cE^t, \cE_s\{m\}[n])$ is always an isomorphism.  Both $\Hom$-groups clearly vanish unless $n = 0$.  When $n = 0$, these $\Hom$-groups can be computed inside $\Parity_\scT(X) \subset \Db_\scT(X)$ instead, and then the fact that the map is an isomorphism follows from~\eqref{eqn:recolle-unmixed} and the fact that $\Hom(j_!j^*\cE^t, \cE_s\{m\}) = 0$.

Since objects of the form $\cE_s\{m\}[n]$ generate $\Dmix_\scS(Z)$, we deduce from~\eqref{eqn:recolle-test} that $\Hom(\cE^{t,+}, i_*\cG) = 0$ for all $\cG \in \Dmix_\scS(Z)$.  On the other hand, we clearly have $i^*\cE^t \in \Dmix_\scS(Z)$.  It follows from general principles of recollement that the distinguished triangle~\eqref{eqn:recolle-single} must be canonically isomorphic to
\[
j_{(!)}j^*\cE^t \to \cE^t \to i_* i^{(*)}\cE^t \to.
\]
In particular,~\eqref{eqn:recolle-single} shows us that $j_{(!)}j^*\cE^t$ lies in $\Dmix_\scT(X)$, and that $i^{(*)}\cE^t$ lies in $\Dmix_\scT(Z)$. Now, objects of the form $j^*\cE^t\la n\ra$ generate $\Dmix_\scT(U)$, and those of the form $\cE^t\la n\ra$ generate $\Dmix_\scT(X)$.  We conclude that $j_{(!)}$ takes $\Dmix_\scT(U)$ to $\Dmix_\scT(X)$, and that $i^{(*)}$ takes $\Dmix_\scT(X)$ to $\Dmix_\scT(Z)$.  The results for $j_{(*)}$ and $i^{(!)}$ follow by Verdier duality.

Now suppose that $Z$ contains more than one $\scT$-stratum.  Choose a $\scT$-stratum $X^t$ that is open in $Z$.  Let $V = U \cup X^t$ and $Y = Z \smallsetminus X^t$, and let $k: Y \to X$ be the inclusion map.  Given $\cF \in \Dmix_\scT(X)$,  we can form the distinguished triangle
\[
i^t_{(!)}(i^t)^{(*)}\cF \to i_*i^{(*)}\cF \to k_*k^{(*)}\cF \to.
\]
Since $Y$ contains fewer $\scT$-strata than $Z$, $k^{(*)}\cF$ lies in $\Dmix_\scT(Y)$ by induction.  Next, let $v: V \hookrightarrow X$, $a: X^t \hookrightarrow V$, and $b: X^t \hookrightarrow Z$ be the inclusion maps.  Note that $v$ and $b$ are open inclusions, and $a$ is a closed inclusion of a single $\scT$-stratum.  The cases of the result that we have already established show that $i^t_{(!)} (i^t)^{(*)} \cong i_* b_{(!)} a^{(*)} v^*$ takes $\cF$ to an object of $\Dmix_\scT(X)$.  We conclude that $i_*i^{(*)}\cF$ lies in $\Dmix_\scT(X)$, and hence that $i^{(*)}\cF$ lies in $\Dmix_\scT(Z)$.  Finally, from the distinguished triangle
\[
j_{(!)}j^*\cF \to \cF \to i_*i^{(*)}\cF \to,
\]
we see that $j_{(!)}j^*\cF$ lies in $\Dmix_\scT(X)$ as well. Since objects of the form $j^*\cF$ generate $\Dmix_\scT(U)$, $j_{(!)}$ takes all objects in $\Dmix_\scT(U)$ to $\Dmix_\scT(X)$.  Again, the results for $j_{(*)}$ and $i^{(!)}$ follow by Verdier duality.
\end{proof}

The following two statements are easy consequences of the previous lemma.  The proofs are left to the reader.

\begin{cor}
For an object $\cF \in \Dmix_\scS(X)$, the following are equivalent:
\begin{enumerate}
\item $\cF$ lies in $\Dmix_\scT(X)$.
\item For every $\scT$-stratum $i^t: X^t \hookrightarrow X$, $(i^t)^*\cF$ lies in $\Dmix_\scT(X^t)$.
\item For every $\scT$-stratum $i^t: X^t \hookrightarrow X$, $(i^t)^!\cF$ lies in $\Dmix_\scT(X^t)$.
\end{enumerate}
\end{cor}

\begin{cor}
The perverse t-structure on $\Dmix_\scS(X)$ induces a t-structure on $\Dmix_\scT(X)$.
\end{cor}

\subsection{Hom-groups in the equivariant derived category}
\label{ss:eqvt-der-cat}

We now return to the setting of a space stratified by affine spaces. 
Let $H$ be a connected algebraic group or pro-algebraic group acting on $X$, such that the strata of our stratification are $H$-stable. Assume that the $H$-equivariant cohomology of a point $\sH^\bullet_H(\pt)$ vanishes in odd degrees, and is free over $\bk$ in even degrees.

Let $\Db_H(X)$ be the $H$-equivariant derived category of $X$ with coefficients in $\bk$, in the sense of Bernstein--Lunts~\cite{bl}.  We also consider the full subcategory $\Db_{H,\scS}(X) \subset \Db_H(X)$ consisting of complexes that are constructible with respect to the stratification $\scS$.  The latter also has a ``mixed'' version $\Dmix_{H,\scS}(X) := \Kb\Parity_{H,\scS}(X)$, as explained in~\cite[\S3.5]{arc:f2}.  Let $\For: \Db_{H,\scS}(X) \to \Db_\scS(X)$ and $\For: \Dmix_{H,\scS}(X) \to \Dmix_\scS(X)$ denote the forgetful functors.

Our goal in this subsection is to understand $\gRHom$ on $\Dmix_{H,\scS}(X)$ in terms of modules over the equivariant cohomology ring $\sH^\bullet_H(\pt)$.

\begin{lem}\label{lem:free-eqvt-point}
Let $\cF \in \Db_H(\pt)$.  If $\Hom^\bullet(\ubk,\cF)$ is a free $\sH^\bullet_H(\pt)$-module, then there is a natural isomorphism
\[
\Hom^\bullet_{\Db(\pt)}(\ubk, \For(\cF)) \cong \Hom^\bullet_{\Db_H(\pt)}(\ubk, \cF) \mathop{\otimes}_{\sH^\bullet_H(\pt)} \bk.
\]
\end{lem}
\begin{proof}
A straightforward adaptation of~\cite[Lemma~6.1]{arid} shows that if $\Hom^\bullet(\ubk,\cF)$ is a free $\sH^\bullet_H(\pt)$-module, then $\cF$ must be a direct sum of copies of various $\ubk\{n\}$.  Thus, it suffices to prove the lemma in the special case where $\cF = \ubk$, and in this case, it is obvious.
\end{proof}

\begin{lem}\label{lem:parity-eqvt-hom}
Let $\cF, \cG \in \Db_H(X)$.  If $\cF$ is $*$-parity and $\cG$ is $!$-parity, then $\Hom^\bullet(\cF,\cG)$ is a free $\sH^\bullet_H(\pt)$-module, and there is a natural isomorphism
\[
\Hom^\bullet_{\Db_\scS(X)}(\For(\cF),\For(\cG)) \cong \Hom^\bullet_{\Db_H(X)}(\cF,\cG) \mathop{\otimes}_{\sH^\bullet_H(\pt)} \bk.
\]
\end{lem}
\begin{proof}
We proceed by induction on the number of strata in $X$. If $X$ consists of a single stratum, then $\cF$ and $\cG$ are both parity sheaves, i.e., direct sums of objects of the form $\ubk\{n\}$.  Thus, it suffices to prove the result when $\cF = \cG = \ubk$.  Since $X$ is isomorphic to an affine space, it is clear that $\Hom^\bullet_{\Db_H(X)}(\ubk,\ubk)$ is isomorphic to $\sH^\bullet_H(\pt)$, and that $\Hom^\bullet_{\Db_\scS(X)}(\ubk,\ubk) \cong \Hom^\bullet_{\Db_H(X)}(\ubk,\ubk) \otimes_{\sH^\bullet_H(\pt)} \bk$.

Now suppose that $X$ has more than one stratum.  Let us also assume without loss of generality that $\cF$ is $*$-even and that $\cG$ is $!$-even.  Let $i_s: X_s \hookrightarrow X$ be the inclusion of an open stratum, and let $h: X \smallsetminus X_s \hookrightarrow X$ be the inclusion of the closed complement.  By a standard recollement argument, we have a natural long exact sequence
\begin{multline*}
\cdots \to \Hom^k(h^*\cF, h^!\cG) \to \Hom^k(\cF,\cG) \to \Hom^k(i_s^*\cF,i_s^*\cG)
\\
 \to \Hom^{k+1}(h^*\cF, h^!\cG) \to \cdots.
\end{multline*}
Note that $h^*\cF$ and $i_s^*\cF$ are both $*$-even, and that $h^!\cG$ and $i_s^*\cG$ are both $!$-even.  By~\cite[Corollary~2.8]{jmw}, these $\Hom^k$-groups vanish when $k$ is odd, so this long exact sequence breaks up into a collection of short exact sequences.  Indeed, we obtain a short exact sequence
\[
0 \to \Hom^\bullet(h^*\cF, h^!\cG) \to \Hom^\bullet(\cF,\cG) \to \Hom^\bullet(i_s^*\cF, i_s^*\cG) \to 0
\]
of graded $\sH^\bullet_H(\pt)$-modules.  By induction, the first and last terms above are free $\sH^\bullet_H(\pt)$-modules, and hence the middle term is as well.

Now, let $a: X \to \pt$ be the constant map, and recall that $\Hom^\bullet(\cF,\cG) \cong \Hom^\bullet(\ubk, a_* \cRHom(\cF,\cG))$.  Since the functors $a_*$ and $\cRHom$ commute with $\For$, the last assertion of the lemma follows from Lemma~\ref{lem:free-eqvt-point}.
\end{proof}

\begin{cor}\label{cor:eqvt-hom-free}
For $\cF, \cG \in \Parity_{H,\scS}(X)$, the graded $\bk$-module $\gHom(\cF,\cG)$ naturally has the structure of a free graded $\sH^\bullet_H(\pt)$-module.
\end{cor}

Since $\gRHom$ is defined as a complex whose terms are $\gHom$-groups of parity sheaves, we can regard it as a functor
\[
\gRHom: \Dmix_{H,\scS}(X)^{\mathrm{op}} \times \Dmix_{H,\scS}(X) \to \Db(\sH^\bullet_H(\pt)\lgmod),
\]
where $\Db(\sH^\bullet_H(\pt)\lgmod)$ is the bounded derived category of the category of graded $\sH^\bullet_H(\pt)$-modules.  The following result is an immediate consequence of Lemma~\ref{lem:parity-eqvt-hom}.

\begin{prop}\label{prop:eqvt-hom-dg}
For any $\cF,\cG \in \Dmix_{H,\scS}(X)$, there is a natural isomorphism
\[
\gRHom_{\Dmix_\scS(X)}(\For(\cF),\For(\cG)) \cong \gRHom_{\Dmix_{H,\scS}(X)}(\cF,\cG) \mathop{\otimes}^{\scriptscriptstyle L}_{\sH^\bullet_H(\pt)} \bk.
\]
\end{prop}

\begin{cor}\label{cor:eqvt-hom-ff}
On the abelian category of mixed perverse sheaves, the forgetful functor $\For: \Perv^\mix_{H,\scS}(X) \to \Perv^\mix_\scS(X)$ is fully faithful.
\end{cor}
\begin{proof}
Let us call a graded $\sH^\bullet_H(\pt)$-module \emph{weakly free} if it is of the form $P \otimes_\bk \sH^\bullet_H(\pt)$, where $P$ is some graded $\bk$-module.  (If $\bk$ is a field, this is the same as a free $\sH^\bullet_H(\pt)$-module, but in general, we do not require $P$ to be free over $\bk$.)  Weakly free $\sH^\bullet_H(\pt)$-modules are acyclic for the functor $({-}) \otimes_{\sH^\bullet_H(\pt)} \bk$.  

Let $M$ be a graded $\sH^\bullet_H(\pt)$-module, and assume that it is concentrated in graded degrees${}\ge n_0$.  It is easy to see that $M$ admits a weakly free resolution $\cdots F^{-1} \to F^0 \twoheadrightarrow M$ with the property that $F^{-i}$ is concentrated in graded degrees${}\ge n_0+2i$.  As a consequence, we see that
\begin{equation}\label{eqn:graded-tor-bound}
\text{$\Tor^{\sH^\bullet_H(\pt)}_i(M, \bk)$ is concentrated in graded degrees${}\ge n_0+2i$.}
\end{equation}

Now, take $\cF, \cG \in \Perv^\mix_{H,\scS}(X)$.  Note that $\gHom^i(\cF,\cG)$ is concentrated in graded degrees${}\ge -i$: indeed, for $n < -i$, we have
\[
\gHom^i(\cF,\cG)_n = \Hom(\cF,\cG[i]\{n\}) = \Hom(\cF,\cG\la -n\ra[i+n]) = 0.
\]
There is a convergent spectral sequence of graded $\bk$-modules
\begin{equation}\label{eqn:tor-sseq}
H^p(\gHom^q(\cF,\cG) \mathop{\otimes}^{\scriptscriptstyle L}_{\sH^\bullet_H(\pt)} \bk) \Longrightarrow H^{p+q}(\gRHom_{\Dmix_{H,\scS}(X)}(\cF,\cG) \mathop{\otimes}^{\scriptscriptstyle L}_{\sH^\bullet_H(\pt)} \bk).
\end{equation}
(See, for instance,~\cite[Proposition~5.7.6]{weibel}.)  By Proposition~\ref{prop:eqvt-hom-dg}, the right-hand side can be identified with $\gHom^{p+q}(\For(\cF),\For(\cG))$. Picking out the graded components of degree $0$ on both sides of~\eqref{eqn:tor-sseq}, we obtain a convergent spectral sequence of (ungraded) $\bk$-modules
\[
E_2^{pq} = H^p(\gHom^q(\cF,\cG) \mathop{\otimes}^{\scriptscriptstyle L}_{\sH^\bullet_H(\pt)} \bk)_0 \Longrightarrow \Hom^{p+q}(\For(\cF), \For(\cG)).
\]
When $p > 0$, we obviously have $E_2^{pq} = 0$.  On the other hand,~\eqref{eqn:graded-tor-bound} tells us that $H^p(\gHom^q(\cF,\cG) \otimes^{\scriptscriptstyle L}_{\sH^\bullet_H(\pt)} \bk)$ is concentrated in degrees${}\ge -q-2p$, so $E_2^{pq} = 0$ if $-q-2p > 0$.  More generally, we conclude that for all $r \ge 2$, we have
\begin{equation}\label{eqn:Erpq}
E_r^{pq} = 0 \qquad\text{if $p > 0$ or $q < -2p$.}
\end{equation}
We claim that there are natural isomorphisms
\begin{equation}\label{eqn:Einfty}
E_\infty^{p,-p} \cong
\begin{cases}
E_2^{00} & \text{if $p = 0$,} \\
0 & \text{if $p \ne 0$.}
\end{cases}
\end{equation}
For $p \ne 0$, this follows from~\eqref{eqn:Erpq}.  For $p = 0$, we must show that the differentials
\[
d_r: E_r^{-r,r-1} \to E_r^{00}, \qquad
d_r: E_r^{00} \to E_r^{r,-r+1}
\]
vanish for all $r\ge 2$. But this follows from~\eqref{eqn:Erpq} as well. Next,~\eqref{eqn:Einfty} implies that we have a natural isomorphism
\[
(\gHom(\cF,\cG) \otimes_{\sH^\bullet_H(\pt)} \bk)_0 \cong \Hom_{\Dmix_\scS(X)}(\For(\cF),\For(\cG)).
\]
Let $\sH^{>0}_H(\pt) \subset \sH^\bullet_H(\pt)$ be the kernel of the obvious map $\sH^\bullet_H(\pt) \to \bk$.  Then
\[
\gHom(\cF,\cG) \otimes_{\sH^\bullet_H(\pt)} \bk \cong \gHom(\cF,\cG)/\sH^{>0}_H(\pt)\gHom(\cF,\cG).
\]
Since $\gHom(\cF,\cG)$ is concentrated in degrees${}\ge 0$ and $\sH^{>0}_H(\pt)\gHom(\cF,\cG)$ in degrees${}>0$, we see that $(\gHom(\cF,\cG) \otimes_{\sH^\bullet_H(\pt)} \bk)_0 \cong \gHom(\cF,\cG)_0 \cong \Hom(\cF,\cG)$.  We thus obtain the desired isomorphism $\Hom_{\Dmix_{H,\scS}(X)}(\cF,\cG) \cong \Hom_{\Dmix_\scS(X)}(\For(\cF),\For(\cG))$.
\end{proof}

\begin{rmk}
It is possible to carry out a common generalization of~\S\ref{ss:non-affine} and the present section.  Suppose that $\scT$ is a stratification of $X$, not necessarily by affine spaces, and which is refined by $\scS$.  Suppose furthermore that $H$ acts on $X$ and preserves the $\scT$-strata, but not necessarily the $\scS$-strata.  Then one can study the category $\Dmix_{H,\scT}(X) := \Kb\Parity_{H,\scT}(X)$.
\end{rmk}

\subsection{Applications to Kac--Moody groups}
\label{ss:kac-moody}

\newcommand{\scB}{\mathscr{B}}

We conclude with two results about the flag variety of a Kac--Moody group $G$.  We follow the notation of~\cite[\S4.1]{arc:f2}.  Specifically, let $B \subset G$ be the standard Borel subgroup, $W$ the Weyl group, and $\scB = G/B$ the flag variety.  Recall that the equivariant derived category $\Dmix_B(\scB)$ is equipped with a convolution product $\star: \Dmix_B(\scB) \times \Dmix_B(\scB) \to \Dmix_B(\scB)$.

\begin{prop}\label{prop:adv-conv}
Let $w \in W$.
\begin{enumerate}
\item The functors\label{it:adv-right-exact}
\[
({-}) \star \nmix_w,\ \nmix_w \star ({-}): \Dmix_B(\scB) \to \Dmix_B(\scB)
\]
are right adverse-exact.
\item The functors\label{it:adv-left-exact}
\[
({-}) \star \dmix_w,\ \dmix_w \star ({-}): \Dmix_B(\scB) \to \Dmix_B(\scB)
\]
are left adverse-exact.
\end{enumerate}
In particular, for any $w, y \in W$, the objects $\nmix_y \star \dmix_w$ and $\dmix_w \star \nmix_y$, when regarded as objects of the \emph{nonequivariant} derived category $\Dmix_{(B)}(\scB)$, are adverse.
\end{prop}
\begin{proof}[Proof Sketch]
The proof is essentially identical to that of~\cite[Proposition~4.6]{arc:f2}.  A brief outline for $({-}) \star \nmix_w$ is as follows.  Thanks to associativity of the convolution product, it suffices to prove the right adverse-exactness of $({-}) \star \nmix_s$ when $s$ is a simple reflection.  The proof of~\cite[Proposition~4.6]{arc:f2} exhibits, for any $y \in W$, a distinguished triangle whose middle term is $\dmix_y \star \nmix_s$, and whose first and last terms obviously lie in $\p\Dmix_B(\scB)^{\le 0}$.  Using Lemma~\ref{lem:adv-smooth}, it is easy to see that the first and last terms of that triangle also lie in $\ad\Dmix_B(\scB)^{\le 0}$.  It follows that $({-}) \star \nmix_s$ is right adverse-exact.
\end{proof}

Recall that there is a natural action of $W$ on the ring $\sH^\bullet_B(\pt)$.  
A linear map $f: M_1 \to M_2$ of $\sH^\bullet_B(\pt)$-modules is called a \emph{$w$-twisted homomorphism} if, for all $p \in \sH^\bullet_B(\pt)$ and $m \in M_1$, we have $f(pm) = (w\cdot p)f(m)$.

\begin{prop}\label{prop:conv-twist}
Let $w \in W$, and let $\cF, \cG \in \Dmix_B(\scB)$.  The natural maps
\[
\gHom(\cF, \cG) \to \gHom(\dmix_w \star \cF, \dmix_w \star \cG), \qquad
\gHom(\cF, \cG) \to \gHom(\nmix_w \star \cF, \nmix_w \star \cG)
\]
are $w$-twisted $\sH^\bullet_B(\pt)$-homomorphisms.
\end{prop}
\begin{proof}
Let $\cE_1,\cE_2 \in \Parity_B(\scB)$. Via the equivalence $\Db_B(\scB) \cong \Db_{B \times B}(G)$, we can equip $\gHom(\cE_1,\cE_2)$ with the structure of a module over $\sH^\bullet_{B \times B}(\pt) \cong \sH^\bullet_B(\pt) \otimes \sH^\bullet_B(\pt)$.  (The general setting of~\S\ref{ss:eqvt-der-cat} gives us an action of only one copy of $\sH^\bullet_B(\pt)$, which we identify with the left-hand copy in $\sH^\bullet_B(\pt) \otimes \sH^\bullet_B(\pt)$.)

Let $\cE_3, \cE_4 \in \Parity_B(\scB)$ be two additional parity sheaves. Convolution induces a homomorphism of $(\sH^\bullet_B(\pt) \otimes \sH^\bullet_B(\pt))$-modules
\begin{equation}\label{eqn:conv-bilinear}
\gHom(\cE_1,\cE_2) \otimes_{\sH^\bullet_B(\pt)} \gHom(\cE_3,\cE_4) \to \gHom(\cE_1 \star \cE_3, \cE_2 \star \cE_4).
\end{equation}
Here, the ring $\sH^\bullet_B(\pt)$ under the tensor product symbol acts as the right-hand copy on $\gHom(\cE_1,\cE_2)$, and as the left-hand copy on $\gHom(\cE_3,\cE_4)$.  Since the latter is always a free $\sH^\bullet_B(\pt)$-module, maps like~\eqref{eqn:conv-bilinear} induce corresponding maps at the derived level.  That is, given $\cF_1, \cF_2, \cF_3, \cF_4 \in \Dmix_B(\scB)$, we obtain a natural morphism
\begin{equation}\label{eqn:conv-bilinear-der}
\gRHom(\cF_1, \cF_2) \mathop{\otimes}^{\scriptscriptstyle L}_{\sH^\bullet_B(\pt)} \gRHom(\cF_3, \cF_4) \to
\gRHom(\cF_1 \star \cF_3, \cF_2 \star \cF_4).
\end{equation}

Let us now study this map in the special case where $\cF_1 = \cF_2 = \dmix_w$.  In this case, by adjunction, we have
\[
\gRHom(\dmix_w,\dmix_w) \cong \sH^\bullet_B(\scB_w) \cong \sH^\bullet_{B \times B}(B \dot w B/B),
\]
where $\dot w$ is a representative in $G$ of $w \in W$.  It is well known that $\sH^\bullet_{B \times B}(B \dot w B/B)$ is a rank-$1$ free module for both the left and right copies of $\sH^\bullet_B(\pt)$, and that the action of the right copy coincides with the $w$-twist of the action of the left copy.

In particular, because this module is free for the right copy of $\sH^\bullet_B(\pt)$, we can apply $H^0$ to~\eqref{eqn:conv-bilinear-der} and obtain a homomorphism of $(\sH^\bullet_B(\pt) \otimes \sH^\bullet_B(\pt))$-modules
\[
\sH^\bullet_B(\scB_w) \otimes_{\sH^\bullet_B(\pt)} \gHom(\cF_3,\cF_4) \to \gHom(\dmix_w \star \cF_3, \dmix_w \star \cF_4).
\]
From this, we deduce that $\dmix_w \star ({-}): \gHom(\cF_3,\cF_4) \to \gHom(\dmix_w \star \cF_3, \dmix_w \star \cF_4)$ is a $w$-twisted homomorphism, as desired.  The proof for $\nmix_w$ is similar.
\end{proof}



\begin{thebibliography}{JMW1}

\bibitem[A]{a}
P.~Achar, {\em Perverse coherent sheaves on the nilpotent cone in good
  characteristic}, Recent developments in Lie algebras, groups and
  representation theory, Proc. Sympos. Pure Math., vol.~86, Amer. Math. Soc.,
  2012, pp.~1--23.

\bibitem[ARc1]{arc:kd}
P.~Achar and S.~Riche, {\em Koszul duality and semisimplicity of Frobenius},
  Ann. Inst. Fourier {\bf 63} (2013), 1511--1612.

\bibitem[ARc2]{arc:f2} %
P.~Achar and S.~Riche, {\em Modular perverse sheaves on flag varieties II:
  Koszul duality and formality}, Duke Math. J. {\bf 165} (2016), 161--215.

\bibitem[ARc3]{arc:f3}
P.~Achar and S.~Riche, {\em Modular perverse sheaves on flag varieties III:
  positivity conditions}, arXiv:1408.4189.

\bibitem[AR]{arid}
P.~Achar and L.~Rider, {\em Parity sheaves on the affine Grassmannian and the
  Mirkovi\'{c}--Vilonen conjecture}, Acta Math., to appear, arXiv:1305.1684.

\bibitem[AB]{ab}
S.~Arkhipov and R.~Bezrukavnikov, {\em Perverse sheaves on affine flags and
  Langlands dual group}, Israel J. Math. {\bf 170} (2009), 135--183.

\bibitem[ABG]{abg}
S.~Arkhipov, R.~Bezrukavnikov, and V.~Ginzburg, {\em Quantum groups, the loop
  Grassmannian, and the Springer resolution}, J. Amer. Math. Soc. {\bf 17}
  (2004), 595--678.

\bibitem[BBD]{bbd}
A.~Be{\u\i}linson, J.~Bernstein, and P.~Deligne, {\em Faisceaux pervers},
  Analyse et topologie sur les espaces singuliers, I (Luminy, 1981),
  Ast\'erisque, vol. 100, Soc. Math. France, Paris, 1982, pp.~5--171.

\bibitem[BBM]{bbm}
A.~Be{\u\i}linson, R.~Bezrukavnikov, and I.~Mirkovi{\'c}, {\em Tilting
  exercises}, Moscow Math. J. {\bf 4} (2004), 547--557.

\bibitem[BGS]{bgs}
A.~Be{\u\i}linson, V.~Ginzburg, and W.~Soergel, {\em Koszul duality patterns in
  representation theory}, J. Amer. Math. Soc. {\bf 9} (1996), 473--527.

\bibitem[BL]{bl}
J.~Bernstein and V.~Lunts, {\em Equivariant sheaves and functors}, Lecture
  Notes in Mathematics, vol. 1578, Springer-Verlag, Berlin, 1994.

\bibitem[B1]{bez:qes}
R.~Bezrukavnikov, {\em Quasi-exceptional sets and equivariant coherent sheaves
  on the nilpotent cone}, Represent. Theory {\bf 7} (2003), 1--18.

\bibitem[B2]{bez:ctm}
R.~Bezrukavnikov, {\em Cohomology of tilting modules over quantum groups and
  {$t$}-structures on derived categories of coherent sheaves}, Invent. Math.
  {\bf 166} (2006), 327--357.
  
\bibitem[B3]{bez:psaf}
R.~Bezrukavnikov, {\em Perverse sheaves on affine flags and nilpotent cone of
  the Langlands dual group}, Israel J. Math. {\bf 170} (2009), 185--206.

\bibitem[BF]{bf}
R.~Bezrukavnikov and M.~Finkelberg, {\em Equivariant Satake category and
  Kostant--Whittaker reduction}, Mosc. Math. J. {\bf 8} (2008), 39--72.

\bibitem[BM]{bm}
R.~Bezrukavnikov and I.~Mirkovi\'{c}, {\em Representations of semi-simple Lie
  algebras in prime characteristic and noncommutative Springer resolution},
  arXiv:1101.2562.

\bibitem[BoKa]{bk:rfsfr}
A.~Bondal and M.~Kapranov, {\em Representable functors, Serre functors, and
  reconstructions}, Izv. Akad. Nauk SSSR Ser. Mat. {\bf 53} (1989), 1183--1205.

\bibitem[BrKu]{bk}
M.~Brion and S.~Kumar, {\em Frobenius splitting methods in geometry and
  representation theory}, Progr. Math., vol. 231, Birkh\"auser Boston, Boston,
  MA, 2005.

\bibitem[D]{deligne}
P.~Deligne, {\em La conjecture de {W}eil. {II}}, Inst. Hautes \'Etudes Sci.
  Publ. Math. (1980), no.~52, 137--252.

\bibitem[GLS]{gls}
C.~Geiss, B.~Leclerc, and J.~Schr\"{o}er, {\em Preprojective algebras and
  cluster algebras}, Trends in representation theory of algebras and related
  topics, EMS Ser. Congr. Rep., Eur. Math. Soc., Z\"urich, 2008, pp.~253--283.

\bibitem[FM]{fm}
M.~Finkelberg and I.~Mirkovi\'{c}, {\em Semi-infinite flags I. Case of global
  curve $\mathbf{P}^1$}, Differential topology, infinite-dimensional {L}ie
  algebras, and applications, Amer. Math. Soc. Transl. Ser. 2, vol. 194, Amer.
  Math. Soc., 1999, pp.~81--112.

\bibitem[Ja]{jantzen}
J.~C. Jantzen, {\em Representations of algebraic groups}, 2nd ed., Mathematical
  Surveys and Monographs, no. 107, Amer. Math. Soc., Providence, RI, 2003.

\bibitem[JMW1]{jmw}
D.~Juteau, C.~Mautner, and G.~Williamson, {\em Parity sheaves}, J. Amer. Math.
  Soc., to appear, arXiv:0906.2994.

\bibitem[JMW2]{jmw2}
D.~Juteau, C.~Mautner, and G.~Williamson, {\em Parity sheaves and tilting
  modules}, arXiv:1403.1647.

\bibitem[K]{kumar}
S.~Kumar, {\em Kac--Moody groups, their flag varieties and representation
  theory}, Progress in Mathematics, vol. 204, Birkh\"auser Boston Inc., 2002.

\bibitem[KLT]{klt}
S.~Kumar, N.~Lauritzen, and J.~F. Thomsen, {\em Frobenius splitting of
  cotangent bundles of flag varieties}, Invent. Math. {\bf 136} (1999),
  603--621.

\bibitem[LG]{lg}
V.~Lakshimbai and N.~Gonciulea, {\em Flag varieties}, Travaux en Cours, no.~63,
  Hermann, Paris, 2001.
  
\bibitem[MR]{mr}
C.~Mautner and S.~Riche, {\em Exotic tilting sheaves, parity sheaves on affine
  Grassmannians, and the Mirkovi{\'c}--Vilonen conjecture}, arXiv:1501.07369.

\bibitem[Mi]{minnthuaye}
M.~Minn-Thu-Aye, {\em Multiplicity formulas for perverse coherent sheaves on
  the nilpotent cone}, Ph.D. thesis, Louisiana State University, 2013.  Available at \url{http://etd.lsu.edu/docs/available/etd-07082013-113917/}.

\bibitem[Mu]{mumford}
D.~Mumford, {\em Stability of projective varieties}, Ensiegnement Math.~(2)
  {\bf 23} (1977), 39--110.

\bibitem[W]{weibel}
C.~A. Weibel, {\em An introduction to homological algebra}, Cambridge Studies
  in Advanced Mathematics, vol.~38, Cambridge University Press, Cambridge,
  1994.
  
\end{thebibliography}
\end{document}